\documentclass[10pt]{amsart} 

\usepackage{fullpage}
\usepackage[L7x,T1]{fontenc}
\usepackage[utf8]{inputenc}
\usepackage[lithuanian,english]{babel}
\usepackage[pdfauthor={Lorenzo Mantovani}, %
                pdftitle={Localizations and completions in motivic homotopy theory}%
            dvips,colorlinks=true]{hyperref}
\hypersetup{
    bookmarksnumbered=true,
    linkcolor=red,
    citecolor=blue,
}
\usepackage{amsmath,amsthm,stmaryrd}
\usepackage{mathrsfs,amssymb,amsfonts}
\usepackage{mathcomp}
\usepackage[all,cmtip]{xy}
\usepackage{enumitem}
\usepackage{color}
\usepackage{verbatim}
\usepackage{cite}
\usepackage[switch, displaymath, mathlines,right,left,pagewise]{lineno}

\newcommand{\Mod}[1]{\mathbf{Mod}_{#1}}
\newcommand{\spt}{\mathbf{Spt}}
\newcommand{\Ho}{\mathrm{Ho}\,}
\newcommand{\sm}[1]{\mathbf{Sm}_{#1}}

\newcommand{\bb}{\mathbb}
\renewcommand{\bf}{\mathbf}
\newcommand{\cal}{\mathcal}
\newcommand{\fk}{\mathfrak}

\renewcommand{\rm}{\mathrm}

\DeclareMathOperator\Hom{Hom}

\DeclareMathOperator\spec{Spec}

\DeclareMathOperator\im{Im}
\newcommand{\coker}{\text{coker }}
\newcommand{\ra}{\longrightarrow}
\newcommand{\bc}[1]{\big <#1 \big >}

\DeclareMathOperator\hocolim{hocolim}
\DeclareMathOperator\holim{holim}
\DeclareMathOperator\hofib{hofib}
\DeclareMathOperator\hocofib{hocofib}
\DeclareMathOperator\colim{colim}

\newcommand{\longlongrightarrow}{}
\DeclareRobustCommand{\longlongrightarrow}{\relbar\joinrel \relbar\joinrel\relbar\joinrel\rightarrow}
\newcommand{\rra}{\longlongrightarrow}

\DeclareRobustCommand{\longlonglongrightarrow}{\relbar\joinrel \relbar\joinrel \relbar\joinrel\relbar\joinrel\rightarrow}
\newcommand{\rrra}{\longlonglongrightarrow}

\newcommand{\medslant}[2]{{\raisebox{.15em}{$#1$}\left/\raisebox{-.15em}{$#2$}\right.}}

\theoremstyle{plain}
\newtheorem{thm}[subsubsection]{Theorem}
\newtheorem{prop}[subsubsection]{Proposition}
\newtheorem{lemma}[subsubsection]{Lemma}
\newtheorem{cor}[subsubsection]{Corollary}

\theoremstyle{definition}
\newtheorem{defin}[subsubsection]{Definition}
\newtheorem{notat}[subsubsection]{Notation}
\newtheorem{const}[subsubsection]{Construction}
\theoremstyle{remark}
\newtheorem{rmk}[subsubsection]{Remark}
\newtheorem{ex}[subsubsection]{Example}

\DeclareRobustCommand{\gobblefive}[5]{}

\title{Localizations and completions in Motivic homotopy theory}
\author{Lorenzo Mantovani}
\begin{document}

\begin{abstract}
  Let $K$ be a perfect field and let $E$ be a homotopy commutative ring spectrum in the Morel-Voevodsky stable motivic homotopy category $\mathcal{SH}(K)$. In this work we investigate the relation between the $E$-homology localization and $E$-nilpotent completion of a spectrum X. Under reasonable assumptions on $E$ and $X$ we show that these two operations coincide and can be expressed in terms of formal completions or localizations in the usual sense of commutative algebra. We deduce convergence criteria for the $E$-based motivic Adams-Novikov spectral sequence.
\end{abstract}

\maketitle

\setcounter{tocdepth}{1}
\tableofcontents

\section{Introduction} 
  \label{sec:introduction}
  \subsubsection*{Motivation}
  One of the most long-standing open problems in algebraic topology is that of describing the stable homotopy groups of the sphere spectrum $\pi_k(\bb S)=\varinjlim_n \pi_{k+n} \bb S^n$. For small values of $k$ it is possible to run ad-hoc geometric arguments, but for higher values of $k$ the difficulties increase very quickly. Adams observed that one could try to attack the problem by studying the $p$-primary torsion one prime at the time, by using the \emph{mod $p$ Adams spectral sequence}:
  \[\mathrm{Ext}_{A_\ast}^{\ast,\ast}(\bb Z/p,\bb Z/p) \Rightarrow \pi_\ast(\bb S)\otimes \bb Z_p.\]
  The spectral sequence converges completely to the $p$-adic completion of $\pi_\ast(\bb S)$, and takes as an input the singular homology of the one point space with its natural structure of co-algebra over the dual Steenrod algebra $A_\ast$. It is worth mentioning that thanks to the work of Milnor and many others we possess a very explicit presentation of $A_\ast$ and of the $E_2$ term of the above spectral sequence.

  Soon after it became clear that this was only an example of a much wider set of tools. Indeed, for any spectrum $X$ and a well behaved generalized ring homology theory $E$ one could produce a similar spectral sequence
  \[\mathrm{Ext}_{\pi_\ast(E\wedge E)}^{\ast,\ast}(\pi_\ast(E),\pi_\ast(E\wedge X)) \Rightarrow \pi_\ast(X^{\wedge}_E),\]
  called the \emph{$E$-based Adams-Novikov spectral sequence}.
  Even when $X=\bb S$ is fixed, different choices of $E$ give very different spectral sequences.
  One of the reasons for considering these spectral sequences is that, at the cost of having a less explicit description of the $E_2$ term and of the target, one would obtain better structural properties of the pages. For instance, when $E=KU$ is non-connective complex topological K-theory and $X=\bb S$ the spectral sequence collapses at page $E_3$ as suggested by Adams in \cite{MR0339178}.

  For a general $E$ it is hard to say what $ X^{\wedge}_E$ is. This spectrum is called the \emph{$E$-nilpotent completion of $X$}, and by design the spectral sequence converges conditionally to its homotopy groups, but we lack of a concrete description of this gadget. However, thanks to the seminal work of Adams \cite{MR0402720} and Bousfield \cite{MR543337,MR551009}, we know that in many good cases there is an equivalence  between the target $ X^{\wedge}_E$ of the spectral sequence and the $E$-homology localization of $X$. We wish to briefly explain this point. Given a spectrum $E$, we can formally invert all the maps in the stable homotopy category $\cal{SH}$ that induce isomorphisms in $E$-homology. We thus have a calculus of fractions and a localization functor $\cal{SH}\ra \cal{SH}[\{E-\textrm{homology isomorphisms}\}^{-1}]$. The localized category can then be embedded fully-faithfully into $\cal{SH}$ with essential image being the subcategory of $E$-local objects. The localization functor and the embedding of the localized category form an adjoint pair of functors and the unit of the adjunction $X\ra X_E$ is called the $E$-localization of $X$. In the work cited above the authors argue that, under some even more restrictive assumptions on $E$ and $X$, one can actually make $X_E$ really explicit: it takes the form either of a formal completion or of a localization in the standard commutative algebra meaning of these terms. The fact that the mod $p$ Adams spectral sequence converges to the $p$-adic completion of $\pi_\ast(\bb S)$ is an example of such results.

  \medskip

  Let us now turn to the realm of algebraic geometry and fix a perfect field $K$. In their seminal work \cite{MR1813224,MR2257774,morel:stabA1ht,MR2175638},  Morel and Voevodsky construct analogues of the unstable category $\cal H$ and of the stable category $\cal{SH}$ for smooth algebraic $K$-varieties. We denote these categories by $\cal H(K)$ and $\cal{SH}(K)$ respectively. One of the main novelties of these constructions is that they allow one to manipulate algebraic varieties and their cohomology theories in a way which is unarguably much closer to algebraic topology than any other previous approach. 
  For instance there is a sphere spectrum over $K$, which we simply denote by $\bb S_K$. It is defined as the infinite suspension of the space $S^0$, which is obtained by adding a disjoint base point to the base scheme $\spec K$. The stable homotopy groups of $\bb S_K$ are now bi-graded and, for keeping the best topological intuition, we can organize them into a family of graded groups $\pi_k(\bb S_K)_\ast:=\bigoplus_{n\in \bb Z}\pi_k(\bb S_K)_{n}$ (see \ref{sub:motivic_homotopy_cateogory} for a precise definition). It is thus natural to ask about some structural properties of the homotopy group of the motivic sphere spectrum. 
  
  In his pioneering work \cite{MR2240215} Morel proved that $\pi_k(\bb S_K)_\ast=0$ if $k<0$. In \cite{MR2175638} he gave a wonderful explicit presentation of $\pi_0(\bb S_K)_\ast$ in terms of the Milnor-Witt K-theory $K^{MW}_\ast(K)$: this is a graded ring, close to Milnor's K-theory, and whose weight zero part is canonically isomorphic to the Grothendieck-Witt ring of quadratic spaces over $K$. Moreover in \cite{MR1696188} Morel set up the basics for a motivic version of the mod $2$ Adams spectral sequence, where the homotopy groups $\pi_k(\bb S_K)_\ast$ are approximated by Voevodsky's motivic cohomology with mod $2$ coefficients. As an application he showed that, by only looking at the contributions to $\pi_0(\bb S_K)_0$, Voevodsky's computation of the motivic Steenrod algebra implies the Milnor's conjecture on the structure of quadratic forms over $K$. More recently several people have used this spectral sequence. Much work has been carried out by Dugger and Isaksen who have pioneered computations of $\pi_\ast(\bb S_K)_\ast$ in a certain range over the fields $K=\bb C, \bb R$, see for instance \cite{MR2629898}, \cite{MR3652813} and \cite{MR3590344}. Analogous computations have been carried out for $\bb S_K[1/\eta]$ in the work of Guillou-Isaksen \cite{MR3346515} \cite{MR3572357} and of Andrews-Miller \cite{MR3743072}. In a different but related direction, Ormsby and \O stv\ae r \cite{MR2811717,MR3073932} use the motivic Adams spectral sequence for determining the homotopy groups of the truncations of the algebraic Brown-Peterson spectra $BP_K$ over the field $K=\bb Q$ and its archimedean and non-archimedean completions.

  \medskip

  \subsubsection*{Main Content}

  Given a homotopy commutative ring spectrum $E$ in $\cal{SH}(K)$, we investigate the relation between $E$-localizations and $E$-nilpotent completions of motivic spectra. The notion of $E$-localization of a motivic spectrum is a close generalization of its topological counterpart and satisfies a similar formalism. $E$-localizations have been considered already a number of times in motivic homotopy theory. The existence of a $E$-local category was addressed in the appendix of \cite{MR2399164} and this has been used explicitly in \cite{MR2197578} and in \cite{Joachimi:MTI} in a chromatic perspective. Furthermore the notion is more or less explicitly used in many other works that need the formalism of $\ell$-adic completion of motivic spectra (see Section \ref{sec:moore_spectra} and in particular Proposition \ref{prop:localization_at_mod_x_moore_spectrum}). In algebraic topology, when $E$ is connective, a Theorem of Bousfield (see Theorem $3.1$ in \cite{MR551009}) explicitly describes the $E$-localization map $\lambda_E(X): X \ra X_E$ of a connective spectrum $X$. Our first result is a motivic analogue of such a theorem, albeit under some very strong assumptions on $E$. Roughly speaking we ask that the $0$-th homotopy module of our cohomology theory $E$ is of the form $\underline \pi_0 E=\underline \pi_0 (\bb S_K)/(f_1,\dots,f_r)$ (see \ref{sub:homotopy_t_structure} for a precise definition of homotopy modules), where $(f_1,\dots,f_r)$ is the unramified ideal generated by elements $f_i \in K^{MW}_{q_i}(K)$.

  \begin{thm}[\ref{thm:red_to_moore_spt} and \ref{prop:localization_at_mod_x_1-x_n_moore_spectrum}]
    \label{thm:1}
    Let $E$ be a homotopy commutative ring spectrum in $\cal {SH}(K)$. Assume that $E$ satisfies assumption \ref{sub:assumption_A1} in the special case of $J=\emptyset$. Then for every connective spectrum $X$ in $\cal{SH}(K)$ there is a canonical isomorphism $X_E\simeq X^{\wedge}_{f_1} {\cdots}^\wedge _{f_r}$ under which the localization map $\lambda_E(X): X\ra X_E$ and the completion map $\chi_{\underline f}(X): X \ra X^{\wedge}_{f_1} {\cdots}^\wedge _{f_r}$ coincide.
  \end{thm}
  
  The proof of such result is subdivided in two main parts. In the first part we find a spectrum $\bb S\underline \pi_0E$ such that $(-)_{\bb S\underline \pi_0E} \simeq (-)^{\wedge}_{f_1} {\cdots}^\wedge _{f_r}$. The spectrum $\bb S\underline \pi_0E$ is called, by analogy with Topology, a Moore spectrum associated to $\underline \pi_0 E$. In fact, such spectrum is not well defined, as it depends on choices that have to be done for performing its construction. However the dependence on these choices disappears when we consider the localization functor $(-)_{\bb S \underline \pi_0 E}$ defined by $\bb S \underline \pi_0 E$. In Proposition \ref{prop:localization_at_mod_x_1-x_n_moore_spectrum} we give the expected explicit form for the natural transformation $X \ra X_{\bb S\underline \pi_0 E}$ in terms of an adic completion. The second part of the proof of \ref{thm:1} is more involved and is morally subdivided in two steps. The first step consists in proving that we have a chain of inequalities of Bousfield classes
  \[  
    \bc{H\underline \pi_0 \bb S\wedge \bb S \underline \pi_0 E} \leq \bc{H\underline \pi_0 E} \leq \bc{E} \leq \bc{\bb S\underline \pi_0 E}.
  \]
  The notation and formalism of Bousfield classes is introduced in \ref{sub:bousfield_classes}. For the time being, the above chain  of inequalities can be understood as a chain of deeper and deeper Bousfield localizations where $\bc{\bb S\underline \pi_0 E}$ is the one with fewer local equivalences.
  Proving the existence of such a chain turns out to be not hard. The second step consists in showing that for a connective spectrum $X$, the $\bb S\underline \pi_0 E$-localization $X_{\bb S\underline \pi_0 E}$ is already $H\underline \pi_0 \bb S\wedge \bb S \underline \pi_0 E$-local. We find this part to be the most involved. 

  We now state an application of \ref{thm:1}. Recall that we have functors $\bb Z_{tr}(-)$ (resp. $\tilde {\bb Z}_{tr}(-)$) associating with every smooth algebraic variety $V$ over $K$, the presheaf with transfers (resp. with generalized transfers) represented by $X$. This functors lift to derived functors
  \[M:\cal{SH}(K) \ra \cal{DM}(K)\]
  \[\tilde M: \cal{SH}(K) \ra \widetilde{\cal{DM}}(K).\] 
  In \ref{ssub:ex_on_motives_of_spectra} and \ref{ex:Loc-at_MWMC} we deduce from \ref{thm:1} the following results.
  \begin{cor}
    \label{thm:cons}
    Let $K$ be a field of characteristic $0$, and let $X$ be a connective spectrum. 
    \begin{enumerate}
      \item Assume that $-1$ is a sum of squares in $K$. If $M(X)=0$ then $X[1/2]=0$.
      \item If $\tilde M(X)=0$ then $X=0$.
    \end{enumerate}
  \end{cor}
  Similar results have been obtained by Bachmann in \cite{Bacmann_Conservativity}; we discuss the relation with his work in \ref{rmk:our_cons_res_vs_Bachmann}.

  \bigskip
  The second main statement of this paper is about the comparison between the $E$-localization of a connective spectrum $X$ and the $E$-nilpotnent completion of $X$. 
  The $E$-nilpotent completion of a spectrum $X$ is the spectrum $X^\wedge_E$ obtained as homotopy inverse limit of the cosimplicial spectrum
  \[
    \xymatrix{
      E\wedge X  \ar@<.4ex>[r] \ar@<-.4ex>[r] & E\wedge E\wedge X \ar@<.8ex>[r] \ar@<-.8ex>[r] \ar[r] &  E\wedge E \wedge E \wedge X \cdots \\
    }
  \]
  Such a gadget is obtained from the ring structure of $E$ using the insertion of the unit $e:\bb S\ra E$ as cofaces and the multiplication $\mu : E\wedge E \ra E$ as codegeneracies. Also the $E$-nilpotent completion $X^\wedge _E$ is the natural generalization of its topological counterpart. In particular, given a motivic spectrum $X$, the $E$-nilpotent completion $X^\wedge_E$ is the natural target of the $E$-based motivic Adams-Novikov spectral sequence. Such a spectral sequence is a generalization of those used so far by several authors. When $E=H\bb Z/2$ is Voevodsky's motivic cohomology with $\bb Z/2$-coefficients, it recovers the homological version of the Adams spectral sequence used by Dugger-Isaksen and Guillou-Isaksen in the works cited above. For $E=BP^\wedge_\ell$, a $\ell$-completed version of the spectral sequence has been used in works of Ormsby-\O stv\ae r \cite{MR3255457}, to compute $\pi_\ast(\bb S_{K})_\ast$ in a certain range over the field $K=\bb Q$.

  Given the definition of $E$-nilpotent completion $X^\wedge_E$, it follows by formal nonsense that the natural map $\alpha_E(X): X \ra X^\wedge_E$ factors canonically through the $E$-localization. As a consequence we have a natural map $\beta_E(X): X_E \ra X^\wedge _E$. Our second result provides some control over the map $\beta_E(X)$.

  \begin{thm}[\ref{thm:convergence_mod_stuff}]
     \label{thm:2}
     Let $E$ be an homotopy commutative ring spectrum in $\cal {SH}(K)$ and assume that $E$ satisfies assumption \ref{sub:assumption_A1} in the special case of $J=\emptyset$. Then for every connective spectrum $X$ in $\cal{SH}(K)$ the map $\beta_E(X): X_E \ra X^{\wedge}_E$ is an isomorphism in $\cal{SH}(K).$
   \end{thm} 

  Though the present literature about $E$-based Motivic Adams-Novikov spectral sequences is quite vast, at the moment the main reference about the convergence of such spectral sequences is \cite{MR2811716}. It was the reading of this specific work that stimulated our interest in the topic. Our results partially overlap with those of \cite{MR2811716} since the motivic ring spectrum $E=H\bb Z/p$, representing motivic cohomology with $\mod p$ coefficients, satisfies our assumptions. Even in this very specific example, our results actually generalize those of \cite{MR2811716}: the only requirement we ask on the spectrum $X$, in order to describe $X^\wedge _{H\bb Z/p}$, is the connectivity. In other words, we have no assumptions on the existence of a cell presentation of $X$ of any specific form. Our approach to the problem is in fact very different in spirit from that of \cite{MR2811716} and we refer the reader to \ref{rmk:HKOcomparison} for a detailed comparison of our work with  \cite{MR2811716}.

  \medskip

  \subsubsection*{Organization of the paper}
  Section \ref{sec:a_reminder_on_motivic_categories} recalls some well known properties we need about $\cal {SH}(K)$ that can be found in the current literature, some basics about Bousfield's homology localizations and the formalism of Bousfield classes. In section \ref{sec:moore_spectra} we introduce a motivic analogue of Moore spectra and study the localizations they define. In section \ref{sec:localizations} we prove Theorem \ref{thm:1} which, for some very special spectra $E$, allows to reduce the study of $E$-localizations to the study of localizations at the Moore spectrum associated to $\underline \pi_0(E)$. Section \ref{sec:examples} is dedicated to some examples and to some immediate corolloraries of Theorem \ref{thm:1}.
  At this point we turn to the $E$-based Adams-Novikov spectral sequence. The construction is outlined in section \ref{sec:the_E_based_MANSS}, where we also revise the possible convergence properties of the spectral sequence. In section \ref{sec:nilpotent_resolutions} we address $E$-nilpotent resolutions and completions, which allow more flexible constructions of the $E$-based motivic Adams-Novikov spectral sequence. These tools provide enough technology to prove Theorem \ref{thm:2} up to some facts on pro-spectra, to whom we dedicate the Appendix \ref{sec:pro_spectra}.

  \subsubsection*{Aknowledgements} This is essentially the content of my Ph.D. thesis. I wish to heartily thank my advisor Marc N. Levine for the constant encouragement, his interest and for the many conversations about the content of this paper. A special thank goes to T. Bachmann for the interest and the conversations on the topic of this paper. Another special thank goes to J. I. Kylling for the interest and the rigorous spell checking. I wish to thank my colleagues and the whole ESAGA group at the university of Duisburg-Essen for providing a comfortable and productive working environment. This work was supported by the Humbolt Foundation and the DFG Schwerpunkt Programme 1786 \emph{Homotopy Theory and Algebraic Geometry}.


\section{A reminder on motivic categories} 
  \label{sec:a_reminder_on_motivic_categories}

  	In this section we review some well-known facts about the motivic stable homotopy category $\cal{SH}(S)$ of a base scheme $S$, and we take the occasion to introduce the notation that we will need later on. In \ref{sub:motivic_homotopy_cateogory} we recall some basic structural properties of the category $\cal{SH}(S)$ and briefly review its construction. In \ref{sub:homotopy_t_structure} we review Morel's homotopy $t$-structure, which is by far the most important tool we need. After this we introduce homology localizations, which are the main subject of study in this work. Given a spectrum $E \in \cal{SH}(S)$, we introduce $E$-homology localizations and co-localizations, we briefly address the existence of such constructions (both in the triangulated  and in the model category setting), and finally deduce some formal properties of these constructions. This will take place in \ref{sec:homology_localizations}.  We conclude with a review of the formalism of Bousfield classes in \ref{sub:bousfield_classes}. Bousfield classes are an abstract tool which is very useful to keep track of the mutual relations between the localization functors $(-)_E$, defined for the various $E \in \cal{SH}(S)$. Our choice of symbols and nomenclature follows closely that of \cite{MR543337} and \cite{MR551009}.

  \subsection{Motivic Stable Homotopy Cateogory} 
    \label{sub:motivic_homotopy_cateogory}
    
    A \emph{base scheme} is a Noetherian scheme of finite Krull dimension. Given a a base scheme $S$, we denote by $\cal{SH}(S)$ the Morel-Voevodsky stable homotopy category. For the uses of the present paper we will be mostly interested in the case where $S=\spec K$ is the spectrum of a perfect field $K$. In this case the category was constructed in \cite{morel:stabA1ht,MR2175638}. More generally, over any base scheme $S$, the category $\cal {SH}(S)$ was constructed by Jardine in \cite{jard:motsym}. Namely $\cal{SH}(S)$ is the homotopy category associated with the the \emph{stable motivic model structure} (see \cite[Section 4.2]{jard:motsym}) on the category $\spt^\Sigma_T(S)$ of symmetric $T$-spectra over $S$; here $T$ is the Thom space associated to the trivial bundle of rank one over $S$. The stable  motivic model structure is simplicial, stable, proper, cofibrantly generated and combinatorial, as proved in \cite[Theorem 4.15]{jard:motsym}. Furthermore the stable motivic model structure on $\spt_T^\Sigma(S)$ is compatible with the smash product of symmetric spectra (see \cite[Proposition 4.19]{jard:motsym}) and makes $\spt_T^\Sigma(S)$ into a closed symmetric monoidal model category. As a result the homotopy category $\cal{SH}(S)$ is a closed symmetric-monoidal triangulated category. We denote by $[X,Y]$ the abelian group of maps between $X$ and $Y$ in $\cal {SH}(S)$. The triangulated category $\cal {SH}(S)$ is actually compactly generated by the set 
      \[ \{\Sigma ^{p+q\alpha}\Sigma^\infty X_+ : \; X \in \sm K, \; p,q \in \bb Z \}, \]
      where $\Sigma^{p+q\alpha}$ is defined as $\Sigma_{S^1}^{p}\Sigma_{\bb G_m}^{q}$.

  \subsection{Homotopy t-structure} 
    \label{sub:homotopy_t_structure}
    
    Let $S$ be a base scheme. To every motivic spectrum $X$ we associate homotopy sheaves $\underline \pi_p(X)_q$ which are defined as the Nisnevich sheafification of the presheaves
    \[\pi_p(X)_q: U \mapsto [ \Sigma^{p-q\alpha} \Sigma^\infty U_+,X], \]
    defined on the category $\sm S$ of smooth $S$-schemes, and with values in abelian groups.
    We will denote by $\underline \pi_p(X):= \bigoplus_{q \in \bb Z} \underline\pi_p(X)_q$ and regard them as homotopy modules in the sense of Morel \cite[Deinition 5.2.4]{MR2175638}. A spectrum $X$ is called $k$-connective if for every $p\leq k$, $\underline \pi_p(X)=0$; $X$ is called connective if it is $k$-connective for some integer $k \in \bb Z$. 

    Under a different point of view, we can define $\cal {SH}(S)_{\geq n}$ as the localizing subcategory of $\cal{SH}(S)$ generated by the set 
    \[\{\Sigma ^{p+q\alpha}\Sigma ^\infty X_+\; : X \in \sm K, \; p \geq n,\; q\in \bb Z \}.\]
    The general formalism of accessible localizations, which in fact works for every base scheme $S$, implies that the inclusion $i_n: \cal {SH}(S)_{\geq n} \subseteq \cal {SH}(S)$ has a right adjoint $\tau_{\geq n}: \cal {SH}(S) \ra \cal {SH}(S)_{\geq n}$. We denote the composition $i_n\circ \tau_{\geq n}$ by the  symbol $P_n$. Similarly we denote by $\cal{SH}(S)_{\leq (n-1)}$ the right orthogonal of $\cal {SH}(S)_{n}$; i.e. $\cal {SH}(S)_{\leq n-1}$ is the full subcategory of those objects $X\in \cal {SH}(S)$ with the property that, for every $A \in \cal {SH}(S)_{\geq n}$, $[A,X]=0$. In particular the pair of subcategories $\big (\cal{SH}(S)_{\geq 0},\cal{SH}(S)_{\leq -1} \big )$ define an accessible $t$-structure on $\cal {SH}(S)$.

    We thus get a cofibre sequence
    \begin{equation}
    \label{eqn:postfibseq}
    P_n(X) \overset{\delta_n}{\ra} X \overset{\pi_{n-1}}{\ra} P^{n-1}(X)
    \end{equation}
    which is functorial in $X \in \cal {SH}(S)$, where $\delta_n$ is the co-unit of the adjunction $(i_n,\tau_{\geq n})$.
    
    If $S=\spec K$ and $K$ is a (perfect) field, then a spectrum $X$ is $n$-connective if and only if $X$
    lies in the subcategory $\cal{SH}(S)_{\geq n}$: this is proven for instance in Theorem 2.3 of \cite{hoyois:from_algebraic_cobordism_to_motivic_cohomology} and is a consequence of Morel's stable $\bb A^1$-connectivity Theorem (see \cite{MR2240215}). In such situation, the map $\delta_n$ (resp. $\pi_n$) induces an isomoprhism on the homotopy sheaves $\underline\pi_k$ for every $k\geq n$ (resp. on $\underline \pi_k$ for $k\leq n-1$). In addition $P_n(X)$ is $(n-1)$-connective, while the homotopy sheaves $ \underline \pi_k (P^{n-1}(X))$  are trivial for $k \geq n$. In this case the $t$-structure defined above is called \emph{homotopy $t$-structure}.

    Morel identified the heart of the homotopy $t$-structure with the category of homotopy modules. A homotopy module $\cal F$ is a collection $\{\cal F_n, \sigma ^{\cal F}_n\}_{n \in \bb Z}$, where $\cal F_n$ is a strictly $\bb A^1$-invariant Nisnevich sheaf on $\sm K$, and $\sigma_n: \cal F_n \overset{\simeq}{\ra} (\cal F_{n+1})_{-1}$ is an isomorphism of $\cal F_n$ with Voevodsky's contraction of $\cal F_{n+1}$. A morphism of homotopy modules $\phi: \cal F \ra \cal G$ is a collection of maps of abelian sheaves $\phi_n: \cal F_n \ra \cal G_n$ that is compatible with the bonding maps $\sigma_n^{\cal F}, \sigma_n^{\cal G}$. We adopt Morel's notation and denote the category of homotopy modules over the field $K$ by $\Pi_\ast(K)$. For every spectrum $X \in \cal{SH}(K)$ and every integer $p$, the collection $\{\underline \pi_p(X)_q\}_{q\in \bb Z}$ has a natural structure of homotopy module and the restriction of the functor $\underline \pi_0$ to the heart
    \[\underline \pi_0: \cal{SH}(K)_{\geq 0} \cap \cal{SH}(K)_{\leq 0}=:\cal {SH}(K)^{\heartsuit} \ra \Pi_\ast(K)\]
    is an equivalence \cite[Theorem 5.2.6]{MR2175638}. We denote by $H$ a quasi-inverse, in analogy with the Eilenberg-MacLane functor in Algebraic Topology. It should not be confused with Voevodsky's Eilenberg-MacLane functor, associating with a simplicial abelian group $A$ the motivic cohomology spectrum with coefficients in $A$. 
     
    The category $\cal {SH}(K)^{\heartsuit}$, and thus $\Pi_\ast(K)$, is a symmetric monoidal category with respect to the functor
    \[-\otimes -:  (X, Y) \mapsto (X\wedge Y)_{\tau_{\leq 0}}.\]
    The Postnikov tower $P_n(-)$ satisfies the axiomatic framework of \cite[Section 2]{MR2971612}, and thus from \cite[Theorem 5.1]{MR2971612} we deduce that $H\underline \pi_0\bb S$ is a $\bf E_\infty$-ring spectrum. The functor $H$ is fully faithful, exact, and identifies the heart of the homotopy $t$-structure with the category of discrete $H\underline \pi_0 \bb S$-modules, i.e. those $H\underline \pi_0 \bb S$-modules in $\cal{SH}(K)$ for which $\underline \pi_0$ is the only possibly non-trivial homotopy module.

    Recall that when $K$ is a perfect field of characteristic not $2$ we can construct a map (see \cite[Section 6.1]{MR2061856})
    \[ \sigma:K^{MW}_\ast(K) \ra [\bb S, \bb G_m^{\wedge \ast}] \]
    by defining it on the generators and checking that it passes to the quotient through the relations of $K^{MW}_\ast(K)$ (See \cite[Definition 3.1 pg. 49]{MR2934577} for precise formulas).
    We set $K^{MW}_{-1} \ni \eta \mapsto \sigma(\eta)$ where $\sigma(\eta): \bb S \ra \Sigma^{-\alpha}\bb G_m$ is the desuspension of the infinite $T$-suspension of the algebraic Hopf map $\bb A^2 \setminus 0 \ra \bb P^1$ which maps $(x,y)$ to $[x,y]$. Furthermore we set, for every element $u \in K^\times$, $\sigma(u): \bb S \ra \bb G_m$ to be the natural map induced on infinite $T$-suspensions. The compatibility of $\sigma$ with the relations defining $K^{MW}_\ast(K)$ is checked in \cite[Section 6.1]{MR2061856} and in \cite[Theorem 6.2.1]{MR2061856} it is proven that $\sigma$ is actually an isomorphism, and it extends uniquely to an isomorphism of homotopy modules $\sigma: \cal K^{MW}_\ast \overset{\simeq}{\ra} \underline \pi_0 \bb S$.
    
      More generally let $E$ be a $(-1)$-connective $\bf E_\infty$-ring spectrum. Then the definitions of connectivity above induce on $\Ho(\Mod E)$ a $t$-structure whose heart is equivalent to the category of $\underline \pi_0(E)$-modules in $\Pi_\ast(K)$. Similarly the heart of such a $t$-structure is equivalent to the category of discrete $H\underline \pi_0 E$-modules.


  \subsection{Homology Localizations} 
  \label{sec:homology_localizations}
    
    \begin{defin}
      \label{defin:E-local_stuff}
      Let $E$ a spectrum in $\cal{SH}(S)$. We say that a map in $\cal{SH}(S)$ of motivic spectra $f:X\ra Y$ is an $E$-homology equivalence (or, shortly, an $E$-equivalence) if the induced map $f\wedge \mathrm{id} :X\wedge E \ra Y\wedge E$ is an isomorphism in $\cal{SH}(S)$. We say that a spectrum $C$ is \emph{$E$-acyclic} if $E\wedge C$ is isomorphic to $0$ in $\cal {SH}(S)$. Finally we say that a spectrum $X\in \cal{SH}(S)$ is \emph{$E$-local} if for every $E$-acyclic spectrum $C$, we have that $[C,X]=0$.
    \end{defin}
    \begin{rmk}
      \label{rmk:prop_of_acyc_obj}
      One sees immediately that the class $\rm{Ac}(E)$ of $E$-acyclic objects is closed under arbitrary (small) homotopy colimits and retracts, and that $E$-acyclic objects have the 2-out-of-3 property in fibre sequences. We abuse the notation and write $\rm{Ac}(E)$ for the full triangulated subcategory of $\cal{SH}(S)$ whose objects are the $E$-acyclic spectra. Then $\rm{Ac}(E)$ is a thick localizing triangulated subcategory of $\cal{SH}(S)$ and the inclusion $\rm{Ac}(E)\subseteq \cal{SH}(S)$ respects homotopy colimits. Note that $\rm{Ac}(E)$ is also closed under smashing with an arbitrary spectrum, i.e. it is a thick ideal.
    \end{rmk}
    \begin{rmk}
      \label{rmk:prop_of_loc_obj}
      Similarly it is immediate to see that the class $\rm{Loc}(E)$ of $E$-local objects is closed under arbitrary (small) homotopy limits and retracts, and that $E$-local objects have the 2-out-of-3 property in fibre sequences. Again we abuse the notation and write $\rm{Loc}(E)$ for the full triangulated subcategory of $\cal{SH}(S)$ whose objects are the $E$-local spectra. Then $\rm{Loc}(E)$ is a thick triangulated subcategory of $\cal{SH}(S)$ and the inclusion $\rm{Loc}(E)\subseteq \cal{SH}(S)$ respects homotopy limits. An object $X$ is $E$-local if and only if for every $E$-equivalence $A\ra B$ the induced map $[B,X]\ra [A,X]$ is an isomorphism.
    \end{rmk}

    \begin{rmk}
      \label{rmk:E-hom_eq_vs_E-loc_eq}
      Note that we could define an, a priori, more general notion of equivalence: we call a map $A\ra B$ an \emph{$E$-local equivalence} if, for every $E$-local object $X$ the natural map $[B,X]\ra [A,X]$ is an isomorphism. Clearly all $E$-equivalences are $E$-local equivalences. The reverse holds if and only if the class $\rm{Ac}(E)$ of $E$-acyclic objects coincides with its double orthogonal 
      \[
      {}^{\perp} (\rm{Ac}(E)^{\perp}):=\{X\in \cal{SH}(S) \; s.t. \; \forall \; C\in \rm{Loc}(E), [X,C]=0\}.
      \]
      In view of the general theory of Bousfield localizations of triangulated categories (see for instance Chapter $9$ of \cite{MR1812507} or Sections $4$ and $5$ of \cite{MR2681709}) the equality of full subcategories ${}^{\perp} (\rm{Ac}(E)^{\perp})=\rm{Ac}(E)$ is implied by the existence of a Bousfield localization functor $(-)_E: \cal{SH}(S) \ra \cal{SH}(S)$ having $\rm{Ac}(E)$ as kernel.
     \end{rmk}
    \begin{defin}
    \label{defin:bousf_loc_func}
      Given a triangulated category $\cal T$ a Bousfield localization functor is an exact functor $L:\cal T \ra \cal T$ together with a natural transformation $\lambda: \rm{id} \ra L$ such that the natural transformations $L\circ \lambda, \lambda \circ L: L(-)\ra L^2(-)$ coincide and are isomorphisms.

      Given an exact functor $F: \cal T\ra \cal T$ of triangulated categories, we call \emph{Kernel} of $F$ the full subcategory $\ker F$ of $\cal T$ spanned by the objects $X\in \cal T$ such that $F(X)\simeq 0$.
    \end{defin}
    \begin{prop}[see $4.9.1$ of \cite{MR2681709}]
      \label{prop:eq_cond_for_ex_triang_loc}
      Let $S$ be a base scheme, $E$ be a spectrum and $\cal A:=\rm{Ac}(E)$ be the full subcategory of $\cal{SH}(S)$ spanned by $E$-acyclic objects. The the following are equivalent:
      \begin{enumerate}[label=L.\arabic*]
        \item \label{loc_func} there exists a localization functor $(-)_E:\cal{SH}(S)\ra \cal{SH}(S)$ having $\cal A$ as kernel;
        \item \label{acc_loc} the inclusion functor $\cal A \subset \cal{SH}(S) $ has a right adjoint;
        \item \label{ref_loc} the quotient functor $\cal{SH}(S) \ra \cal{SH}(S)/\cal A$ has a right adjoint;
        \item  the inclusion $\rm{Loc}(E) \subset \cal{SH}(S)$ has a left adjoint and $\cal A = {}^{\perp}(\cal A^{\perp})$;
        \item \label{funct_fibseq} for every spectrum $X \in \cal{SH}(S)$ there is an exact triangle ${}_E X \ra X \ra X_E$ with ${}_E X \in \cal A$ and $X_E$ in $\rm{Loc}(E)$.
      \end{enumerate}
    \end{prop}

    \begin{defin}
    \label{defin:E-localization}
      Assume that any of the equivalent conditions of \ref{prop:eq_cond_for_ex_triang_loc} is satisfied. Then the functor $(-)_E$ of \ref{loc_func} is called \emph{$E$-localization functor}, while the fucntor ${}_E(-)$ in \ref{funct_fibseq} is called \emph{$E$-acyclicization functor}. We denote by $\lambda_E$ the natural transformation $X \mapsto (\lambda_E(X): X\ra X_E)$ of \ref{funct_fibseq}, and we call it \emph{the $E$-localization map}. Similarly the natural transformation $\nu_E: X\mapsto (\nu_E(X): {}_E X \ra X)$ of \ref{funct_fibseq} is called $E$-acyclicization map.    
    \end{defin}

    \subsubsection{}
    The existence of a $E$-localization functor is not a completely formal statement. It would follow from the adjoint functor theorem if we knew in advance that the subcategory of $E$-acyclic objects is the localizing subcategory generated by a \emph{small set} of $E$-acyclic objects. This point, however, is the central difficulty that one encounters in the construction of the $E$-local stable motivic model structure $\spt_T^{\Sigma}(S)_E$ on motivic spectra; see \cite[Appendix]{MR2399164} for the definition. The fibrant replacement in $\spt_T^{\Sigma}(S)_E$ is indeed designed to produce, functorially in $X$, an $E$-equivalence $X \ra L_E(X)$ with an $E$-local target $L_E(X)$. Thus, if we assume the existence of $\spt_T^{\Sigma}(S)_E$, the Quillen adjunction
      \[ id: \spt_T^{\Sigma}(S) \rightleftarrows \spt_T^{\Sigma}(S)_E: id\]
    induces a derived adjunction
      \[ \bf L id: \cal{SH}(S) \rightleftarrows  \cal{SH}(S)_E: \bf R id, \]
    and the unit natural transformation $id \ra \bf R id \circ \bf L id$ satisfies the condition \eqref{loc_func} of \ref{prop:eq_cond_for_ex_triang_loc}. In particular, thanks to Theorem \ref{thm:existence_of_loc_mod_str}, we gain the full formalism described in \ref{prop:eq_cond_for_ex_triang_loc}.

    \begin{thm}[See Appendix of \cite{MR2399164}]
      \label{thm:existence_of_loc_mod_str}
      Let $E$ be a motivic symmetric $T$-spectrum over a base scheme $S$. Then the category $\spt_T^\Sigma(S)$ of motivic symmetric $T$-spectra over $S$ together with the classes of stable cofibrations, $E$-equivalences, and $E$-local fibrations satisfies the axioms of a left proper, combinatorial simplicial monoidal model category
    \end{thm}

  \subsection{Bousfield classes} 
    \label{sub:bousfield_classes}

      \subsubsection{}
      We introduce an equivalence relation on the class of isomorphism classes of spectra: we set that $E\sim_B F$ if, for every motivic spectrum $X$, we have that $E\wedge X=0$ if and only if $F\wedge X=0$. By theorem \ref{thm:existence_of_loc_mod_str} localization functors at $E$ and $F$ exist and they are (canonically) isomorphic exactly when $E\sim_BF$. We denote by $\cal A(S)$ the class of Bousfield classes in $\cal {SH}(S)$ and by $\bc E$ the element in $\cal A(S)$ represented by a spectrum $E$. On $\cal A(S)$ we introduce a partial ordering by setting $\bc E \leq \bc F$ if every $F$-acyclic spectrum is $E$-acyclic.
      
      \subsubsection{}
      Given a possibly infinite collection of Bousfield classes $\bc {E_i}_{i \in I}$ we have a \emph{join} operation which is defined as $\vee_{i \in I}\bc {E_i}:=\bc{\vee_{i \in I}E_i}$. We note that the join is always the minimal upper bound of its summands. 

      Similarly, given a finite collection of Bousfield classes $\bc {E_i}_{i \in I}$ we have a \emph{meet} operation which is defined as $\wedge_{i \in I}\bc {E_i}:=\bc{\wedge_{i \in I}E_i}$. We note that the meet operation is a lower bound for its factors, but in general doesn't need to be the maximal lower bound.

      \subsubsection{}
      The operations of join and meet are both commutative and associative. For every pair of spectra $E,F \in \cal{SH}(S)$       
      the \emph{first absorption law}
      \begin{equation}
        \label{eqn:vee_absorbes_a_smash}
        \bc{E}\vee \bc{E\wedge F}=\bc{E}
      \end{equation}
       is always verified. On the other hand we believe that the \emph{second absorption law}
       \begin{equation}
        \label{eqn:smash_absorbes_a_vee}
         \bc{E} \wedge \bc{E\vee F}=\bc{E}
       \end{equation}
        is not always satisfied. Indeed, if both the absorption laws \eqref{eqn:vee_absorbes_a_smash} and \eqref{eqn:smash_absorbes_a_vee} were verified for every spectrum $E$, then by combining them together we could deduce that both the \emph{first idempotence law}
        \begin{equation}
          \label{eqn:idemp_of_vee}
          \bc{E} \vee \bc{E} =\bc{E}
        \end{equation}
        and the \emph{second idempotence law}
        \begin{equation}
          \label{eqn:idemp_of_smash}
          \bc{E} \wedge \bc{E}=\bc{E}
        \end{equation}
       hold true. However, while \eqref{eqn:idemp_of_vee} does indeed hold true, we expect to have counterexamples for the second \eqref{eqn:idemp_of_smash}. More precisely, in the classical topological setting of $\cal{SH}$, Proposition 2.5 of \cite{MR543337} gives an example of failure of \eqref{eqn:idemp_of_smash}. We finally note that, since both \eqref{eqn:vee_absorbes_a_smash} and \eqref{eqn:idemp_of_vee} hold true in general, then \eqref{eqn:smash_absorbes_a_vee} holds if and only if \eqref{eqn:idemp_of_smash} holds.

       The meet operation distributes with respect to the join of possibly infinitely many classes:
       \begin{equation}
        \label{eqn:wedge_dist_over_vee}
         \bc{E} \wedge (\vee_i\bc{F_i})=\vee_i \bc{E\wedge F_i}.
       \end{equation}
       We note furthermore that if both the absorption laws \eqref{eqn:vee_absorbes_a_smash} and \eqref{eqn:smash_absorbes_a_vee} above were satisfied, then the distributivity of the meet operation over the join would imply the distributivity property of the join over the meet, i.e. that 
       \begin{equation}
         \label{eqn:vee_dist__over_wedge}
         \bc{E} \vee \bc{F_1\wedge F_2}=\bc{E\vee F_1}\wedge \bc{E\vee F_2}.
       \end{equation}

        The collection of Bousfield classes has an upper bound $\bc{\bb S}$ and a lower bound $\bc{0}$, and for every spectrum $E$ the following relations are satisfied: $$\bc{E} \vee \bc{0}=\bc{E}, \;\; \bc{E}\wedge \bc{0}=\bc{0}, \;\; \bc{E}\vee \bc{\bb S}=\bc{\bb S}, \;\; \bc{E} \wedge \bc{\bb S}=\bc{E}.$$

        We finally note that the join and meet operations are compatible with the partial order $\leq$. More explicitly: given spectra $E_i$ and $F_i$ for $i=1,2$ with the property that $\bc {E_1} \leq \bc{E_2}$ and $\bc{F_1} \leq \bc{F_2}$, then 
        \[\bc{E_1\vee F_1}\leq \bc{E_1\vee F_2} \leq \bc{E_2\vee F_2}\]
        and
        \[\bc{E_1\wedge F_1}\leq \bc{E_1\wedge F_2} \leq \bc{E_2\wedge F_2}.\] 
       
        \subsubsection{}
        Following \cite{MR543337} we denote by $\cal{DL}(S)$ the subclass of $\cal A(S)$ of those Bousfield classes satisfying the second idempotence law \eqref{eqn:idemp_of_smash}. The operations of meet and join restrict to $\cal{DL}(S)$ and the observations above imply that both the absorption laws and the distributive laws hold true in $\cal{DL}(S)$. This collection satisfies the axioms of a distributive lattice. The partial ordering $\leq$ on $\cal A(S)$ restricts to a partial ordering on $\cal{DL}(S)$. We wish to observe that for given $\bc E,\bc F \in \cal{DL}(S)$ their meet $\bc{E\wedge F}$ is actually their greatest upper bound. Most of the spectra we will consider later actually belong to this subclass: for instance, every homotopy ring spectrum $E$ belongs to $\cal {DL}(S)$, since $E$ is a retract of $E\wedge E$.

        \subsubsection{}
        We say that a Bousfield class $\bc{E} \in \cal A(S)$ has a complement if there is another Bousfield class $\bc{F}$ such that $\bc{E}\wedge\bc{F}=\bc{0}$ and $\bc{E}\vee \bc{F}=\bc{\bb S}$. It follows from the first distributive law \eqref{eqn:wedge_dist_over_vee} that if $\bc{E}$ has a complement then such complement is unique, and we denote it by $\bc{E}^c$. The same law implies that when $\bc E$ has a complement then $\bc{E} \in \cal{DL}(S)$. We denote by $\cal{BA}(S)$ the sublattice of $\cal{DL}(S)$ of those Bousfield classes admitting a complement. We do not expect every element of $\cal{DL}(S)$ to have a complement since the topological analogue of this statement has a counterexample, see Lemma 2.7 of \cite{MR543337}. Assume that both $\bc E, \bc F$ have complements, then the following equalities are satisfied:
        \begin{equation}
          \label{eqn:bc_complements_formulas}
          \bc E^{cc}=\bc E, \;\;\; \bc{E\vee F}^c=\bc{E}^c\wedge \bc{F}^c, \;\;\; \bc{E\wedge F}^c=\bc{E}^c\vee \bc{F}^c.
        \end{equation}
             


\section{Moore spectra} 
  \label{sec:moore_spectra}

  In algebraic topology given an abelian group $A$ one can construct a Moore spectrum $\bb S A$ associated to $A$. More precisely $\bb S A$ is a $(-1)$-connective spectrum with the property that all but its $0$-th singular homology groups vanish and whose $0$-th homotopy group is isomorphic to $A$. These three properties characterize $\bb S A$ up to non-canonical isomorphism. Furthermore one can consider the Bousfield class $\bc{\bb S A}$ and it turns out that $\bc{\bb S A}$ depends quite loosely on $A$. 
  
  More precisely we say that two abelian groups $A$  and $B$ have the same \emph{acyclicity type} if they satisfy the following requirements:
  \begin{enumerate}
    \item $A$ and $B$ are either both torsion groups or both non-torsion groups;
    \item for every prime number $l$, $A$ is uniquely $l$-divisible exactly when $B$ is uniquely $l$-divisible.
  \end{enumerate}
  In Proposition $2.3$ of \cite{MR551009} Bousfield shows that $\bc{\bb S A}=\bc{\bb S B} \in \cal{A}$ if and only if $A$ has the same acyclicity class of $B$.
  
  Let us turn to algebraic geometry and fix a perfect field $K$. In this section we introduce a weak algebraic version of Moore spectra and we give an explicit form to the localization functor they define. Our construction goes as follows. Starting with a homotopy module $\cal R_\ast \in \Pi_\ast(K)$ of a very special form, we construct a $(-1)$-connective spectrum $\bb S \cal R_\ast$ with the correct homotopy (and homology) in non-positive degrees: we call it Moore spectrum associated to $\cal R_\ast$. Unfortunately $\bb S \cal R_\ast$ strongly depends on choices that have to be made to perform its construction, and different choices might result in non-isomorphic Moore spectra, so that $\bb S \cal R_\ast$ is not well defined as a symbol. However, once we pass to the associated Bousfield class $\bc{\bb S \cal R_\ast}$, the dependence on the choices disappears, so that $\bc{\bb S \cal R_\ast}$ really only depends on $\cal R_\ast$.

 \subsection{Coning off homotopy elements} 
    \label{sub:coning_off_homotopy_elements}
  
    \begin{defin}
    \label{defin:kill_one_homot_elt}
      Let $f\in \Hom_{\spt_T^{\Sigma}(K)}(\Sigma^{p+q\alpha}\bb S, \bb S)$ be map representing an element $f \in \pi_p(\bb S)_{-q}(K)$. We define the \emph{mod $f$ Moore spectrum} as $C(f):=\hocofib(x: \Sigma^{p+q\alpha}\bb S \ra \bb S )$. Sometimes we use the alternative notation $ M(f):=C(f)$. Let $r>1$ be an integer and for $i=1,\dots, r$ assume that $f_i \in  \Hom_{\spt_T^{\Sigma}(K)}(\Sigma^{p_i+q_i\alpha}\bb S, \bb S)$ is a map representing an element $f_i \in \pi_{p_i}(\bb S)_{-q_i}(K)$. We call \emph{mod $\underline f$ Moore spectrum} the spectrum $C(f_1)\wedge \cdots \wedge C(f_r)$ and sometimes we will use the alternative notation $M(\underline f).$
    \end{defin}
    \begin{rmk}
      \label{rmk:killing_x_commutes_with_smashing}
      Observe that $f: \Sigma^{p+q\alpha} \bb S\ra \bb S$ acts on every spectrum $X\in \cal{SH}(K)$ by setting $f\cdot$ as the composition $\Sigma^{p+q\alpha}X\simeq \Sigma^{p+q\alpha} \bb S \wedge X \overset{f\wedge \rm{id}_X}{\ra} \bb S\wedge X \simeq X$. Moreover the derived smash product commutes with homotopy colimits so that $\hocofib(x\cdot: \Sigma^{p+q\alpha} X\ra X)\simeq C(x)\wedge X$ in $\cal {SH}(K)$. 
    \end{rmk}
    \begin{rmk}
      \label{rmk:diff_wrt_top_moore_spt}
      Let $f_i \in \pi_{0}(\bb S)_{q_i}(K)=K^{MW}_{q_i}(K)$ for $i=1,\dots,r$. Then multiplying by $f_i$ gives an exact sequence of homotopy modules 
      \[ \bigoplus\cal K^{MW}_{\ast - q_i} \overset{\cdot f_i}{\ra} \cal K^{MW}_{\ast} \ra \cal K^{MW}_\ast/(f_1,\dots f_r) \ra 0\]
      and it is tempting to call $M(\underline f)$ the Moore spectrum associated to the homotopy module $$\cal K^{MW}_\ast/(f_1,\dots,f_r).$$ However another choice of generators $g_1,\dots ,g_s$ for the ideal $(f_1,\dots, f_r) \subseteq K^{MW}_\ast(K)$ gives another Moore spectrum $M(\underline g)$ which in general will not be equivalent to $M(\underline f)$. Indeed consider the spectrum $H\cal K^{MW}_\ast \in \cal{SH}(K)$: by design its only non trivial homotopy module is $\underline\pi_0(H\cal K^{MW}_\ast)=\cal K^{MW}_\ast$. Let $f$ be an element of $\pi_{0}(\bb S)_{q}(K)$ and consider the Moore spectra $M(f)$ and $M(f,f)$. From the cofiber sequences 
      \[ H\cal K^{MW}_\ast \overset{\cdot f}{\ra} H\cal K^{MW}_\ast \ra H\cal K^{MW}_\ast\wedge M(f)  \]
      and 
      \[H\cal K^{MW}_\ast\wedge M(f) \overset{\cdot f}{\ra} H\cal K^{MW}_\ast\wedge M(f) \ra H\cal K^{MW}_\ast\wedge M(f,f)\]
      we immediately deduce that $H\cal K^{MW}_\ast\wedge M(f)$ has at most two non-trivial homotopy modules, namely
      \begin{equation*}
          \underline \pi_k ( H\cal K^{MW}_\ast\wedge M(f) )=
          \begin{cases}
            \cal K^{MW}_\ast/(f) & \text{if $k=0$,}\\
            \cal K^{MW}_\ast[f] & \text{if $k=1$, }\\
            0 & \text{else.}
          \end{cases}
        \end{equation*}
      On the other hand for $M(f,f)$ we have that 
      \begin{equation*}
          \underline \pi_k ( H\cal K^{MW}_\ast\wedge M(f,f) )=
          \begin{cases}
            \cal \cal K^{MW}_\ast/(f) & \text{if $k=0$,}\\
            \cal \cal K^{MW}_\ast[f] & \text{if $k=2$, }\\
            0 & \text{if $K\not = 0,1,2$}
          \end{cases}
        \end{equation*}
        and a short exact sequence
        \[ 0\ra \cal K^{MW}_\ast[f] \ra \underline \pi_1(H\cal K ^{MW}_\ast\wedge M(f,f)) \ra \cal K^{MW}_\ast/(f) \ra 0.\]
        In particular, as soon as $f$ is not invertible $M(f)\not\simeq M(f,f)$, hence with our definition we can not really talk about the Moore spectrum associated to $\cal K^{MW}_\ast/(f_1,\dots,f_r)$, but only of the Moore spectrum associated to a family of generators $f_1,\dots,f_r$ of $(f_1,\dots, f_r)$. For the applications we are interested in, $H\cal K^{MW}_\ast=H\underline \pi_0(\bb S)$ plays the role of $H\bb Z \in \cal{SH}$ in algebraic topology. It follows that, according to our definition, $M(f)$ might have non-trivial homology in degrees greater that $0$ which depends on the choice of the $f_i$'s. Nevertheless the two remaining properties of topological Moore spectra still hold in our setting.
      \end{rmk}
      \begin{lemma}
      \label{lemma:Moore_is_conn_and_with_corr_pi0}
        If $f_1,\dots, f_r$ are global sections of $\cal K^{MW}_\ast$ then $M(\underline f)$ is $(-1)$-connective,  
        \[\underline \pi_0 M(\underline f)=\cal K^{MW}_\ast/(f_1,\dots,f_r),\]
        and the canonical map $\bb S \ra M(\underline f)$ induces on $\underline\pi_0$ the quotient map $\cal K^{MW}_\ast \ra \cal K^{MW}_\ast/(f_1,\dots,f_r)$. 
      \end{lemma}
      \begin{proof}
        The proof can be made by induction along the lines of \ref{rmk:diff_wrt_top_moore_spt}.
      \end{proof}

    \begin{const}
      \label{const:modx_1_x_n_Moore_sp}
      Let $r$ be a non-negative integer, and for $i\in \{ 1,\dots ,r\}$ let $f_i: \Sigma^{p_i +q_i\alpha} \bb S \ra \bb S$ be maps of symmetric spectra lifting elements $f_i \in \pi_{p_i}(\bb S)_{-q_i}(K)$. We wish to recall the construction of the adic completion of a spectrum at the elements $f_1,\dots,f_r$. We start by assuming $r=1$ and choosing maps $f^n:\Sigma^{n(p+q\alpha)}\bb S \ra \bb S$ representing the powers $f^n \in \pi_{np}(\bb S)_{-nq}(K)$. For every integer $n\geq 1$ we consider the fibre sequence
      \[ \Sigma^{n(p+q\alpha)}\bb S \overset{f^n \cdot}{\ra}\bb S \ra C(f^n);\]
      and chose maps $p_n: C(f^n) \ra C(f^{n-1})$ making the diagram of fibre sequences 
      \begin{equation}
          \label{eqn:mod_x^n_mod_x^n-1_fib_seqs}
          \xymatrix{
            \Sigma ^{n(p+q\alpha)}\bb S \ar[rrr]^{f^{n}\cdot }\ar[d]_{\Sigma^{(n-1)(p+q\alpha)}f \cdot } & & & \bb S \ar[rr] \ar@{=}[d] & & C(f^n) \ar[d]_{p_n} \\
            \Sigma ^{(n-1)(p+q\alpha)} \bb S \ar[rrr]^{f^{(n-1)}\cdot} & & & \bb S \ar[rr] & & C(f^{n-1}).\\ 
          }
        \end{equation}
      to commute in $\cal {SH}(K)$. Note that we can choose equivalences
      \begin{equation}
          \label{eqn:mod_x^n/mod_x^n-1_is_mod_x}
        \begin{split}
          \hofib \Big (C(f^n) \overset{p_n}{\ra} C(f^{n-1}) \Big) & \approx \hocofib \Big(\Sigma ^{n(p+q\alpha)}\bb S \overset{\Sigma^{(n-1)(p+q\alpha)}f \cdot }{\ra} \Sigma ^{(n-1)(p+q\alpha)} \bb S \Big )\\
          & \approx \Sigma ^{(n-1)(p+q\alpha)}C(f)
        \end{split}
        \end{equation}
        and once we make this choice we obtain a fibre sequence
        \begin{equation}
            \label{eqn:mod_x_mod_x^2_mod_x_cof_seq}
            \Sigma^{(n-1)(p+q\alpha)}C(f)\ra C(f^n) \overset{p_n}{\ra} C(f^{n-1}).
          \end{equation}
        The $f$-adic completion $X^{\wedge}_f$ of a spectrum $X$ is thus defined as the homotopy inverse limit
        \begin{equation}
        \label{eqn:def_adic_comp}
          X^{\wedge}_f:=\holim_n\big (\cdots \ra X\wedge C(f^n) \overset{id_X\wedge p_n}{\rra} \cdots \ra X\wedge C(f)\big ).
        \end{equation}
        Since the operations of taking homotopy inverse limits and of smashing with a fixed $X$ commute with finite homotopy limits, the tower of fibre sequences introduced in \eqref{eqn:mod_x^n_mod_x^n-1_fib_seqs} yelds the fibre sequence
        \begin{equation}
        \label{eqn:def_seq_of_adic_comp}
          \holim_n \Sigma ^{n(p+q\alpha)}X \rra X \ra X^{\wedge}_f.     
         \end{equation} 
        The map on the right hand side of \ref{eqn:def_seq_of_adic_comp} is denoted $\chi_f(X)$ and is referred to as the $f$-adic completion map. A spectrum $X$ is called $f$-complete if $\chi_f(X)$ is an isomorphism in $\cal {SH}(K)$. We observe that when $X$ is a connective spectrum and $p>0$ then $X$ is $f$-complete.

        If $r\geq 1$ then for every spectrum $X$ we call $\underline f$-adic completion of $X$ the spectrum 
        \[X^{\wedge}_{\underline f}:=X^{\wedge}_{f_1}{}^{\wedge}_{f_2}{\cdots}^{\wedge}_{f_r}.\] The $\underline f$-adic completion map $\chi_{\underline f}(X)$ is defined as the composition of the natural maps 
        \[\chi_{f_r}(X^{\wedge}_{f_1,\dots,f_{r-1}}) \circ \cdots \circ \chi_{f_2}(X^\wedge_{f_1})\circ\chi_{f_1}(X).\]

        We observe that since homotopy inverse limits commute with finite limits we have a chain of equivalences
        \begin{equation}
            \label{eqn:big_diag_for_coomp_x1n_compl}
            \begin{split}
              {X^{\wedge}_{f_1}} \cdots^{\wedge}_{f_r} & \approx \holim_{j_r\in \bb N_{>0}} \big ( X^{\wedge}_{f_1} \cdots^{\wedge}_{f_{r-1}}\big ) \wedge C(f_r^{j_r}) \\
                & \approx \holim_{j_r\in \bb N_{>0}} \big ( \holim_{j_{r-1}\in \bb N_{>0}} \big (X^{\wedge}_{f_1} \cdots^{\wedge}_{f_{r-2}}\big )\wedge  C(f_{r-1}^{j_{r-1}}) \big ) \wedge  C(f_r^{j_r})\\
                & \approx \holim_{j_r\in \bb N_{>0}} \big ( \cdots \big ( \holim_{j_1\in \bb N_{>0}} X\wedge C(f_1^{j_1}) \big) \wedge \cdots \big )\wedge C(f_r^{j_r})\\
                & \approx \holim_{j_r\in \bb N_{>0}}  \cdots  \holim_{j_1\in \bb N_{>0}} X\wedge C(f_1^{j_1})  \wedge \cdots \wedge C(f_r^{j_r})
            \end{split}
          \end{equation}
          and by a co-finality argument we conclude that 
          \[{X^{\wedge}_{f_1}}\cdots^{\wedge}_{f_r}\simeq \holim_j X\wedge C(f_1^j)\wedge \cdots \wedge C(f_r^j).\]

    \end{const}

    \begin{prop}
      \label{prop:localization_at_mod_x_moore_spectrum}
      Let $f: \Sigma^{p+q\alpha}\bb S \ra \bb S$ be a map of symmetric spectra representing an element $f \in \pi_p(\bb S)_{-q}(K)$. Then for every spectrum $X$ the natural map $\chi_f(X): X\ra X^{\wedge}_f$ presents the $f$-completion of $X$ as the $C(f)$-localization of $X$ in $\cal {SH}(S)$.
      \end{prop}
      \begin{proof}
        Let $\phi: A\ra B$ be a map in $\cal{SH}(K)$. For every spectrum $F$ we get an induced map
        \begin{equation}
          [C(f)\wedge B,F] \ra [C(f)\wedge A,F],
        \end{equation}
        and since $C(f)$ has a Spanier-Whitehead dual $D(C(f))$, we get a natural map 
        \begin{equation}
        \label{eqn:mapdual}
          [B, D(C(f))\wedge F] \ra [A,D(C(f))\wedge F]
        \end{equation}
        functorially in $F \in \cal{SH}(K)$. $\phi$ is a $C(f)$-equivalence if and only if the map \eqref{eqn:mapdual} is an isomorphism for every $F$. In particular, for every $F$ the spectrum $D(C(f))\wedge F$ is $C(f)$-local. Furthermore it is an immediate check that $D(C(f))\simeq \Sigma ^{-(p+1)-q\alpha}C(f)$ in $\cal{SH}(K)$. Since local objects are stable under suspensions and de-suspensions, we conclude that for every spectrum $F$ the spectrum $C(f)\wedge F$ is $C(f)$-local. 

        Now we are ready to show that $X^{\wedge}_f$ is $C(f)$-local. As recalled above $X^{\wedge}_f$ is defined as an homotopy inverse limit of the tower \eqref{eqn:def_adic_comp} and since $C(f)$-local objects are stable under homotopy inverse limits we only need to know that each of the spectra $X\wedge C(f^n)$ is $C(f)$-local. This easily follows by induction: the base case being that $X\wedge C(f)$ is $C(f)$-local, which was observed above. Assume we know that $X\wedge C(f^{n-1})$ is $C(f)$-local. Using the fibre sequence 
        \begin{equation}
          \label{eqn:C_x^n/C_c^n-1_is_Cx}
          \Sigma^{(n-1)(p+q\alpha)}C(f) \wedge X \ra C(f^{n}) \wedge X \ra C(f^{n-1})\wedge X
        \end{equation}
        deduced from \eqref{eqn:mod_x_mod_x^2_mod_x_cof_seq} and the 2-out-of-3 property of $C(f)$-local objects in fibre sequences we conclude.

        In order to show that the canonical map $X \ra X^{\wedge}_f$ is a $C(f)$-local equivalence it suffices to show that $C(f)\wedge F \simeq 0$ in $\cal {SH}(S)$, where $F:=\hofib(X \ra X^{\wedge}_f)$. For this note that
        \begin{equation}
          \label{eqn:tower_desc_fibre_of_x_completion}
          F\approx \holim \Big ( \cdots \ra   \Sigma^{n(p+q\alpha)} X \overset{\Sigma^{(n-1)(p+q\alpha)}\cdot f}{\rra} \Sigma^{(n-1)(p+q\alpha)} X \ra \cdots \ra \Sigma^{p+q\alpha} X \big)
        \end{equation}
        and that $C(f)\wedge F\approx \hocofib(f\cdot : \Sigma^{p+q\alpha}F \ra F)$. However the multiplication by $f$ on $F$ is induced by the multiplication by $f$ on each component of the tower \eqref{eqn:tower_desc_fibre_of_x_completion}. Since the homotopy limit of a tower only depends on the pro-isomorphism class of the tower defining it (see Appendix \ref{sec:pro_spectra}), and since the multiplication by $f$ is clearly an isomorphism of towers, we conclude that the multiplication by $f$ on $F$ is actually an isomorphism in $\cal {SH}(S)$ and so its homotopy cofibre is zero.
        \end{proof}
        \begin{prop}
          \label{prop:localization_at_mod_x_1-x_n_moore_spectrum}
          Let $r$ be a non-negative integer and for every $i=1, \dots, r$ let $f_i: \Sigma^{p_i+q_i\alpha}\bb S \ra \bb S$ be maps of symmetric spectra representing elements $f_i \in \pi_{p_i}(\bb S)_{-q_i}(K)$. Then for every spectrum $X$ the natural map $\chi_{\underline f}(X): X\ra {X^{\wedge}_{f_1}} \cdots^{\wedge}_{f_r}$ is a $M(\underline f)$-localization of $X$ in $\cal {SH}(S)$.
        \end{prop}
        \begin{proof}
          Let us first set the notation $M(\underline f)=C(f_1)\wedge \cdots \wedge C(f_r)$. The spectrum $M(\underline f)$ has a Spanier-Whitehead dual which is explicitly given by 
          \begin{equation}
            D(M(\underline f)\big) \simeq \Sigma^{-P-Q\alpha}D(M(\underline f)),
          \end{equation}
          where $P=p_1 +\cdots + p_r+r$ and $Q=q_1 +\cdots + q_r$. Hence by running the same argument as in the proof of  \ref{prop:localization_at_mod_x_moore_spectrum} we deduce that the spectrum $F\wedge M(\underline f)$ is $M(\underline f)$-local for every $F\in \cal{SH}(K)$. In order to show that ${X^{\wedge}_{f_1}} \cdots^{\wedge}_{f_n}$ is $M(\underline f)$-local, thanks to the identification \eqref{eqn:big_diag_for_coomp_x1n_compl}
          \[{X^{\wedge}_{f_1}} \cdots^{\wedge}_{f_r} \simeq \holim_{(j_1,\dots,j_r)\in \bb N^{r}}(X\wedge C(f_1^{j_1})\wedge \cdots \wedge C(f_r^{j_r})),\]
          we only need to prove that each of the spectra $X\wedge C(f_1^{n})\wedge \cdots \wedge C(f_r^n)$ is $M(\underline f)$-local. This can be done by induction using iteratively the fibre sequence \eqref{eqn:C_x^n/C_c^n-1_is_Cx} with $f=f_1,f_2,\dots,f_r$ and the fact that $M(\underline f)$-local objects satisfy the $2$-out-of-$3$ property in fibre sequences.

           The natural map 
           \begin{equation}
             X\overset{\chi_{f_1}}{\rra} X^{\wedge}_{f_1} \overset{\chi_{f_2}}{\rra} X^{\wedge}_{f_1}{}^{\wedge}_{f_2}\overset{\chi_{f_3}}{\rra} \cdots \overset{\chi_{f_r}}{\rra} {X^{\wedge}_{f_1}} \cdots^{\wedge}_{f_r} 
           \end{equation}
           is a composition of $M(\underline f)$-equivalences since $\bc{M(\underline f)}\leq \bc{C(f_i)}$ and since $\chi_{f_i}$ is a $C(f_i)$-equivalence for every $i=1,\dots, r$ by \ref{prop:localization_at_mod_x_moore_spectrum}.
        \end{proof}
        
        \begin{lemma}
          \label{lemma:moore_loc_is_well_def}
          Let $I\subseteq K^{MW}_\ast(K)$ be an ideal and assume that it admits two finite sets of generators $\{f_1,\dots,f_r \}\subseteq I$ and $\{g_1,\dots g_s\} \subseteq I$ so that $(f_1,\dots,f_r)=I=(g_1,\dots,g_s)\subseteq K^{MW}_\ast$. Then $\bc{M(\underline f)}=\bc{M(\underline g)}$. 
        \end{lemma}
        \begin{proof}
          Note that by symmetry we only need to show that $\bc{M(\underline f)} =\bc{ M(\underline f)\wedge M(\underline g)}$ and by induction we reduce to the case $s=1$. For this consider the cofibre sequence
          \[M(\underline f) \ra M(\underline f) \ra M(\underline f)\wedge M(g). \]
          It is clear that $\bc{M(\underline f)\wedge M(g)} \leq \bc{M(\underline f)}$. To show the converse, take a spectrum $X$ such that $M(\underline f)\wedge M(g)\wedge X=0$. On one hand this condition implies that $g$ acts as an isomorphism on all the homotopy modules of $M(\underline f)\wedge X$. On the other hand we have that each of the $f_i$ acts nilpotently on all such homotopy modules and since $g$ is a linear combination of such $f_i$'s we deduce that $g$ acts nilpotently too. As a consequence all the homotopy modules of $M(\underline f)\wedge X$ must be zero and hence the spectrum has to be zero.
        \end{proof}

        \begin{rmk}
          \label{rmk:Moore_loc_well_def}
          Let $\cal R_\ast$ be a commutative monoid in homotopy modules of the form $\cal K^{MW}_\ast/\cal I$ where $\cal I$ is the unramified ideal generated by a finite set of global sections $\{f_1,\dots, f_r\} \subseteq \cal I(\spec K)\subseteq K^{MW}_\ast$. Lemma \ref{lemma:moore_loc_is_well_def} implies that, although our definition of Moore spectrum $M(\underline f)$ depends on the choice of generators $f_i$, the Bousfield class $\bc{M(\underline f)}$ actually does not. So it makes sense to define $\bc{\bb S\cal R_\ast}$ as $\bc{M(\underline f)}$ for some choice $\{f_1,\dots,f_r\}$ of a finite set of generators of $\cal I(\spec K)$.
        \end{rmk}

  \subsection{A remark on $\eta$-completions}
      
    Let $\eta \in \pi_0(\bb S)_{-1}(\spec K)$ be the algebraic Hopf map introduced in \ref{sub:homotopy_t_structure}. We have proven above in \ref{prop:localization_at_mod_x_moore_spectrum} that for every spectrum $X$ the $\eta$-completion map $\chi_\eta(X):X\ra X^\wedge_\eta$ is the $M(\eta)$-localization map in $\cal{SH}(K)$. We want to bring the discussion on $\eta$-completions a bit further.
    \begin{lemma}
      \label{lem:Xonehalf_is_eta_complete}
      Assume that the base field $K$ is not formally real. Then for every spectrum $X \in \cal {SH}(K)$ the spectrum $X[1/2]$ is $\eta$-complete.
    \end{lemma}
    \begin{proof}
      It follows from \cite[Ch. 2, Theorem 7.9]{MR770063} that there exists an integer $n$ such that $2^n$ acts as $0$ on the Witt ring of $K$. In particular we deduce that in $GW(K)$ the relation $2^n=h\omega$ holds, where $h$ is the rank $2$ hyperbolic space and $\omega$ is some element of $GW(K)$. It follows that $2^n\eta=h\omega\eta=0$ in $K^{MW}_\ast(K)$. It follows that on $X[1/2]$ the multiplication by $\eta$ is the zero map, which in view of \ref{const:modx_1_x_n_Moore_sp} is enough to conclude. 
    \end{proof}
    \begin{lemma}
      \label{lemma:compact_is_eta_complete}
      Assume that the base field has finite $2$-cohomological dimension. Then every dualizable object of $\cal {SH}(K)$ is $\eta$-complete.
    \end{lemma}
    \begin{proof}
      If $C$ be a dualizable object of $\cal{SH}(K)$, the operation of smashing with $C$ commutes to homotopy inverse limits. In particular $C^\wedge_\eta \simeq C\wedge \bb S^\wedge_\eta$ and thus we reduce to show that $\bb S$ is $\eta$-complete. In view of \ref{prop:fracture_square_for_mod_eta_cohomologies} we just need to show that the spectra $\bb S[1/2]$, $(\bb S^\wedge _2)[1/2]$ and $\bb S^\wedge_2$ are $\eta$-complete. For the two former spectra the previous claim follows from \ref{lem:Xonehalf_is_eta_complete}, while for $\bb S^\wedge_2$  the claim follows from the combination of Proposition 4 and Lemma 21 of \cite{MR2811716}.
    \end{proof}

 \subsection{Inverting homotopy elements} 
   \label{sub:inverting_homotopy_elements}
   
    \begin{defin}
      \label{defin:inverting_one_hom_element}
     Let $x: \Sigma^{p+q\alpha}\bb S \ra \bb S$ be a map of symmetric spectra representing an element $\pi_q(\bb S)_{-q}(K)$. The \emph{$x$-inverted Moore spectrum}, which we denote by $\bb S[x^{-1}]$ is the spectrum
     \[\xymatrix{
       \hocolim \big( \bb S \ar[rr]^{\cdot \Sigma^{-p,-q\alpha}x} & &  \Sigma^{-p,-q\alpha} \bb S \ar[rr]^{\cdot \Sigma^{-2p,- 2q\alpha}x} & & \Sigma^{-2p,-2q\alpha} \bb S \ar[rr]^{\cdot \Sigma^{-3p,-3q\alpha}x} & & \cdots \big)
      }.
      \]
      We denote the canonical map $\bb S \ra \bb S[x^{-1}]$ by $\iota_x(\bb S)$.
      \end{defin}
  
   \begin{rmk}
     \label{rmk:inv_one_elt_moore_is_ring_sp}
      Let $X$ be a spectrum. Recall that $x$ acts on $X$, as we observed in Remark \ref{rmk:killing_x_commutes_with_smashing}. Let us define the spectrum $X[x^{-1}]$ as 
      \[\xymatrix{
        \hocolim \big( X \ar[rr]^{\cdot \Sigma^{-p,-q\alpha}x} & &  \Sigma^{-p,-q\alpha} X \ar[rr]^{\cdot \Sigma^{-2p,- 2q\alpha}x} & & \Sigma^{-2p,-2q\alpha} X \ar[rr]^{\cdot \Sigma^{-3p,-3q\alpha}x} & & \cdots \big)
      }
      \]
      and call $\iota_x(X)$ the natural map $X \ra X[x^{-1}]$. Since the derived smash product commutes with homotopy colimits we have a canonical identification $X[x^{-1}]\simeq X\wedge \bb S[x^{-1}]$ in $\cal {SH}(K)$ under which the diagram
      \begin{equation}
        \label{eqn:X_inv_vs_XwedgeS_inv}
      \xymatrix{
      X \ar[d]^{\simeq} \ar[rr]^{\iota_x(X)} & &  X[x^{-1}]\ar[d]^{\simeq}\\
      X\wedge \bb S \ar[rr]^{id_X\wedge \iota_x(X)} & & X\wedge \bb S[x^{-1}]\\
      }
      \end{equation}
      commutes.
   \end{rmk}
   \begin{rmk}
     \label{rmk:sym_ring_sp_model}
      Assume that $R$ is an associative (resp. commutative) monoid in $\spt_T^{\Sigma}(K)$. We believe that it should be possible to construct $R[x^{-1}]$ as an associative (resp. commutative) monoid in $\spt_T^{\Sigma}(K)$ too. Similarly it should be possible to produce the map $\iota_x(R): X\ra X[x^{-1}]$ as a map of monoids in $\spt_T^{\Sigma}(K)$. However this is probably technical and in any case not needed for our purposes.
   \end{rmk}
   
   \begin{const} 
     \label{constr:inv_many_hom_elt}
     Let $n$ be a natural number and let us assume that we are given, for every $j=0,\dots, n$, an element $s_j \in \Hom_{\spt_T^{\Sigma}(S)}(\Sigma ^{p_j+q_j\alpha}\bb S,\bb S)$, where $p_j$ and $q_j$ are suitable integers. We denote by $S$ the ordered set of the $s_j$'s for $j=0,\dots, n$ with the order induced by the standard ordering $\{ 0 \leq 1 \leq \cdots \leq n\}$. We thus define for any given spectrum $X$ the spectrum
      \[ X[S^{-1}]:=X[s_0^{-1}]\cdots[s_n^{-1}].\]

      Note that an inductive application of remark \ref{rmk:inv_one_elt_moore_is_ring_sp} shows that we have a canonical equivalence $X[S^{-1}]\overset{\approx}{\ra} X \wedge \bb S[S^{-1}]$. Note that a priori the spectrum $\bb S[S^{-1}]$ might depend on the order chosen on $S$.
      
      Suppose now we are given a countable collection $S$ of elements $ x_i \in \Hom_{\spt_T^{\Sigma}(S)}(\Sigma ^{p_i+q_i\alpha}\bb S,\bb S)$ for $i\in \bb N$. We chose a total ordering on $S$ so that $x_n\leq x_{n+1}$ for every natural number $n$. The sets $S_i=\{x_0,\dots, x_i\}\subseteq \bb S$ have thus an induced total order. The above construction associates to every natural number $j$ a spectrum 
      $$X[S_j^{-1}].$$
      Moreover for every such $j$ we have a canonical map 
      \[
      \xymatrix{
      X[S_j^{-1}] \ar[rrr]^{\iota_{x_{j+1}}(X[S_j^{-1}])} & && X[S_j^{-1}][x_{j+1}^{-1}] \ar@{=}[r] & X[S_{j+1}^{-1}].\\
      }
       \]
      We define the spectrum $X[S^{-1}]$ as the homotopy colimit of the diagram we have just described:
      \[ X[S^{-1}]:=\hocolim_{j \in \bb N} \Big (X[S_0] \ra X[S_1^{-1}] \ra \cdots \ra X[S_{j+1}^{-1}] \ra \cdots\Big).\]

     \end{const}
     Once again the construction apparently depends on the choice of an ordering on the set $S$.

    \begin{prop}
      \label{prop:localization_at_S_inverted_moore_spectrum}

      Let $J$ be a countable set and assume we are given, for every $j\in J$, an element $s_j \in \Hom_{\spt_T^{\Sigma}(S)}(\Sigma ^{p_j+q_j\alpha}\bb S,\bb S)$, where $p_j$ and $q_j$ are suitable integers. Let $X$ be any spectrum and $\iota_S(X): X \ra  X[S^{-1}]$ be the map defined in \ref{defin:inverting_one_hom_element}. Then:
        \begin{enumerate}[label=M.\arabic*]
          \item \label{locstS1} For every integer $p$ there is a canonical isomorphism $\underline \pi_p(X)_\ast[S^{-1}]\simeq \underline \pi_p(X[S^{-1}])_\ast$ making the following diagram to commute;
          \begin{equation*}
            \xymatrix{
              \underline \pi_p(X)_\ast \ar[dr] \ar[rr]^{\underline \pi_p(\iota_S(X))_\ast} & & \underline \pi_p(X[S^{-1}])_\ast \\
                &  \underline \pi_p(X)_\ast [S^{-1}] \ar[ur]^{\simeq} & \\
            }
          \end{equation*}
          \item \label{locstS2} $\bb S[S^{-1}]$ has a natural structure of homotopy commutative ring spectrum with unit $e_{\bb S[S^{-1}]}=\iota_S(\bb S)$;
          \item \label{locstS3} The map $:\iota_S(X): X \ra  X[S^{-1}]$ exhibits $X[S^{-1}]$ as the $\bb S[S^{-1}]$-localization of $X$ in $\cal {SH}(S)$.
          \item \label{locstS4} A spectrum $X \in \cal{SH}(S)$ is $\bb S[S^{-1}]$-local if and only if each element of $S$ acts invertibly on all its homotopy modules $\underline\pi_p(X)_\ast$. If the elements of $S$ act by zero on the homotopy modules of a spectrum $X$, then $X$ is $\bb S[S^{-1}]$-acyclic.
          \end{enumerate}
    \end{prop}

    \begin{proof}
      \eqref{locstS1} follows by noting that taking homotopy groups and sheafifying commute with filtered homotopy colimits.

      For \eqref{locstS2} note that, since the derived smash product commutes with homotopy colimits, we have a natural isomorphism in $\cal{SH}(K)$  
      \[\epsilon: \bb S[S^{-1}] \wedge \bb S[S^{-1}] \overset{\simeq}{\ra} \bb S[S^{-1}][S^{-1}].\]
      Moreover by the above point the natural map $ \iota_{S}(\bb S[S^{-1}]): \bb S[S^{-1}] \ra \bb S[S^{-1}][S^{-1}]$ induces an isomorphism on homotopy sheaves and so it is an isomorphism. So we have a multiplication map in $\cal {SH}(K)$
      \[ \mu_{\bb S[S^{-1}]} : \bb S[S^{-1}]\wedge \bb S[S^{-1}] \ra \bb S[S^{-1}] \] obtained by composing the first isomorphism with the inverse of $\iota_{S}(\bb S[S^{-1}])$ and it is thus an isomorphism in $\cal{SH}(K)$. The map $\mu_{\bb S[S^{-1}]}$ is right and left unital with respect to $e_{\bb S[S^{-1}]}=\iota_S(\bb S)$ thanks to the commutativity of \eqref{eqn:X_inv_vs_XwedgeS_inv}; moreover $\mu_{\bb S[S^{-1}]}$ makes $\bb S[S^{-1}]$ into an homotopy commutative ring spectrum since homotopy colimits commute with each other.
      
      \eqref{locstS3} For every spectrum $X$, the spectrum $X[S^{-1}]$ is a $\bb S[S^{-1}]$-module and so it is $\bb S[S^{-1}]$-local. The commutativity of \eqref{eqn:X_inv_vs_XwedgeS_inv} implies that we can identify $\iota_S(X)$ with $id_X\wedge \iota_S(\bb S)$, which is an $\bb S[S^{-1}]$-equivalence since the multiplication $\mu_{\bb S[S^{-1}]}$ is an isomorphism in $\cal{SH}(K)$.
      
      \eqref{locstS4} Obvious.

    \end{proof}

    \begin{cor}
    \label{cor:S-inv_Moore_sp_is_well_defined}
    Let $J\subseteq K^{MW}_\ast(K)$ be a multiplicative system and assume that $S_1, S_2\subseteq J$ are two possibly infinite sets of generator of $J$. Then there is a natural isomorphism of homotopy ring spectra $\bb S[S_1^{-1}]\simeq \bb S[S_2^{-1}]$ in $\cal {SH}(K)$. 
    \end{cor}
    \begin{proof}
    Point \eqref{locstS4} of Proposition \ref{prop:localization_at_S_inverted_moore_spectrum} implies that $\bc{\bb S[S_1^{-1}]}=\bc{\bb S[S_2^{-1}]}$. Combining with point three of Proposition \ref{prop:localization_at_S_inverted_moore_spectrum} we deduce that there is an isomorphism $\phi: \bb S[S_1^{-1}]\simeq \bb S[S_2^{-1}]$ under which the unit maps $\iota_{S_1}(\bb S)$ and $\iota_{S_2}(\bb S)$ coincide. Finally $\phi$ is a map of homotopy ring spectra by the very definition of the multiplication on $\bb S[S_i^{-1}]$.
    \end{proof}
    \begin{defin}
    \label{defin:mixed_moore_sp}
    Let $\cal R_\ast$ be a homotopy module and assume there exist an unramified ideal $\cal I \subseteq \cal K^{MW}_\ast$ and an unramified multiplicative system $\cal J\subseteq \cal K^{MW}_\ast$ such that the following conditions are verified:
    \begin{enumerate}
      \item $\cal I$ is the unramified ideal generated by a finite set of its global sections $\{f_1,\dots,f_r\} \subseteq \cal I(\spec K)$;
      \item $\cal J$ is the unramified multiplicative system generated by a possibly infinite set of its global sections $S:=\{g_j\}_{j\in J} \subseteq \cal J(\spec K)$.
      \item there is an isomorphism $\phi: \cal R_\ast \simeq (\cal K^{MW}_\ast/\cal I)[\cal J^{-1}]$.
    \end{enumerate}
    We define the Bousfield class
    \[\bc{\bb S \cal R_\ast} :=\bc{M(\underline f)\wedge \bb S[S^{-1}]}.\]
    \end{defin}
    \subsection{} 
    In definition \ref{defin:mixed_moore_sp}, the Bousfield class $\bc{\bb S \cal R_\ast}$ might depend on the choice of a generating set for $\cal I $ and $\cal J$, at least if $\cal J\not =1$ (see \ref{rmk:Moore_loc_well_def}). Since for our main results $\cal J$ will always be empty, this will cause no problem to us.
    \begin{cor}
    \label{cor:Cf_has_a_comp}
    Let $\{f_i\}_{i \in I}$ be a finite set of elements of $K^{MW}_\ast(\spec K)$ and let $S:=\{g_j\}_{j \in J}$ be a possibly infinite set of elements in $K^{MW}_\ast(\spec K)$. Denote by $M$ the spectrum $M(\underline f) \wedge \bb S [S^{-1}]$. Then $\bc{M}$ has a complement.
   \end{cor}

    \begin{proof}
    We first prove that given an element $f\in K^{MW}_\ast(\spec K)$, $\bc{C(f)}$ is the complement of $\bc{\bb S[f^{-1}]}$. Indeed on one hand $C(f)\wedge \bb S[f^{-1}]=0$ since $f$ acts invertibly on $\bb S[f^{-1}]$ by \ref{prop:localization_at_S_inverted_moore_spectrum}\ref{locstS1}. On the other hand if we have a spectrum $X$ such that $X\wedge C(f)=0$ we conclude that $f$ acts invertibly on $X$ so that by \ref{prop:localization_at_S_inverted_moore_spectrum}\ref{locstS4} $X$ is $\bb S[f^{-1}]$-local. If follows that if we also assume that $X$ is $\bb S[f^{-1}]$-acyclic, then $X=0$. This shows that $\bc{C(f)}\vee \bc{\bb S[f^{-1}]}=\bc {\bb S}$.

    In particular, since $M(\underline f)=\wedge _{i \in I} C(f_i)$, the formulas \eqref{eqn:bc_complements_formulas} allow us to conclude that
    \[\bc{M(\underline f)}^c= \bc{\vee_{i\in I} \bb S[f_i^{-1}]}.\]

    We are thus reduced to prove that \[\bc{\bb S[S^{-1}]}^c=\bc{\vee_{j\in J} C(g_j)}.\]
    On one hand $\bb S[S^{-1}] \wedge \big (\vee_j C(g_j)\big)=0$ by Proposition \ref{prop:localization_at_S_inverted_moore_spectrum}\ref{locstS1}. On the other hand if a spectrum $X$ is $\big (\vee _j C(g_j) \big )$-acyclic then each of the $g_j$ acts invertibly on $X$, so that $X$ is $\bb S[S^{-1}]$-local by Proposition \ref{prop:localization_at_S_inverted_moore_spectrum}\ref{locstS4}. Thus, if we also assume that $X$ is $\bb S[S^{-1}]$-acylic we conclude that $X=0$. This implies that $\bc{\bb S[S^{-1}]} \vee \bc{\vee_{j\in J} C(g_j)}=\bc{\bb S}$ and this concludes.
    \end{proof}

    \begin{rmk}
    \label{rmk:Moore_lattice}
    In the discussion carried out around \cite[Proposition 2.13]{MR543337}, Bousfield observes that the assignment $A \mapsto \bb S A$ gives an embedding of the set of acyclicity classes of abelian groups into the set (thanks to Ohkawa's Theorem \cite{MR1712921}) of Bousfield classes $\cal A$. The image is denoted by $\cal{MBA}$ and is strictly smaller than $\cal{BA}$. Note that $\cal{MBA}$ is identified with the set of subsets of 
    \[ \{ \bb S \bb Q, \bb S \bb Z/2, \bb S \bb Z/3, \bb S \bb Z/5, \dots  \};  \]
    In other words $\cal{MBA}$ is in bijection with the set of subsets of $\spec(\bb Z)$. We believe one might try to prove a similar statement using as an input Thornton's computation of the homogeneous spectrum of $K^{MW}_\ast(K)$ \cite{MR3503978}. We do not pursue this problem here.     
    \end{rmk}

\section{Localization at some ring homology theories}
  \label{sec:localizations}
  In this section we prove our main results about homology localizations. 
  We dedicate \ref{sub:fracture_squares} to some technical results and \ref{sub:assumption_A1} to the statement of our main assumption. These are later used along \ref{sub:comparision_of_E_and_pi_0E-loc} and in the proof of Theorem \ref{thm:red_to_moore_spt}, which is the main result of this section. Throughout this section we work over a perfect base field $K$.

  \subsection{Fracture squares} 
    \label{sub:fracture_squares}
    
    \begin{prop}
      \label{prop:fracture_square_for_mod_eta_cohomologies}
      Let $E$ be a spectrum. Let $f$ be an element of $\pi_p(\bb S)_q(\spec K)$ and denote $E\wedge C(f)=:E/f$. Then for every spectrum $X \in \cal{SH}(K)$ we have a canonical homotopy pull-back square
      \begin{equation}
        \label{eqn:frc_square_for_mod_eta_cohomologies}
      \xymatrix{
        X_E \ar[r]^{\lambda^E_{E/f}} \ar[d]^{\lambda ^E_{E[f^{-1}]}} & X_{E/f} \ar[d]^{i} \\
        X_{E[f^{-1}]} \ar[r]^{c} & (X_{E/f})_{E[f^{-1}]}\\
      }
      \end{equation}
      Here $\lambda^E_{E/f}$ is the canonical map through which $\lambda_E(X): X \ra X_E$ factors, given that $\bc{E/f}\leq\bc{E}$. Similarly $\lambda^E_{E[f^{-1}]}$ is induced by the inequality $\bc{E[f^{-1}]} \leq \bc E$. The map $c=\lambda_{E[f^{-1}]}(\lambda_{E/f}(X))$ is the $E[f^{-1}]$-localization of the map $\lambda_{E/f}(X): X \ra X_{E/f}$, and finally $i=\lambda_{E[f^{-1}]}(X_{E/f})$.
    \end{prop}
    \begin{proof}
      The proof is completely analogous to the topological case. We start by denoting by $P(X)$ the homotopy pull-back of the diagram
      \begin{equation}
        \label{eqn:part_frc_square_for_mod_eta_cohomologies}
      \xymatrix{
         & X_{E/f} \ar[d]^i \\
        X_{E[f^{-1}]} \ar[r]^c & (X_{E/f})_{E[f^{-1}]}.\\
      }
      \end{equation}
      The first step consists in showing that $P(X)$ is $E$-local. In order to achieve this, note that $E/f=E\wedge C(f)$ and $E[f^{-1}]=E\wedge \bb S[f^{-1}]$ so that all the objects appearing in the diagram \eqref{eqn:part_frc_square_for_mod_eta_cohomologies} are actually $E$-local. Since $E$-local objects are closed under homotopy inverse limits, we conclude that $P(X)$ is itself $E$-local.

      As a second step we consider the natural map $u: X \ra P(X)$ in $\cal{SH}(K)$ which is induced by the localization maps $\lambda_{E/f}(X): X \ra X_{E/f}$ and $\lambda_{E[f^{-1}]}(X): X\ra X_{E[f^{-1}]}$. We aim to prove that $u$ is an $E$-equivalence, and thanks to \ref{cor:Cf_has_a_comp} it suffices to show that $u$ is both an $E/f$-equivalence and a $E[f^{-1}]$-equivalence. In particular, by the 2-out-of-3 property of local equivalences, we reduce to show that
      \begin{itemize}
        \item the natural map $P(X)\overset{\alpha}{\ra} X_{E[f^{-1}]}$ is a $E[f^{-1}]$-equivalence;
        \item the natural map $P(X)\overset{\beta}{\ra} X_{E/f}$ is a $E/f$-equivalence.
      \end{itemize}
      We proceed in order and start with $\alpha$. Since smashing with $E[f^{-1}]$ preserves homotopy (co)fibre sequences and direct sums, it preserves homotopy pull-back squares. We can thus apply it to \ref{eqn:part_frc_square_for_mod_eta_cohomologies} and we need to check that 
      \[ \hofib(P(X) \overset{\alpha}{\ra} X_{E[f^{-1}]}) \wedge E[f^{-1}]=0. \] 
      But this is actually the case since 
      \[ \hofib(\alpha) \wedge E[f^{-1}] \overset{\approx}{\ra} \hofib (i)\wedge E[f^{-1}]\]
      and $i$ is an isomorphism in $E[f^{-1}]$-homology by definition.
      
      For $\beta$ we proceed similarly so we just need to show that $c$ is an $E/f$-equivalence, i.e.~that $c\wedge E/f$ is an equivalence. On one hand, if we look at the source of $c\wedge E/f$, we see that 
      \begin{equation}
        \begin{split}
        X_{E[f^{-1}]} \wedge E/f &\simeq  X_{E[f^{-1}]} \wedge (\hocofib(\Sigma^{1,1}E \overset{f\cdot}{\rra} E))\\            & \simeq X_{E[f^{-1}]} \wedge M(f) \wedge E.
        \end{split}
      \end{equation}
      On the other hand $X_{E[f^{-1}]}$ is $\bb S[f^{-1}]$-local, which by \ref{prop:localization_at_S_inverted_moore_spectrum}\ref{locstS4} is equivalent to the fact that $f$ acts invertibly on its homotopy modules, and this implies that $X_{E[f^{-1}]} \wedge M(f) \simeq 0$ forcing the source of $c\wedge E/f$ to be zero. The same argument can be made for the target of $c\wedge E/f$, so that this map is actually an isomorphism.
    \end{proof}

    \begin{cor}
      \label{cor:E/eta_loc_is_E_cofeta_loc}
      Let $E$ and $X$ be spectra in $\cal{SH}(K)$. Then $X_{E/f}\simeq (X_E)_{C(f)}$.
    \end{cor}
    \begin{proof}
      We start by considering the following commutative diagram in $\cal{SH}(K)$:
      \begin{equation}
        \label{eqn:E/eta_loc_is_E_cofeta_loc}
      \xymatrix{
       X_E \ar[r]^{\lambda_1}\ar[d]^{\lambda_3} & (X_E)_{C(f)}\ar[d]^{\lambda_4}\\
       X_{E/f} \ar[r]^{\lambda_2} & (X_{E/f})_{C(f)}\\
      }
      \end{equation}
      where $\lambda_1=\lambda_{C(f)}(X_E)$, $\lambda_3=\lambda^E_{E/f}$ is the map introduced above, $\lambda_4=\lambda_{C(f)}(\lambda_3)$, and finally $\lambda_2=\lambda_{C(f)}(X_{E/f})$.
      
      Since $\bc{E/f}\leq\bc{C(f)}$, $\lambda_2$ is actually an isomorphism in $\cal{SH}(K)$ and we are left to prove that $\lambda_4$ is too. For this we apply the $C(f)$-localization functor to the square \eqref{eqn:frc_square_for_mod_eta_cohomologies} and we use \ref{prop:fracture_square_for_mod_eta_cohomologies} to reduce the proof to checking that 
      \[X_{E[f^{-1}]} \overset{c}{\ra} (X_{E/f})_{E[f^{-1}]} \] is a $C(f)$-equivalence. Now both the source and target of $c$ are $\bb S[f^{-1}]$-local, being in fact $E[f^{-1}]$-local. Thus, after $C(f)$-localization both source and target of $c$ become zero.
    \end{proof}

    \begin{cor}
      \label{cor:E/x,y-loc_is_E_cofx_cof_y_loc}
      Let $r$ be a positive integer and for every $i=1, \dots, r$ let $f_i\in \underline \pi_{p_i}(\bb S)_{q_i}(\spec K)$. Let $M:=C(f_1)\wedge \cdots \wedge C(f_r)$ be the Moore spectrum associated to the global sections $\{f_1, \dots ,f_r\}$. Then for every pair of spectra $E$ and $X$ 
      \begin{equation}
        X_{E\wedge M} \simeq (\cdots((X_E)_{C(f_1)})_{C(f_2)}\cdots)_{C(f_r)}\approx (X_E)_M.
      \end{equation}
    \end{cor}
    
    \begin{proof}
      Apply inductively Corollary \ref{cor:E/eta_loc_is_E_cofeta_loc}.
    \end{proof}

  \subsection{Assumption} 
    \label{sub:assumption_A1}
    $E$ is an object of $\cal{SH}(K)$ satisfying the following properties:
    \begin{enumerate}
      \item $E$ is a homotopy commutative ring spectrum, i.e. a commutative monoid in $\cal{SH}(K)$;
      \item $E$ is $(-1)$-connective for Morel's homotopy $t$-structure;
      \item there is a finite set 
      \[\{f_i\}_{i \in I} \subseteq K^{MW}_\ast,\]
      a possibly infinite set 
      \[\{g_j\}_{j \in J}\subseteq K^{MW}_\ast,\]
      and an isomorphism of homotopy modules
      \[\varphi: \underline \pi_0 E \simeq (\cal K^{MW}_\ast/\cal I)[\cal J^{-1}],\]
      where $\cal I$ is the unramified sheaf of ideals generated by $\{f_i\}_{i \in I}$ and $\cal J$ is the unramified sheaf of multiplicative systems genrated by $\{g_j\}_{j \in J}$.
      \item The square 
      \[
      \xymatrix{
        \underline \pi_0 \bb S \ar[r] \ar[d]^{\simeq}_{\sigma} & \underline \pi_0 E \ar[d]^{\varphi}_{\simeq}\\
        \cal K^{MW}_\ast \ar[r] & (\cal K^{MW}_\ast/\cal I)[\cal J^{-1}]\\
      }
      \] is commutative.
      Here the upper horizontal map is that induced on $\underline \pi_0$ by the unit $e_E:\bb S \ra E$, the lower horizontal map is the canonical map, and $\sigma$ is Morel's isomorphism introduced in \ref{sub:homotopy_t_structure}.
    \end{enumerate}
    

  \subsection{Comparison of $E$- and $\pi_0E$-localization} 
    \label{sub:comparision_of_E_and_pi_0E-loc}
    
    \begin{prop}
      \label{prop:fund_ineq_of_loc}
      Let $E$ be a spectrum satisfying Assumption \ref{sub:assumption_A1} and let $H\underline \pi_0E$ be the Eilenberg-Maclane spectrum associated to the $0$-th homotopy module of $E$. Then:
      \[\bc{H\underline \pi_0E} \; \leq  \; \bc{E} \; \leq \; \bc{\bb S \underline \pi_0E}.\]
      \end{prop}
    
    \begin{proof}
      Since $E$ is a $(-1)$-connected spectrum by Assumption \ref{sub:assumption_A1}, the projection to the Postnikov truncation induces a map $p:E\ra H\underline \pi_0E$. Using the fibre sequence \eqref{eqn:postfibseq}, the multiplication map $\mu_E$ of $E$ is easily seen to descend to a multiplication map $\mu_{H\underline \pi_0E}$. Such a multiplication is homotopy associative and commutative since $\mu_E$ is. The map $\mu_{H\underline \pi_0E}$, together with the composition $p\circ e_E=:e_{H\underline \pi_0 E}$ of the projection $p$ and the unit $e_E: \bb S \ra E$, makes $H\underline \pi_0 E$ into a homotopy commutative ring spectrum. With these definitions $p$ is automatically a ring map. Using the commutative diagram
      \begin{equation}
        \xymatrix{
          H\underline \pi_0 E \wedge H\underline \pi_0 E \ar[rr]^{\mu} & & H\underline \pi_0 E\\
          E\wedge H\underline \pi_0 E \ar[u]^{p\wedge id} & & \bb S\wedge H\underline \pi_0 E \ar[u]^{\simeq}\ar[ull]^{e_{H\underline \pi_0 E}\wedge id}\ar[ll]^{e_E\wedge id},\\
        }
      \end{equation}
      it is immediately clear that if $X$ is $E$-acyclic then $X$ is $H\underline \pi_0 E$-acyclic too. It remains to show that $\bc{E}= \bc{E\wedge \bb S\underline \pi_0 E}$, which directly implies our claim that $\bc{E}\leq \bc{\bb S\underline \pi_0 E}.$

      For this, recall that
       \[\bc{\bb S \underline\pi_0E}=\bc{C(f_1)\wedge \cdots \wedge C(f_n) \wedge \bb S[S^{-1}]}.\] 
      From proposition \ref{cor:Cf_has_a_comp}, the Bousfield class $\bc{\bb S \underline \pi_0E}$ has a complement $\bc{\bb S \underline \pi_0E}^c=\bc{M}$, where 
        \[ M:=\big (\bigvee_{i\in I} \bb S[f_i^{-1}] \big ) \vee \big(\bigvee_{j\in J} C(g_j)\big).\]
      In particular
      \[ \bc{E} = \bc{E}\wedge \bc{(\bb S \underline \pi_0E \vee M)} = \bc{E\wedge \bb S \underline \pi_0E} \vee \bc{E\wedge M}.\]
        We observe that 
        \[
          E\wedge M \simeq \Big ( \bigvee_{i\in I}  E\wedge\bb S[f_i^{-1}] \Big ) \vee \Big( \bigvee_{j\in J} E\wedge C(g_j) \Big).
        \] 
      Since by Assumption \ref{sub:assumption_A1} $E$ is a ring spectrum, the homotopy modules $\underline \pi_kE$ are a modules over $\underline \pi_0 E$. Thus for every $i$, $f_i$ acts by zero $\underline \pi_0 E$. It follows that for every $i=1,\dots,r$ and every $k\in \bb Z$, $f_i$ acts by zero on $\underline \pi_k E$. This implies that
        \[ \bigvee_{i\in I}  E\wedge\bb S[f_i^{-1}] =0.\]
        Similarly, for every $j\in J$ the element $g_j$ acts invertibly on $\underline \pi_0 E$ and hence on each of the homotopy modules $\underline \pi_k E$ for $k \in \bb Z$. We deduce that
        \[\bigvee_{j\in J} E\wedge C(g_j)=0,\]
        and this concludes the proof.
    \end{proof}

    \begin{lemma}
      \label{prop:con_implies_HK^MW-loc}
      Let $X$ be any connective spectrum. Then $X$ is $H\cal K^{MW}_\ast$-local.
    \end{lemma}
    \begin{proof}
      The statement follows directly by looking at the Postnikov tower, using that $H\cal K^{MW}_\ast$-local objects form a subcategory of $\cal{SH}(K)$ which is closed under homotopy inverse limits of towers, extensions, and containing $H\cal K^{MW}_\ast$-modules. We are also using here that the layers of the Postnikov tower are $H\cal K^{MW}_\ast$-modules.
    \end{proof}

    \begin{lemma}
      \label{lemma:HR/x_vs_H_R/x}
      Let $\cal R$ be an commutative monoid in homotopy modules and let $f \in \cal R_q(K)$ be a global section. Consider in addition the following two spectra:  
      \begin{equation}
        \hocofib \big ( \Sigma^{q\alpha}H\cal R_\ast \ra H \cal R_\ast \big )=: C_{H\cal R_\ast}(f)
      \end{equation}
      and 
      \begin{equation}
        H(\cal R/(f)_\ast).
      \end{equation}
      Then $\bc{C_{H\cal R_\ast}(f)} \leq \bc{H(\cal R/(f)_\ast)}$. 
    \end{lemma}

    \begin{proof}
      From the exact sequence of homotopy modules
      \begin{equation}
        0 \ra \cal K_{\ast-q} \ra \cal R_{\ast-q}\overset{\cdot f} \ra \cal R_\ast \ra \cal R/(f)_\ast \ra 0
      \end{equation}
      we deduce that the cofibre $C_{H\cal R_\ast}(f)$ of the multiplication by $f$ on $H\cal R_\ast$ has the following homotopy modules:
      \begin{equation}
        \underline \pi_k C_{H\cal R_\ast}(f)=
        \begin{cases}
          \cal R/(f)_\ast & \text{if $k=0$,}\\
          \cal K_\ast & \text{if $k=1$, }\\
          0 & \text{else.}
        \end{cases}
      \end{equation}
      In particular it follows that we have a fibre sequence 
      \begin{equation}
        \Sigma ^1 H\underline \pi_1(C_{H\cal R_\ast}(f)) \ra C_{H\cal R_\ast}(f) \ra H\underline\pi_0 (C_{H\cal R_\ast}(f)).
      \end{equation}
      relating $C_{H\cal R_\ast}(f)$ with its $0$-connected cover and its $\leq 0$ truncation. Observe that $\Sigma ^1 H\underline \pi_1(C_{H\cal R_\ast}(f))$ is a module over $H\underline\pi_0 C_{H\cal R_\ast}(f)=H(\cal R/(f)_\ast)$. More generally for every spectrum $X$, the spectrum $\Sigma ^1 H\underline \pi_1(C_{H\cal R_\ast}(f)) \wedge X$ is as well. In particular one has a commutative diagram in $\cal{SH}(K)$ of the form
      \begin{equation}
        \xymatrix{
        H(\cal R/(f)_\ast)  \wedge \Sigma ^1 H\underline \pi_1(C_{H\cal R_\ast}(f)) \wedge X \ar [r] &  \Sigma ^1 H\underline \pi_1(C_{H\cal R_\ast}(f)) \wedge X \\
        \bb S \wedge \Sigma ^1 H\underline \pi_1(C_{H\cal R_\ast}(f))\wedge X \ar[u] \ar[ur]^{\approx} & \\
        }
      \end{equation}
      from which we immediately conclude.
    \end{proof}

    \begin{cor}
      \label{cor:HK^MW_mod_stuff_vs_H_KMW_mod_stuff}
      Assume that $E$ is a homotopy commutative ring spectrum satisfying assumption \ref{sub:assumption_A1}. Then $\bc{H\cal K^{MW}_\ast\wedge \bb S{\underline\pi_0E_\ast}} \leq \bc{H\underline\pi_0E_\ast}$.
    \end{cor}

    \begin{proof}
      Assume for the moment that $J=\emptyset$. If $I=\{f_1\}$ is a singleton the statement is just \ref{lemma:HR/x_vs_H_R/x} in the case $\cal R_\ast=\cal K^{MW}_\ast$, so that $\bc{H\cal K^{MW}_\ast\wedge C(f_1)}\leq \bc{H({\cal K^{MW}_\ast/f_1})} $. Since by construction $\bb S \underline \pi_0E_\ast=C(f_1)\wedge \cdots \wedge C(f_r)$ we can proceed by smashing with one $C(f_i)$ at the time. Indeed 
      \[\bc{H\cal K^{MW}_\ast\wedge C(f_1)}\wedge \bc{ C(f_2)} \leq \bc{H({\cal K^{MW}_\ast/f_1})} \wedge \bc{C(f_2)}\]
      so that 
      \[ \bc{H\cal K^{MW}_\ast\wedge C(f_1) \wedge  C(f_2)} \leq \bc{H({\cal K^{MW}_\ast/f_1}) \wedge C(f_2)}\]
      and finally, using \ref{lemma:HR/x_vs_H_R/x} with $\cal R_\ast=\cal K^{MW}_\ast/(f_1)$ and $f=f_2$, we get that
      \[ \bc{H(\cal K^{MW}_\ast/(f_1)) \wedge C(f_2)} \leq \bc{ H(\cal K^{MW}_\ast/(f_1,f_2))}\]
      so we conclude that 
      \[\bc{H\cal K^{MW}_\ast\wedge C(f_1) \wedge  C(f_2)} \leq  \bc{ H(\cal K^{MW}_\ast/(f_1,f_2))}.\]
      In a finite number of iterations we conclude that $\bc{H\cal K^{MW}_\ast\wedge M(\underline f)} \leq \bc{H\underline \pi_0 M(\underline f)_\ast}$ where $M(\underline f)=C(f_1)\wedge \cdots \wedge C(f_r)$. Finally we observe that 
      \[\bc{H\cal K^{MW}_\ast \wedge \bb S\underline \pi_0 E_\ast} = \bc{H\cal K^{MW}_\ast \wedge M(\underline f)\wedge \bb S[S^{-1}]} \leq \bc {H(\underline \pi_0 M(\underline f)_\ast) \wedge \bb S [S^{-1}]} = \bc{H\underline \pi_0 E_\ast }.\]
    \end{proof}

    \begin{thm}
    \label{thm:red_to_moore_spt}
      Let $E$ be a connective homotopy commutative ring spectrum satisfying assumption \ref{sub:assumption_A1} in the special case that $J=\emptyset$. Then for every connective spectrum $X$ we have that
      \[X_{\bb S \underline \pi_0E}\simeq X_{E}.\]
     
    \end{thm}

    \begin{rmk}
    \label{rmk:}
      For this statement we could hope to have that \[ \bc{\bb S \underline \pi_0E} = \bc{E} \] i.e. that the connectivity of $X$ is not needed for the above theorem to hold. But already in algebraic topology there are counterexamples to such a statement.
    \end{rmk}

    \begin{proof}
      Thanks to proposition \ref{prop:fund_ineq_of_loc} we know that any connective spectrum $X$ the localization map $X \ra X_{\bb S \underline \pi_0E}$ is an $E$-equivalence so we only have to check that $X_{\bb S \underline \pi_0E}$ is $E$-local. Now consider that 
      \[
      X \ra X_{H\cal K^{MW}_\ast}
      \] is an equivalence by proposition \ref{prop:con_implies_HK^MW-loc} so that 
      \[
      X_{\bb S \underline \pi_0E} \ra (X_{H\cal K^{MW}_\ast})_{\bb S \underline \pi_0E}
      \] is an equivalence too.
      In particular, by combining this with the result of corollary \ref{cor:E/eta_loc_is_E_cofeta_loc} we deduce that
      \[
      X_{\bb S \underline \pi_0E} \overset{\approx}{\ra} (X_{H\cal K^{MW}_\ast})_{\bb S \underline \pi_0E}  \approx X_{H\cal K^{MW}\wedge \bb S \underline \pi_0E}.
      \] 
      Finally we apply Proposition \ref{prop:fund_ineq_of_loc} and \ref{lemma:HR/x_vs_H_R/x} to deduce that $X_{H\cal K^{MW}\wedge \bb S \underline \pi_0E}$ is $E$-local and this concludes.
    \end{proof}

    \begin{rmk}
      We believe that the results of this section, together with those of section \ref{sec:moore_spectra}, can be generalized to the derived category of modules over a connective $\bf E_\infty$-ring spectrum $A$. 

      One would start by defining Moore spectra as cofibres of self-maps of $A$, and by comparing the homology localization they define in $D(\Mod{A})$ with the derived formal completion over $A$. Using the same procedure we used above, it should be possible to prove the following. For a connective $A$-module $X$ and for a $(-1)$-connective homotopy commutative $A$-algebra $E$, the Bousfield localization of $X$ at $E$ in $D(\Mod{A})$ is identified canonically with the derived formal completion of $X$ at the Moore spectrum associated to $\pi_0 (E)$. 

    \end{rmk}


\section{Examples and applications} 
  \label{sec:examples}
  In thi section we include some easy corollaries and applications of the results in section \ref{sec:moore_spectra} and \ref{sec:localizations}.
  \subsection{Algebraic Cobordism} 
    \label{ssub:subs_ex_of_loc_mgl}
    Let $MGL$ be the spectrum representing Voevodsky's algebraic cobordism. In \cite{panin_pimenov_rondigs:a_universality_theorem_for_voevodsky_MGL} the authors construct it as a commutative monoid in $\spt_T^{\Sigma}(K)$. Recall that $MGL$ is $(-1)$-connective and that the unit map $e:\bb S \ra MGL$ induces on $\underline \pi_0$ the quotient map $\cal K^{MW}_\ast \ra \cal K^{MW}_\ast/(\eta)=\cal K^{M}_\ast$. Indeed, as shown in \cite[Theorem 3.8]{hoyois:from_algebraic_cobordism_to_motivic_cohomology}, $e$ factors through the projection $\bb S \ra C(\eta)$, and the induced map
    \[e': C(\eta) \ra MGL\] 
    is an isomorphism on $\underline \pi_k $ for every $k\leq 0$. In particular $MGL$ satisfies assumption \ref{sub:assumption_A1}. We conclude that for every connective spectrum $X$, the $MGL$-localization map is canonically identified with the $\eta$-completion map $\chi_\eta(X): X \ra X^\wedge _\eta$.

  \subsection{Motivic Cohomology}
    \label{ssub:ex_of_loc_HZ}
    Let $H\bb Z$ be the spectrum representing Voevodsky's Motivic cohomology with integral coefficients.
    Recall that we have a category of motives $\cal {DM}(K,\bb Z)$ which is related to $\cal {SH}(K)$ by an adjuction
    \begin{equation}
    \label{eqn:add-forget_transfers}
    \bf L\bb Z_{tr}:\cal {SH}(K) \rightleftarrows \cal{DM}(K,\bb Z):\bf R u_{tr}.
    \end{equation}
    The functor $\bf L\bb Z_{tr}(-)$ is induced on spectra by the functor that associates with every smooth $K$-variety $V$ its presheaf with transfers $\bb Z_{tr}(V)$ on $\sm{K}$. The functor $\bf R u_{tr}$ simply forgets transfers. In \cite[Section 4]{hoyois:from_algebraic_cobordism_to_motivic_cohomology} the author argues that the functor $\bf Ru_{tr}$ is lax symmetric monoidal. Hence $H\bb Z$, which is defined as $\bf R u_{tr}(\bb Z)$ is a homotopy commutative ring spectrum. For the record, the discussion in \cite[Remark 4.7]{hoyois:from_algebraic_cobordism_to_motivic_cohomology} shows how to construct a model of $H\bb Z$ as a commutative monoid in motivic symmetric spectra; we will not need this.

    With this definition, the fact that $H\bb Z$ is $(-1)$-connected follows from the representability of motivic cohomology in $\cal DM(K,\bb Z)$ \cite[Proposition 14.16]{mazza_voevodsky_weibel;lecture_notes_on_motivic_cohomology} combined with \cite[Theorem 3.6]{mazza_voevodsky_weibel;lecture_notes_on_motivic_cohomology}.

    Thanks to the theorem of Suslin-Nesterenko and Totaro, see for instance \cite[Theorem 5.1]{mazza_voevodsky_weibel;lecture_notes_on_motivic_cohomology}, we have an isomorphism
    \[\lambda: \bigoplus_ {n\in \bb Z} K^M_n(K) \overset{\simeq}{\ra} \bigoplus_{n \in \bb Z} H^n(\spec K,\bb Z(n)).\]
    Furthermore, if we denote by $e$ the unit of the ring spectrum $H\bb Z$, we can form the following diagram
    \[
    \xymatrix{
    \pi_0(\bb S)_\ast \ar[r]^e &  \pi_0(H\bb Z)_\ast\\
    K^{MW}_\ast \ar[u]^{\simeq} \ar[r] &  K^{MW}_\ast/(\eta)=K^M_\ast \ar[u]^\lambda_{\simeq}, \\
    }
    \]
    and a direct check on the generators of $K^{MW}_\ast$ shows that the square commutes.
    Thanks to Theorem \ref{thm:red_to_moore_spt}, we deduce that the $H\bb Z$-localization of a connective spectrum $X$ is identified with the $\eta$-completion map $X \ra X^\wedge _\eta$.

    In a similar fashion let $E=H\bb Z/\ell$ be the spectrum representing motivic cohomology with modulo $\ell$ coefficients. The same considerations allow us to conclude that the $H\bb Z/\ell$-localization map $\lambda_{H\bb Z/\ell}(X): X \ra X_{H\bb Z/\ell}$ of a connective spectrum $X$ is identified with the formal completion map $\chi_{\ell,\eta}(X): X \ra X^\wedge_{\ell,\eta}$.

  \subsection{Slice completion}
    We recall that, for a spectrum $X$, a cell presentation of $X$ is a tower of maps 
    \[ 0=X_0 \subseteq X_1 \subseteq X_2 \subseteq \cdots\subseteq X_n \subseteq \cdots X\]
    satisfying the following properties:
    \begin{enumerate}
    \item there is a weak equivalence $\hocolim_n X_n \ra X$;
    \item for every $n$ there exists a set of indices $I_n$, and for each $i \in I_n$ a map $f_i: \bb S^{p_i+q_i\alpha} \ra X_{n}$ such that the map $X_n\subseteq X_{n+1}$ fits in a homotopy push-out square
          \begin{equation}
            \xymatrix{
              X_n \ar[r] & X_{n+1}\\
              \vee_{i \in I_n} \bb S ^{p_i+q_i\alpha} \ar[r] \ar[u]^{\vee {f_i}} & \ast. \ar[u]\\
            }
          \end{equation}
    \end{enumerate}
    In this situation, the map $\bb S^{p_i+q_i\alpha} \ra \ast$ is called \emph{a cell of dimension $p_i+1$ and weight $q_i$}; the map $f_i: \bb S^{p_i+q_i\alpha} \ra X_n$ is called the \emph{attaching map} of the cell.
    A cell presentation is called \emph{of finite type} if there exists an integer $k$ such that none of the cells has dimension smaller than $k$ and, in any case, for every integer $p$ there is only a finite set of cells whose dimension is equal to $p$.
    Observe that this definition works not only for spectra, but in the category of modules over any $\bf A_\infty$ ring spectrum $E$, by substituting, in the above definition, the sphere spectrum $\bb S$ with $E$.
    We recall that for every spectrum $X \in \cal{SH}(K)$, Voevodsky defined the so called \emph{slice tower}
    \[X \ra \cdots \ra f^qX \ra \cdots \ra f^1X \ra f^0X \ra f^{-1}X \ra \cdots  \] 
    whose homotopy inverse limit is called \emph{slice completion} $X_{sc}$ of $X$. We refer to \cite[Section 3]{RoSpOst} for a more detailed description of the tower. For convenience of the reader we state Theorem 3.5 of \cite{RoSpOst}.

    \begin{thm}
      Let $K$ be a field of characteristic $0$. Suppose that $X$ is a spectrum having a cell presentation of finite type. Then in the commutative square
      \[
      \xymatrix{
         X \ar[rr]_{\chi_\eta(X)} \ar[d]^{\sigma(X)} & & X^\wedge_\eta \ar[d]^{\sigma(X)^\wedge _\eta}\\
         X_{sc} \ar[rr]^{\chi_\eta(X_{sc})} & &  (X_{sc})^{\wedge}_\eta\\
      }
      \]
      the maps $\chi_\eta(X_{sc})$ and $\sigma(X)^\wedge_\eta$ are isomorphisms in $\cal{SH}(K)$. In particular there is a natural isomorphism $X_{sc} \simeq X^\wedge _\eta$ in $\cal {SH}(K,\Lambda)$ under which the slice completion map $\sigma(X)$ and the $\eta$-completion map $\chi_{\eta}(X)$ are identified.
    \end{thm}

    By combining \ref{ssub:ex_of_loc_HZ} with the previous result, we deduce that for a cell spectrum $X$ of finite type, $H\bb Z\wedge X=0$ if and only if $X_{sc}=0$; more generally the slice completion of a cell spectrum of finite type $X$ only depends on the $H\bb Z$-localization $X_{H\bb Z}$.

  \subsection{Motives of spectra}
    \label{ssub:ex_on_motives_of_spectra}

    Let $K$ be a field which is not formally real and $X$ be a connective spectrum. Set $Y=X[1/2]$ and assume that $Y\wedge H\bb Z=0$. Then by the assumption on $X$, Theorem \ref{thm:red_to_moore_spt}, and Proposition \ref{prop:localization_at_mod_x_moore_spectrum} we have that
    \[0=Y_{H\bb Z}=Y_{M(\eta)}=Y^\wedge_\eta.\]
    However Lemma \ref{lem:Xonehalf_is_eta_complete} implies that $\chi_{\eta}(Y): Y \ra Y^\wedge_\eta$ is an isomorphism in $\cal {SH}(K)$ and hence $X[1/2]=Y=0$. Let us now assume that $K$ has finite $2$-cohomological dimension. By running the same argument and using \ref{lemma:compact_is_eta_complete} instead of \ref{lem:Xonehalf_is_eta_complete} we deduce that, if $X$ is a dualizable object with $H\bb Z\wedge X=0$, then $X=0$ in $\cal{SH}(K)$. 

    Recall that the adjunction \eqref{eqn:add-forget_transfers} factors as
    \[
     \xymatrix{
     \cal{SH}(K) \ar@<.5ex>[rr]^{\bf L\bb Z_{tr}} \ar@<.5ex>[dr]^{\wedge H\bb Z} & &  \cal{DM}(K,\bb Z) \ar@<.5ex>[ll]^{\bf R u_{tr}} \ar@<.5ex>[dl]^{\Psi}\\
     & D(\Mod{H\bb Z}) \ar@<.5ex>[ul]^{\bf R u} \ar@<.5ex>[ru]^{\Phi}  & \\
     }
    \]
    where $\bf L\bb Z_{tr}(-)\simeq \Phi(H\bb Z\wedge -)$ and $\bf R u_{tr}\simeq \bf R u(\Psi(-))$.
    If $K$ has characteristic $0$, then \cite[Theorem 1]{MR2435654} implies that $(\Phi,\Psi)$ is a pair of adjoint equivalences. The previous observations prove the following statement. 
    \begin{cor}
      Let $K$ be a field of characteristic $0$ which is not formally real. Then the functor $\bf L \bb Z_{tr}$ is conservative on connective $\bb S[1/2]$-local spectra. If in addition $cd_2(K)<\infty$ then $\bf L \bb Z_{tr}$ is conservative on dualizable objects.
    \end{cor}

    \begin{rmk}
      \label{rmk:what_to_do_for_cons_in_char_p}
      Our method can be actually pushed to show that, when $K$ is a field of characteristic $p>0$, the functor $\bf L \bb Z_{tr}$ is conservative on connective objects, up to power-of-$2$-torsion. Indeed one could endow $D(\Mod{H\bb Z})$ with the homotopy $t$-structure and try to modify Corollary $4.1.5$ in \cite{MR3038720} to prove that $\Psi$ induces an equivalence between the hearts of the $t$-structures. Hence Corollary 4 of \cite{Bacmann_Conservativity} would imply that $\Phi$ is conservative on connective objects. It would be then immediate to deduce the conservativity of $\bf L \bb Z_{tr}$ from the conservativity of $H\bb Z \wedge - $, at least up to possible $2$-torsion. Luckily all these conservativity statements have already been proven integrally by Bachmann in \cite{Bacmann_Conservativity}. Furthermore, with $\bb Z[1/2]$-coefficients, such statements follow pretty trivially from \cite{MR3038720}, as we now explain.
    \end{rmk}
    
    \begin{rmk}
      \label{rmk:our_cons_res_vs_Bachmann}
      In \cite[Theorem 1]{Bacmann_Conservativity} Bachmann proves that over a perfect field $K$ of exponential characteristic $p\not = 2$ and with $cd_2(K)<\infty$, the functor $\bf L \bb Z_{tr}$ is conservative on effective and connective spectra where $p$ acts invertibly. Let us concentrate on the case of a base field of characteristic $0$. Even though our assumptions are strictly weaker than those of \cite[Theorem 1]{Bacmann_Conservativity}, if we content ourselves with a $\bb S[1/2]$-local conservativity statement, Bachmann's strategy only needs our assumptions in order to work. We will briefly sketch his strategy. The first ingredient Bachmann needs is that the Witt ring $W(K)$ is of exponent $2$: it follows from \cite[Ch. 2, Theorem 7.1]{MR770063} that this is verified if and only if $-1$ is a sum of squares in $K$. This gives the relation $2^n\eta=0 \in K^{MW}_\ast(K)$ which we used above. The second fact he needs is that $\cal{DM}(K,\bb Z)$ is endowed with an analogue of the homotopy $t$-structure. He further needs that, with respect to the homotopy $t$-structure on $\cal {SH}(K)$ and $\cal{DM}(K,\bb Z)$, the functors $\bf L \bb Z_{tr},\bf R u_{tr}$ are respectively right-exact and exact and that the right adjoint $\bf R u_{tr}$ induces as equivalence between $\cal{DM}(K,\bb Z)^{\heartsuit}$ and the category of $\cal K_\ast^{M}$-modules in $\Pi_\ast(K)$. In particular if $K$ is not formally real, the forgetful functor $\Mod{\cal K^M_\ast} {\ra} \Pi_\ast(K)$ induces an equivalence after inverting $2$ in the coefficients. Most of these facts follow in fact from \cite{MR3038720}. The final ingredient needed is that $\bf L \bb Z_{tr}(-)$ is conservative on connective objects which is proven in \cite[Corollorary 4]{Bacmann_Conservativity}. From this discussion  it follows that our conservativity statement does not add any interesting case to those already covered in \cite{Bacmann_Conservativity}.
    \end{rmk}

  \subsection{Milnor-Witt Motivic Cohomology}
    \label{ex:Loc-at_MWMC}
    Let now $K$ be an infinite perfect field of exponential characteristic $p\not = 2$. In this situation, the formalism we recalled in \ref{ssub:ex_of_loc_HZ} has a quadratic analogue. We have a category of Chow-Witt motives $\widetilde{\cal {DM}}(K,\bb Z)$ with a pair of adjoint functors $(\bf L \tilde {\bb Z}_{tr}, \tilde u_{tr})$ which are the stabilizations of the functors which respectively add and forget generalized  transfers. The adjunction factors as
    \begin{equation}
     \xymatrix{
     \cal{SH}(K) \ar@<.5ex>[rr]^{\bf L\tilde{\bb Z}_{tr}} \ar@<.5ex>[dr]^{\tilde{H\bb Z}\wedge} & &  \widetilde{\cal{DM}}(K,\bb Z) \ar@<.5ex>[ll]^{\bf R \tilde u_{tr}} \ar@<.5ex>[dl]^{\tilde\Psi}\\
     & D(\Mod{\tilde{H\bb Z}}) \ar@<.5ex>[ul]^{\bf R \tilde u} \ar@<.5ex>[ru]^{\tilde\Phi}  & \\
     }
    \end{equation}
    and when $K$ is of characteristic $0$ \cite{Bac-Fas}, or after inverting $p$ in the coefficients, the functors $(\tilde\Phi,\tilde\Psi)$ are inverse equivalences of categories.
    Arguments analogous to those in \ref{ssub:ex_of_loc_HZ} allow to conclude that $\tilde H \bb Z$ is $(-1)$-connective. 
    However, this time, the unit map $\bb S \ra \tilde {H} \bb Z$ induces the identity on $\underline\pi_0$. It follows by Theorem \ref{thm:red_to_moore_spt} that for every connective spectrum $X$, the $\tilde {H} \bb Z$-localization is the identity. Hence the functor
    \[ \tilde{H\bb Z} \wedge -: \cal{SH}(K) \ra D(\Mod{\tilde{H\bb Z}})\] 
    associating with $X$ its free $\tilde{H}\bb Z$-module $\tilde {H}\bb Z\wedge X$ is conservative. 
    Thus if the base field is of characteristic zero, the functor $\bf L\tilde {\bb Z}_{tr}$ is conservative on connective spectra.

    Similarly we can use Theorem \ref{thm:red_to_moore_spt} to deduce that the $\tilde {H} \bb Z/\ell$-localization of connective spectra is identified with the $\ell$-adic completion.

    \begin{rmk}
      \label{rmk:cwconscharp}
      Also in the case of Chow-Witt motives, it should be possible to deduce the conservativity in another way along the lines of \ref{rmk:what_to_do_for_cons_in_char_p}. We can indeed endow the category of Chow-Witt motives $\widetilde{\cal{DM}}(K)$ with an analogue of the homotopy $t$-structure (see \cite{AnaNesh:framed_and_mw_transf}). Next we believe that one could mimic \cite{MR3038720} to show that $\bf R \tilde {u}_{tr}$ induces an equivalence between the hearts of the homotopy $t$-structure on $\widetilde{\cal{DM}}(K)$ and on $\cal {SH}(K)$. In particular this would imply that $\bf L \tilde{\bb Z}_{tr}$ is conservative on connective spectra even in positive characteristic $p\not =2$. 
    \end{rmk}


  \subsection{Connective Johnson-Wilson theories}
    Let $BP_{(\ell)}$ be the spectrum representing $\ell$-local Brown-Peterson cohomology. $BP_{(\ell)}$ is constructed for instance in \cite{MR1876212}. Let instead $BP^\wedge_{\ell}$ be the $\ell$-complete version of the spectrum, i.e. $BP^\wedge_{\ell}:=(BP_{(\ell)})_{M(\ell)}$. Let us denote simply by $BP$ either one of $BP_{(\ell)}$ or $BP^\wedge_{\ell}$. In the $\ell$-local situation we can construct a tower of homotopy commutative ring spectra 
    \begin{equation}
      \label{eqn:BP_trunc_tower}
      BP \ra \cdots \ra BP\bc{n} \ra \cdots BP\bc 1 \ra BP\bc 0
    \end{equation}
    where, by definition,
    \[BP\bc n := BP/(v_{n+1},v_{n+2},\cdots) .\]
    In the $\ell$-complete situation we just localize the whole tower at the Moore spectrum $M(\ell)$. The spectrum $BP\bc n$ is called the $\ell$-local (resp. $\ell$-complete) $n$-truncated Johnson-Wilson spectrum. 
    Since $H\bb Z/\ell \wedge BP<n>$ splits as a wedge of shifts of $H\bb Z/\ell$, we deduce that
    \[\bc{H\bb Z/\ell \wedge BP<n>}=\bc{H\bb Z/\ell}.\]
    Combining with \ref{ssub:ex_of_loc_HZ} we deduce that the $H\bb Z/\ell \wedge BP<n>$-localization of a connective spectrum $X$ is $X^\wedge_{\eta,\ell}$.

\section{The $E$-based motivic Adams-Novikov spectral sequence} 
  \label{sec:the_E_based_MANSS}
  In this section we briefly recall the construction of the motivic Adams-Novikov spectral sequence based on a generalized homology theory represented by a ring spectrum $E$. The actual construction takes place in \ref{sub:construction_of_the_spectral_sequence} by means of the standard Adams tower. In \ref{sub:some_gereral_remark_on_convergence} we review various kinds of convergence that one can expect from the spectral sequence.

  \subsection{Construction of the spectral sequence} 
    \label{sub:construction_of_the_spectral_sequence}
    Let $E$ be a homotopy commutative ring spectrum in $\cal{SH}(K)$. We start by considering fibre sequence
    \begin{equation}
      \label{eqn:Ebar-S-E}
      \overline E \overset{\bar e}{\ra} \bb S \overset{e}{\ra}E
    \end{equation}
     where $\overline E:=\hofib(e:\bb S\ra E)$. We set the notation $\overline {E}^1=\overline E$ and $\overline E^0=\bb S$. By induction, assuming we have already defined $\overline {E}^n$, we obtain a new fibre sequence by applying $-\wedge \overline E^n$ to the fibre sequence \eqref{eqn:Ebar-S-E}: we get the fibre sequence
    \begin{equation}
      \label{eqn:fund_fib_seq_E^n+1_E^n_EE^n}
    \xymatrix{
      \overline E \wedge \overline E^n \ar[r]^{\bar e\wedge id} &  \bb S\wedge \overline E^n\ar[r]^{e \wedge id} &  E\wedge \overline E^{n}\\
    }.
    \end{equation}
    We set $\overline E^{n+1}:=\overline E \wedge \overline E^n$ and as well $W_n:=E\wedge \overline E^{n}$. Furthermore we name the maps $\bar e^n:=\bar e\wedge id_{\overline E^n}$ and $e^n:=e\wedge id_{\overline E^n}$. In this way we have produced a tower $\{\overline E^n\}_{n \in \bb N}$ over $\bb S$ fitting in the following diagram
    \begin{equation}
      \label{eqn:Adams_tower_over_S}
      \xymatrix{
        \bb S =\overline E^0 \ar[d]_e & \overline E^1 \ar[d]_{e^2} \ar[l]_{\bar e^1} & \overline E^2 \ar[d]_{e^3} \ar[l]_{\bar e^2} & \cdots \ar[l]_{\bar e^3} \\
        E = W_0 \ar@{-->}[ur] & E\wedge \overline E^1 = W_1 \ar@{-->}[ur] & E\wedge \overline E^2=W_2 \ar@{-->}[ur] & \cdots\\
      }
    \end{equation}
    where each dashed arrow is pictured to remind that the triangle it bounds is a fibre sequence.
    Given any spectrum $X$ we can smash every part of the previous construction with $X$ and get a tower $\{X\wedge \overline E^n\}_{n\in \bb N}$ over $X$ and actually a whole digram similar to \eqref{eqn:Adams_tower_over_S}. 

    \begin{defin}
      \label{defin:ANSS_first_const}
      Let $X$ and $Y$ be spectra in $\cal{SH}(K)$. We call $E$-based motivic Adams-Novikov spectral sequence the spectral sequence associated to the exact couple 
      \begin{equation}
      \label{eqn:ANSS_first_Ex_couple}
        \xymatrix{
        [\Sigma^{\bullet}Y,X\wedge \overline E^\bullet] \ar[rr]^{j} & & [\Sigma^{\bullet}Y,X\wedge \overline E^\bullet] \ar[dl]^k\\
         & [\Sigma^{\bullet}Y, X\wedge E\wedge \overline E^\bullet]. \ar[ul]^{i} & \\
        }
      \end{equation}
      Here the map
      \[j: [\Sigma^{p}Y,X\wedge \overline E^q] \ra [\Sigma^{p}Y,X\wedge \overline E^{q-1}]\]
      is the natural map induced by $\overline e^q$ and has bi-degree $(0,-1)$;
      \[k: [\Sigma^{p}Y,X\wedge \overline E^q]  \ra [\Sigma^{p}Y, X\wedge E\wedge \overline E^q]\]
      is the natural map induced by $e^q$ and has bi-degree $(0,0)$. Finally the map
      \[i: [\Sigma^{p}Y, X\wedge E\wedge \overline E^q] \ra [\Sigma^{p+1}Y,X\wedge \overline E^{q+1}] \]
      is the map induced by the dashed map in \eqref{eqn:Adams_tower_over_S} and has bi-degree $(-1,1)$.
    \end{defin}

    \begin{rmk}
      \label{rmk:why_alt_constr_of_SS}
      If on one hand this is perfectly fine as a construction of the spectral sequence, on the other hand we prefer to have a construction where the source of the exact couple is a tower under $X$ rather than a tower over $X$. Our motivation is twofold: on one side we will need some notation later on which is implicitly related to the alternative construction that we are about to give; on the other side what follows will give a more precise insight on what the target of the spectral sequence should be. In particular the following alternative construction fits perfectly in the formalism of conditional convergence developed by Boardman in \cite{MR1718076} and which is implicitly present in the early works of Bousfield and Adams that inspired this work.
    \end{rmk}
    
    With this aim in mind we proceed with the construction. We start by introducing, for every integer $n\geq 0$, the cofibre sequence  
    \begin{equation}
      \label{eqn:fund_fib_seq_E^n_to_E_n}
      \overline  E^n \ra \bb S \ra \overline E_{n-1}
    \end{equation}
    where the map on the left is the composition $\bar e^1 \circ \cdots \circ \bar e^n$. In particular we get that $\overline E_{-1}=0$ and that $\overline E_0=E$. Moreover we can choose maps $f_n: \overline E_n \ra \overline E_{n-1}$ making the following diagram commutative
    \begin{equation}
      \label{eqn:fund_comp_of_fib_seq_E^n_to_En}
      \xymatrix{
      \overline E^n \ar[r] & \bb S \ar[r]  & \overline E_{n-1}\\
      \overline E^{n+1} \ar[u]^{\bar e^{n+1}} \ar[r] & \bb S \ar[r] \ar@{=}[u] & \overline E_n \ar[u]^{f_n}.\\
      }
    \end{equation}
    In particular we can chose an identification of $W_n=\hocofib(\bar e^{n+1})\simeq \hofib({f_n})$ in $\cal{SH}(S)$ and thus we can choose a map $l_n:W_n\ra \overline E_n$ that makes the following diagram into a fibre sequence
      \begin{equation}
        \label{eqn:fund_fib_seq_Wn_En_En-1}
        E\wedge \overline E^{n}=W_n \overset{l_n}{\ra} \overline E_n \overset{f_n}{\ra} \overline E_{n-1} \overset{\partial_n}{\ra} \Sigma^1 W_n.
      \end{equation}
     We thus get a new diagram
    \begin{equation}
      \label{eqn:Adams_tower_under_S}
      \xymatrix{
        \cdots  & W_3 \ar[d]^{l_3} &  W_2 \ar[d]^{l_2} & W_1 \ar[d]^{l_1} & W_0 \ar@{=}[d]^{l_0} & \\
        \cdots \ar[r] & \overline E_3 \ar[r]_{f_3} \ar@{-->}[ul] & \overline E_2 \ar[r]_{f_2} \ar@{-->}[ul] & \overline E_1 \ar[r]_{f_1} \ar@{-->}[ul]  & \overline E_0 \ar[r]_{f_0} \ar@{-->}[ul]& 0 \ar@{-->}[ul].\\
      }
    \end{equation}
    where each (small) solid triangle is a fibre sequence and the maps $f_n$ form a tower under $\bb S$.

    As we did above, given any spectrum $X$ we can build similar diagrams by applying $X\wedge -$ to \eqref{eqn:Adams_tower_under_S}: we obtain
    \begin{equation}
      \label{eqn:Adams_tower_under_X}
      \xymatrix{
        \cdots  & X \wedge W_3 \ar[d]^{l_3} & X\wedge  W_2 \ar[d]^{l_2} & X \wedge W_1 \ar[d]^{l_1} & X\wedge W_0 \ar@{=}[d]^{l_0} & \\
        \cdots \ar[r] & X\wedge \overline E_3 \ar[r]_{f_3} \ar@{-->}[ul] & X \wedge \overline E_2 \ar[r]_{f_2} \ar@{-->}[ul]_{\partial_3} &X \wedge \overline E_1 \ar[r]_{f_1} \ar@{-->}[ul]_{\partial_2}  & X\wedge \overline E_0 \ar[r]_{f_0} \ar@{-->}[ul]_{\partial_1}& 0 \ar@{-->}[ul]_{\partial_0=0}.\\
      }
    \end{equation}
    and the solid triangles bounded by the triple of maps $(l_n,f_n,\partial_n)$ form fibre sequences.
    \begin{defin}
    \label{defin:std_adams_tower}
      The tower under $X$
      \begin{equation}
        \label{eqn:std_adams_tower}
        \cdots \ra X\wedge\overline E_n \ra \cdots\ra X \wedge \overline E_1 \ra X \wedge \overline E_0 \ra \ast
      \end{equation}
      is called the \emph{Standard $E$-Adams tower}. The \emph{$E$-nilpotent completion} of $X$, which we denote by $X^\wedge_E$ is thus defined as
      \[ X^\wedge_E:=\holim_n X\wedge\overline E_n .\]
      The natural map $\alpha_E(X): X \ra X^\wedge _E$ is called the $E$-nilpotent completion map of $X$.
    \end{defin}

    For every spectrum $Y$ we can apply the functor $[Y,-]$ to \eqref{eqn:Adams_tower_under_X} and get an exact couple 
    \begin{equation}
      \label{eqn:ANSS_second_ex_couple}
      \xymatrix{
        [\Sigma^{\bullet}Y,X\wedge \overline E_\bullet] \ar[rr]^{j} & & [\Sigma^{\bullet}Y,X\wedge \overline E_\bullet] \ar[dl]^k\\
         & [\Sigma^{\bullet}Y, X\wedge W_\bullet]. \ar[ul]^{i} & \\
        }
    \end{equation}
    Here the map 
    \[j: [\Sigma^{p}Y,X\wedge \overline E_n] \ra [\Sigma^{p}Y,X\wedge \overline E_{n-1}]\]
    is the natural map induced by $f_n$ and has bi-degree $(0,-1)$; the map
    \[k: [\Sigma^{p}Y,X\wedge \overline E_n]  \ra [\Sigma^{p-1}Y, X\wedge W_{n+1}]\]
    is the natural map induced by the dashed arrow $\partial_{n+1}$ and has bi-degree $(-1,1)$. Finally the map
    \[i: [\Sigma^{p}Y, X\wedge W_n] \ra [\Sigma^{p-1}Y,X\wedge \overline E_{n}] \]
    is the map induced by $l_n$ and has bi-degree $(0,0)$.

    The spectral sequence obtained form the exact couple \eqref{eqn:ANSS_second_ex_couple} is the \emph{$E$-based motivic Adams-Novikov spectral sequence}. Note that this is an example of the general procedure described in IX.4 of \cite{MR0365573} for associating the so called \emph{Homotopy Spectral Sequence} to a tower of fibrations over a given space. In our specific example the tower we used is $\{X\wedge \overline E_n, f_n\}$. Following the indexing scheme suggested by Bousfield-Kan, the spectral sequence is the family
    \begin{equation}
      \label{eqn:motivicANSS}
      \Big \{E_r^{s,t}(Y,X;E), d_r^{s,t}: E_r^{s,t}(Y,X;E) \ra E_r^{s+r,t+r-1}(Y,X;E) \Big \}
    \end{equation}
    where, $s \in \bb N$, $t\in \bb Z$ and $r\geq 1$. For $r=1$ we have 
    \begin{equation}
      E_1^{s,t}(Y,X;E):=[\Sigma^{t-s} Y, X\wedge W_s]
    \end{equation}
    with differential
    \begin{equation}
      d_1^{s,t}(Y,X;E)  :E_1^{s,t}(Y,X;E) \ra E_1^{s+1,t}(Y,X;E)
    \end{equation}
    defined by the rule $x\mapsto \partial_{s+1}\circ l_s \circ x$.
    For $r>1$ we set
    \begin{equation}
    \label{eqn:derived_ex_coup_D}
    [\Sigma^{p} Y, X\wedge \overline E_n]^{(r)}:=\rm{Im}([\Sigma^{p} Y, X\wedge \overline E_{n+r}] \overset{f_{n}\circ\cdots \circ f_{n+r}}{\rrra} [\Sigma^{p} Y, X\wedge \overline E_n] )
    \end{equation}
    and 
    \begin{equation}
    [\Sigma^{p} Y, X\wedge W_n]^{(r)}:=\frac{\rm{Ker} \Big ( [\Sigma^{p}Y,X\wedge W_n] \overset{l_n}{\ra} \medslant{ [\Sigma^{p} Y, X\wedge \overline E_n]}{ [\Sigma^{p} Y, X\wedge \overline E_n]^{(r)} }\Big )}{ \partial_n \Big ( \rm{Ker} \big ( [\Sigma^{p+1}Y,X\wedge \overline E_{n-1}] \ra [\Sigma^{p+1}Y,X\wedge \overline E_{n-r-1}]  \big ) \Big )  }.
    \end{equation}
    Similarly, for $s\in \bb N$ and $t\in \bb Z$, we set
    \[E_r^{s,t}(Y,X;E):=[\Sigma^{t-s} Y, X\wedge W_s]^{(r-1)}\]
    with differential
    \[ d_r^{s,t}: E_r^{s,t}(Y,X;E) \ra E_r^{s+r,t+r-1}(Y,X;E)\]
    defined by the composition
    \[ \partial_{s+r}\circ l_s:  [\Sigma^{t-s} Y, X\wedge W_s]^{(r-1)} \ra  [\Sigma^{t-s} Y, X\wedge \overline E_s]^{(r-1)} \ra [\Sigma^{t-s-1} Y, X\wedge W_{s+r}]^{(r-1)}. \]


  \subsection{Some general remark on convergence} 
      \label{sub:some_gereral_remark_on_convergence}
      Here we recall some further notation regarding the spectral sequence in order to deal with some convergence issue.

      \subsection{}
      Since we will be considering only the tower constructed in \eqref{eqn:Adams_tower_under_X} we will simplify the notation and replace $X\wedge \overline E_s$ by $X_s$ and $\holim_s X_s$ by $X_\infty$. Moreover since no role, at least for the moment, is going to be played by the weight $u$ or the spectrum $Y$, we hide them from the notation and leave them implicit. As a result we will write, for instance, 
      \[
         \pi_{k}(X_\infty)
      \] instead of 
      \[
        [\Sigma_{S^1}^{k} Y,\holim_s X\wedge \overline E_s ].
      \]
      For the same reasons we rename $F_n:=X\wedge W_n$.
      
      \subsection{}
      First of all note that we have an analogue of Milnor's sequence:
      \begin{equation}
      \label{eqn:Milnor_seq_conv}
       0\ra \varprojlim_s {}^1 \pi_{k+1}(X_s) \ra \pi_k(X_\infty) \ra \varprojlim _s \pi_k(X_s) \ra 0
      \end{equation}
      for every integers $k \in \bb Z$. Now, both the groups on the centre and on the right hand side of \eqref{eqn:Milnor_seq_conv} inherit a canonical filtration 
      \begin{align}
      F^i\pi_{k}(X_\infty)&:=\ker(\pi_{k}(X_\infty) \ra \pi_{k}( X_{i-1}))\\
      F^i\varprojlim_s\pi_k(X_s)&:=\ker( \varprojlim_s\pi_k(X_s)\ra \pi_{k}( X_{i-1})).
      \end{align}
      The filtration on $\varprojlim_s\pi_k(X_s)$ is by design complete and separated; moreover by construction of the Adams tower $X_{-1}=0$ and so the filtration is exhaustive as well. We wish to have an analogous understanding of the filtration on $\pi_k X_\infty$: at the moment we only know that it is exhaustive for the same reason as above.

      \subsection{}
      On one hand we can start by setting $Q_s\pi_k X:=\im(\pi_k(X_\infty) \ra \pi_k(X_{s})) \subseteq \pi_k(X_s)$: we then gain exact sequences for every $k$ and $s$ as follows
      \begin{equation}
        \label{eqn:F^sQ_s_sequence}
        0 \ra F^{s+1}\pi_{k}(X_\infty) \ra \pi_{k}(X_\infty) \ra Q_s(\pi_kX) \ra 0;
      \end{equation}
      moreover the composition with $f_s: X_s \ra X_{s-1}$ gives, by restriction, well defined maps $Q_s \pi_k X\ra Q_{s-1} \pi_k X$. Now it is immediate to see that $\varprojlim_s Q_s \pi_k X = \varprojlim_s \pi_k(X_s)$ so that, in view of the Milnor sequence \eqref{eqn:Milnor_seq_conv} we must have  
      \begin{gather}
        \label{eqn:F^s_lim^1}
        \cap_s F^s\pi_k (X_\infty)=\varprojlim_s{}^1 \pi_{k+1}(X_s);  \\
        \varprojlim_s{}^1 F^s\pi_k(X_\infty)= 0 = \varprojlim_s{}^1 Q_s\pi_kX.
      \end{gather}
      In particular we conclude that the filtration on $\pi_k X_\infty$ is complete and exhaustive, while the separatedness is obstructed by the left term in the Milnor sequence \eqref{eqn:Milnor_seq_conv}.
      
      \subsection{}
      On the other hand we can take the sequences \eqref{eqn:F^sQ_s_sequence} and by the five-lemma get short exact sequences
      \begin{equation}
       \label{eqn:small_e_infty_term}
        0 \ra F^{s}\pi_{k}(X_\infty)  /F^{s+1}\pi_{k}(X_\infty) \ra Q_s \pi_k(X) \ra Q_{s-1} \pi_k(X) \ra 0
      \end{equation}
      for varying $k\in \bb Z$ and $s\in \bb N$. The same argument with $\varprojlim_s \pi_k X_s$ gives an identification of the graded pieces
      \[F^{s}\pi_{k}(X_\infty)  /F^{s+1}\pi_{k}(X_\infty) = F^{s}\varprojlim_s \pi_k X_s /F^{s+1}\varprojlim_s \pi_k X_s\]
       The group $F^{s}\pi_{k}(X_\infty)  /F^{s+1}\pi_{k}(X_\infty)$ is denoted by $e_\infty^{s,s+k}$ in IX, 5.3 of \cite{MR0365573} and is called the \emph{small $E$-infinity term}. It will turn out in the following paragraph that we actually have $e_\infty^{s,s+k}\subseteq E_\infty^{s,s+k}$.

      \subsection{}In order to better understand $e_\infty$ terms we compare them and the $Q_s$'s with something slightly bigger. Let us thus define $\pi_kX_s^{(r)}: = \im (\pi_k(X_{s+r}) \ra \pi_k(X_s))$ (this is indeed part of the $r$-th derived exact couple deduced in \eqref{eqn:derived_ex_coup_D}) and observe that the inclusion of $Q_s\pi_k X \subseteq \pi_k(X_s)$ factors through $$Q_s\pi_k X \subseteq \varprojlim_r \pi_kX_s^{(r)} \subseteq \pi_k(X_s).$$ In particular after taking inverse limits with respect to $s$ we get that 
      \[ \varprojlim_s Q_s\pi_kX = \varprojlim_s \varprojlim_r \pi_kX_s^{(r)}=\varprojlim_s\pi_k (X_s).\]
      Furthermore, as happened above, the tower maps $f_s$ induce, by restriction, surjective maps $f_s: \pi_k X_s^{(r-1)} \ra \pi_k X_{s-1}^{(r)}$. When $r>s\geq0$ we can identify the kernel of such a map with $E_r^{s, s+k}$. Indeed it follows by the very construction of the spectral sequence \eqref{eqn:motivicANSS} that the kernel of $f_s: \pi_k X_s \ra \pi_k X_{s-1}$ coincides with the quotient of $\pi_k F_s$ by the image of $\partial_s: \pi_{k+1}X_{s-1} \ra \pi_k F_s$. The image of such a map is actually identified with $B_r^{s,s+k}$ when $r>s$. In other words every element $x$ in $\ker(\pi_k X_s \ra \pi_k X_{s-1})$ lifts uniquely to an element of $\pi_kF_s/B_r^{s,s+k}$. Moreover the further condition that $x$ belongs to $\pi_kX_s^{(r)}$ is equivalent to the condition that $x$ is in the kernel of the differential in each of the first $r$ pages, i.e. that it belongs to $Z_r^{s,s+k}$. In conclusion we get the following exact sequence
      \begin{equation}
        \label{eqn:bigE_infty_term}
         0 \ra E_r^{s, s+k} \ra \pi_k X_s^{(r-1)} \overset{f_n}{\ra} \pi_k X_{s-1}^{(r)} \ra 0.
      \end{equation}

      \subsection{}
      We can now combine our last two observations in the following diagram obtained from \eqref{eqn:small_e_infty_term} and \eqref{eqn:bigE_infty_term}:
      \begin{equation}
      \xymatrix{
        0 \ar[r] & F^{s}\pi_{k}(X_\infty)  /F^{s+1}\pi_{k}(X_\infty) \ar[r] \ar@{^{(}->}[d] & Q_s \pi_k(X) \ar[r] \ar@{^{(}->}[d] & Q_{s-1} \pi_k(X) \ar[r] \ar@{^{(}->}[d] &  0 & & & \\
        0 \ar[r] & E_\infty^{s, s+k} \ar[r] & \varprojlim_r \pi_k X_s^{(r-1)} \ar[r] & \varprojlim_r \pi_k X_{s-1}^{(r)}  \ar`r[d]`[l]`^d[lll] `[dll] [dll]  \\ 
          & \varprojlim_{r}^1 E_r^{s, s+k}  \ar[r] & \varprojlim_r^1 \pi_k X_s^{(r-1)} \ar[r] & \varprojlim_r^1 \pi_k X_{s-1}^{(r)} \ar[r] & 0. \\
        }
      \end{equation}

      In the following part we share the terminology conventions used by Boardmann \cite[Sec. 5]{MR1718076} and Bousfield \cite{MR0365573,MR551009}.

      \subsection{}
      In first place we observe that by the very design of the Adams tower, the spectral sequence is conditionally convergent to $\varprojlim_s \pi_k X_s$. Moreover, since by construction the filtration on $\varprojlim_s \pi_k(X_s)$ is always complete separated and exhaustive, the only condition missing in order to have strong convergence to $\varprojlim_s \pi_k(X_s)$ is that the graded group $\rm{Gr}^s(\varprojlim_s \pi_k(X_s))$ associated with the filtration introduced above coincides with the $E_\infty$ term. Since by construction this graded group is the same as $\rm{Gr}^s( \pi_k(X_\infty))$ we deduce that strong convergence to $\varprojlim_s \pi_k(X_s)$ is equivalent to the condition that $e_\infty^{s,s+k}=E_\infty^{s,s+k}$ for every $s\in \bb N$.

      \subsection{}
      We need a last piece of information which is hidden in the filtration $\pi_k X_s^{(r)} \subseteq \pi_k X_s$ obtained by letting $r$ vary in $\bb N$. With a bit of work, which is carried out in Theorem 3.4 of \cite{MR1718076}, one can deduce the existence of a short exact sequence
      \begin{equation}
        \label{eqn:MLsequence}
        0 \ra \varprojlim_s {}^1 \big( \varprojlim_r  \pi_k X_s ^{(r)} \big) \ra \varprojlim_s {}^1 (\pi_k X_s) \ra \varprojlim_s \big( \varprojlim_r{} ^1 \pi_k X_s^{(r)} \big)\ra 0
      \end{equation}
      where the first non trivial map from the right is induced by the inclusions $\pi_k X_s^{(r)}\subseteq \pi_kX_s$.

      The discussion we have set up naturally leads to the following criterion.
      \begin{cor}
       \label{cor:RE_inf_vs_strong_convergence}
       For every fixed integer $k$ the following are equivalent:
        \begin{enumerate}
          \item the $E$-based motivic Adams-Novikov spectral sequence is strongly convergent (in degree $k$) to $\varprojlim_s \pi_k(X_s)$ and $\varprojlim_s^1 \pi_k X_s=0$;
          \item for every $s\in \bb N$ we have that  $e_\infty^{s,s+k}=E_\infty^{s,s+k}$ and, in addition, $\varprojlim_s^1 \pi_k X_s=0$;
          \item for every $s\in \bb N$ we have that $\varprojlim_{r}^1 E_r^{s, s+k}=0$.
        \end{enumerate}
      \end{cor}

      \subsection{}
      Similarly we can apply the terminology above to address convergence to $\pi_k X_\infty$. As we have seen, the filtration on this group is always exhaustive, and in \eqref{eqn:F^s_lim^1} we have observed that it is complete but might fail to be separated. The failure of separatedness is measured by $\varprojlim_s^1 \pi_{k+1}X_s$.
       \begin{defin}
         \label{defin:complete_convergece}
         The $E$-based motivic Adams-Novikov spectral sequence is said to converge completely in degree $k$ if the following conditions are satisfied:
         \begin{enumerate}
           \item for every $s\in \bb N$ we have that  $e_\infty^{s,s+k}=E_\infty^{s,s+k}$;
           \item $\varprojlim_s^1 \pi_{k+1} X_s=0$.
         \end{enumerate}
       \end{defin}
       \begin{defin}
       \label{defin:Mittag-Leffler}
       The $E$-based motivic Adams-Novikov spectral sequence is said to be \emph{Mittag-Leffler in degree $k$} if for every $s \in \bb N$ we have that the tower $\{E_{r\geq s}^{s,s+k}\}_r$ is Mittag-Leffler, i.e. if for every $s \in \bb N$ there is an integer $r(s)$ such that $s\leq r(s) <\infty$ and $E_{r(s)}^{s,s+k}=E_\infty^{s,s+k}$. We say that the spectral sequence is \emph{strongly Mittag-Leffler in degree $k$} if it is Mittag-Leffler in degree $k$ and, in addition, there exists an $s_0$ such that for every $s\geq s_0$ the term $E_\infty^{s,s+k}$ vanishes.  
     \end{defin}
     \begin{rmk}
       \label{rmk:Mittag-Leffler_and_strongML}
       The condition of being Mittag-Leffler in degree $k$ is actually equivalent to the requirement that the tower $\{\pi_k(X_s)\}_s$ is Mittag-Leffler: this follows directly by diagram \eqref{eqn:bigE_infty_term} and the fact that $X_{-1}=0$. Being Mittag-Leffler in degree $k$ is also equivalent to the fact that, for every $s \in \bb N$, each term $E_r^{s,s+k}$ has only a finite number of non-zero differentials leaving it.
     \end{rmk}

\section{Nilpotent Resolutions} 
  \label{sec:nilpotent_resolutions}
  In this section we introduce $E$-nilpotent and $\underline \pi_0 E$-nilpotent resolutions of a spectrum $X$: this takes place in \ref{sub:E-nilpotent-resol} and \ref{sub:R_nilpotent_resolutions} respectively. We use these constructions to produce in a more efficient way the $E$-nilpotnent completion $X^\wedge _E$. As an application we obtain a convergence statement (see Theorem \ref{thm:convergence_mod_stuff}) for the $E$-based motivic Adams-Novikov spectral sequence.
  
  \subsection{$E$-Nilpotent resolutions} 
    \label{sub:E-nilpotent-resol}
    \begin{defin}
    \label{defin:E-nilpotence}
    Let $E$ be a homotopy commutative ring spectrum in $\cal{SH}(K)$. We define the subcategory of \emph{$E$-nilpotent} spectra as the smallest full subcategory $\mathrm{Nilp}(E) \subseteq \cal{SH}(K)$ satisfying the following properties:
    \begin{enumerate}
      \item $E \in \mathrm{Nilp}(E)$;
      \item $\mathrm{Nilp}(E)$ is an ideal, i.e. given any $X\in \cal{SH}(K)$ and any $F \in \mathrm{Nilp}(E)$      then $X\wedge F \in \mathrm{Nilp}(E)$;
      \item $\mathrm{Nilp}(E)$ has the 2-out-of-3 property on fibre sequences, i.e. given a fibre sequence
        \[X\ra Y\ra Z\]
        in $\cal{SH}(K)$ where any two of the three objects $X,Y,Z$ are in $\mathrm{Nilp}(E)$, then the third is in $\mathrm{Nilp}(E)$ as well;
      \item $\mathrm{Nilp}(E)$ is closed under retracts;
    \end{enumerate}
     \end{defin}

    \begin{rmk}
      \label{rmk:mods_over_E-nilp_ringsp_are_E-nilp}
      If $R$ is a homotopy associative ring spectrum and $M$ is an $R$-module then the action map $R\wedge M \ra M$ in split by the unit. So if $R$ is in $\mathrm{Nilp}(E)$ then $R\wedge M$ is also $E$-nilpotent and hence $M$ is $E$-nilpotent too. 
    \end{rmk}

    \begin{lemma}
      \label{lemma:E-nilp_implies_E-loc}
        If $E$ is a homotopy commutative ring spectrum and $X$ is any $E$-nilpotent spectrum, then $X$ is $E$-local  
    \end{lemma}

    \begin{proof}
      The proof goes exactly as in \cite[Lemma 3.8]{MR551009}. We filter $\rm{Nilp}(E)$ by inductively constructed subcategories $C_i$. $C_0$ is defined as the full subcategory of $\cal{SH}(K)$ whose objects are the spectra isomorphic to $E\wedge X$ for some $X \in \cal{SH}(K)$. If $i\geq 1$ we set $C_i$ to be the full subcategory of $\cal{SH}(K)$ of those spectra that are isomorphic to a retract of an object in $C_{i-1}$ or an extension of objects in $C_{i-1}$. It is formal to check that the union of the $C_i$'s coincides with $\rm{Nilp}(E)$.
      Indeed, thanks to \ref{rmk:mods_over_E-nilp_ringsp_are_E-nilp} we have that $C_0\subset \rm{Nilp}(E)$, and since $E$-nilpotent objects are closed under retractions and extensions, we get by induction that each of the $C_n$ is contained in $\rm{Nilp}(E)$. Now the $C_n$'s form an increasing sequence of subcategories of $\rm{Nilp}(E)$ and we need to check that their union, which we denote by $C$, is the whole $\rm{Nilp}(E)$. 
      However this is clear: by construction $C$ satisfies all the four axioms of \ref{defin:E-nilpotence} so we must have $C\supseteq \rm{Nilp}(E)$, and so $\rm{Nilp}(E) = C$. For proving the $E$-locality: $E$-modules are $E$-local, so $C_0\subseteq \rm{Loc}(E)$; since $E$-local objects are closed under extensions and retractions $C_i\subseteq \rm{Loc}(E)$, and hence $\rm{Nilp}(E) = \cup_i C_i\subseteq \rm{Loc}(E)$.
     \end{proof}
    \begin{rmk}
      \label{rmk:havin_a_nilp_loc}
      Let us fix a homotopy commutative ring spectrum $E$. It follows from the previous discussion  that $\rm{Nilp}(E) \subseteq \rm{Loc}(E)$ but the converse inclusion does not in general hold, even in the classical setting of the topological stable category $\cal{SH}$. There are however cases in which $\rm{Loc}(E)=\rm{Nilp}(E)$.
    \end{rmk}
    \begin{defin}
      Let $E$ be a homotopy associative ring spectrum. A spectrum $X$ is called $E$-pre-nilpotent if $X_E$ is $E$-nilpotent.
    \end{defin}
    \begin{prop}
      \label{prop:pre-nilp_sphere}
      Let $E$ be a commutative ring spectrum. Then the following are equivalent
      \begin{enumerate}[label=P.\arabic*]
        \item \label{prenilp_a} $\bb S$ is $E$-pre-nilpotent, i.e. $\bb S_E$ is $E$-nilpotent;
        \item \label{prenilp_b} For every spectrum $X$, $\bb S_E\wedge X$ is $E$-nilpotent;
        \item \label{prenilp_c} Every spectrum $X$ is $E$-pre-nilpotent, i.e. $X_E$ is $E$-nilpotent for every $X$;
        \item \label{prenilp_d} $\rm{Nilp}(E)=\rm{Loc}(E)$.
      \end{enumerate}
      Moreover the following are equivalent:
      \begin{enumerate}[label=S.\arabic*]
        \item \label{prenilp_e} For every spectrum $X$, the map $\lambda_E(\bb S)\wedge id_X: X \ra \bb S_E\wedge X$ is the $E$-localization of $X$;
        \item \label{prenilp_f} The multiplication map of the $E$-local sphere $\bb S_E \wedge \bb S_E \ra\bb S_E$ is an isomorphism and the natural inequality $\bc E \leq \bc {\bb S_E}$ is an equality.
        \end{enumerate} 
      In addition the statement \ref{prenilp_b} implies \ref{prenilp_e}.
      Furthermore if $E$ has a multiplication map $E\wedge E\ra E$ which is an isomorphism, then the unit $e:\bb S \ra E$ coincides with the localization map $\lambda_E(\bb S)$ and condition \ref{prenilp_a} holds.
    \end{prop}
    \begin{proof}
       We start by observing that the localization map $\lambda_E(X)$ factors as
      \begin{equation}
      \label{eqn:E-loc_vs_smash_with_loc_sphere}
      \xymatrix{
      X \ar[r]^{\lambda_E(X)} \ar[d]_{\lambda_E(\bb S)\wedge id_X} & X_E\\
      \bb S_E\wedge X. \ar[ur]_{\tilde\lambda(X)} & \\
      }
      \end{equation}
      Since all the maps in the diagram are $E$-equivalences, $\tilde \lambda$ is an isomorphism if and only if $\bb S_E\wedge X$ is $E$-local. 

      $E$-nilpotent objects are closed under smashing with arbitrary spectra, so that $\bb S_E$ is $E$-nilpotent if and only if for every spectrum $X$, $\bb S_E\wedge X$ is $E$-nilpotent too (\ref{prenilp_a}$\Leftrightarrow$\ref{prenilp_b}). In view of \eqref{eqn:E-loc_vs_smash_with_loc_sphere}, if $\bb S_E\wedge X$ is $E$-nilpotent then it is $E$-local and hence $\tilde\lambda$ is an isomorphism (\ref{prenilp_b}$\Rightarrow$\ref{prenilp_c}). Clearly \ref{prenilp_c} $\Rightarrow$ \ref{prenilp_a} and \ref{prenilp_c} $\Leftrightarrow$ \ref{prenilp_d}. 

      Using \eqref{eqn:E-loc_vs_smash_with_loc_sphere} we immediately deduce that \ref{prenilp_b}$\Rightarrow$\ref{prenilp_e}.

      By applying \ref{prenilp_e} to $X=\bb S_E$, in view of \eqref{eqn:E-loc_vs_smash_with_loc_sphere}, we deduce that the multiplication map $\tilde\lambda(\bb S_E)$ of $\bb S_E$ is an isomorphism. On the other hand, by smashing the fundamental fibre sequence
      \[{}_E\bb S \ra \bb S \ra \bb S_E\]
      with a spectrum $X$, we deduce that if $X$ is $\bb S_E$-acyclic, then it is also $E$-acyclic. This means that $\bc E\leq \bc{\bb S_E}$. The reverse equality is immediate from \ref{prenilp_e}, so \ref{prenilp_e} implies \ref{prenilp_f}. Assume now \ref{prenilp_f}. Since the multiplication of $\bb S_E$ is an isomorphism, for every spectrum $X$ the map $\lambda_{E}(\bb S)\wedge id_X: X \ra \bb S_E\wedge X$ is the $\bb S_E$-localization of $X$; however $\bc{\bb S_E}=\bc E$ so that \ref{prenilp_f} implies \ref{prenilp_e}.

      If $E$ is a ring spectrum with the property that the multiplication $E\wedge E \ra E$ is an isomorphism, then the unit $e: \bb S\ra E$ is an $E$-equivalence; since $E$ is $E$-nilpotent, and thus $E$-local, we conclude.
    \end{proof}

    \begin{defin}
    \label{defin:smashing_localization}
      We say that a spectrum $E$ induces a smashing localization if the map $\tilde \lambda_E(X): \bb S_E \wedge X \ra X_E $ of \ref{eqn:E-loc_vs_smash_with_loc_sphere} is an isomorphism in $\cal{SH}(K)$.
    \end{defin}

    \begin{ex}
      Let $E=\bb S[S^{-1}]$ where $S\subseteq K^{MW}_\ast(K)$ is a possibly infinite subset. Then Proposition \ref{prop:localization_at_S_inverted_moore_spectrum} implies that the multiplication map $\bb S[S^{-1}]\wedge \bb S[S^{-1}] \ra \bb S[S^{-1}]$ is an isomorphism. Hence $\bb S[S^{-1}]$ defines a smashing localization.

      Let $E=H\bb Q$ be the spectrum representing Voevodsky's motivic cohomology with rational coefficients. Combining Proposition 14.1.6 with Corollary 16.1.7 of \cite{cisdeg_triang_mixed}, we deduce that the multiplication map of $H \bb Q$ is an isomorphism. In particular the localization at $H\bb Q$ is smashing. 
    \end{ex}
    \subsubsection{}
      \label{subs:Relation between localization and nilpotent completion}
      We wish to point out how Definition \ref{defin:std_adams_tower} and Lemma \ref{lemma:E-nilp_implies_E-loc} imply that, for every spectrum $X$, the $E$-nilpotent completion $X^{\wedge}_E$ is $E$-local. Indeed
      \[X^\wedge_E=\holim_n \big(  \cdots \overset{f_{n+1}}{\ra} X\wedge \overline E_n \overset{f_n}{\ra} \cdots \ra  X\wedge \overline E_0 \ra 0 \big), \]
      and each of the maps in the tower sits in the fibre sequence
        \[
           E \wedge \overline E^n \wedge X \ra \overline E_{n}\wedge X \overset{f_n}{\ra} \overline E_{n-1}\wedge X
        \]
      that we have deduced from \eqref{eqn:fund_fib_seq_Wn_En_En-1}. As a consequence, by induction, each of the terms in the tower is $E$-nilpotent, hence $E$-local, and thus $X^\wedge_E$ is $E$-local too. In particular the natural map $\alpha_E(X) : X \ra X^{\wedge}_E$ factors as 
      \begin{equation}
      \label{eqn:alphabetagamma}
        \xymatrix{
        X \ar[rr]^{\alpha_E(X)} \ar[dr]_{\lambda_E(X)} & & X^{\wedge}_E \\
        & X_E \ar[ur]_{\beta_E(X)} & \\
        }.
      \end{equation}
      It follows that $\alpha_E(X)$ is an $E$-equivalence if and only if the induced map $\beta_E(X)$ is an equivalence.

      \subsubsection{}
      We wish to point out another fact. On one hand, if $X\ra Y$ is an $E$-equivalence, then it induces an isomorphism of the Standard $E$-Adams Towers (\ref{defin:std_adams_tower}) associated to $X$ and $Y$, so that the natural map induced on homotopy inverse limits $X^\wedge_E \ra Y^\wedge_E$ is an equivalence.
      On the other hand the composition of $\alpha_E(X)$ with the projection to the $0$-th term of the tower
      \[  X \ra X^\wedge_E \ra X\wedge \overline E_0=X\wedge E\] is identified with $id_X\wedge e$, where $e: \bb S\ra E$ is the unit of the homotopy ring spectrum $E$. Thus, after smashing with $E$, 
      the map $\alpha_E(X)\wedge E: X\wedge E \ra X^\wedge_E\wedge E$ has a splitting which is functorial in $X$. So if $f: X\ra Y$ is a map inducing an equivalence on $E$-nilpotent completions $X^\wedge _E \ra Y^\wedge _E$, then $f$ is an $E$-equivalence. We conclude that $\alpha_E(X)$ is an $E$-equivalence if and only if the induced map 
      \[\alpha_E(X)^\wedge_E: X^\wedge_E\ra (X^\wedge _E)^\wedge_E\] is an equivalence.    

    \begin{defin}
      \label{defin:E-nilp_resol}
      For a spectrum $X\in \cal{SH}(K)$, an \emph{$E$-nilpotent resolution of $X$} is a tower of spectra under $X$
      \[X \ra \cdots  \ra X_n \ra X_{n-1} \ra \cdots \ra X_0 \]
      satisfying the following two properties:
      \begin{enumerate}
        \item $X_n \in \mathrm{Nilp}(E)$ for every $n \in \bb N$;
        \item for any $Y\in \mathrm{Nilp}(E)$ the canonical map $\colim_n [X_n,Y] \ra [X,Y]$ is an isomorphism.
      \end{enumerate}
    \end{defin}

    \begin{prop}
      \label{prop:E-nilp_resol_are_unique_and_com_the_nilp_comp}
      Let $X \in \cal{SH}(K)$ be any spectrum. Then:
      \begin{enumerate}
        \item \label{ENR1} the standard Adams tower $\{ \overline{E}_n\wedge X\}$ is an $E$-nilpotent resolution of $X$;
        \item \label{ENR2} there exists a unique $E$-nilpotent resolution of $X$ up to unique isomorphism in the category of pro-towers under $X$;
        \item \label{ENR3} given any $E$-nilpotent resolution $\{ X_n \}_n$ of $X$ we have that $\holim_n X_n \simeq X^{\wedge}_{E}$.
      \end{enumerate}
    \end{prop}

    \begin{proof}
     We start with \eqref{ENR1}. As we already observed in Remark \ref{subs:Relation between localization and nilpotent completion}, the terms $\overline E_n\wedge X$ of the tower are $E$-nilpotent. Let $Y$ be any spectrum. By smashing the fibre sequence \eqref{eqn:fund_fib_seq_E^n_to_E_n} with $X$ and applying $[-,Y]$ we get a long exact sequence
     \begin{equation}
       \cdots \ra [\overline E_n\wedge X,Y] \ra [X,Y] \ra [\overline E^{n+1}\wedge X,Y] \ra \cdots
     \end{equation}
     Moreover, for varying $n$, the map of fibre sequences \eqref{eqn:fund_comp_of_fib_seq_E^n_to_En} induces a map of long exact sequences
     \begin{equation}
      \xymatrix{
       \cdots \ar[r] & [\overline E_n\wedge X,Y] \ar[r] & [X,Y] \ar[r] & [\overline E^{n+1}\wedge X,Y] \ar[r] & \cdots\\
       \cdots \ar[r] & [\overline E_{n-1}\wedge X,Y] \ar[r]\ar[u]^{f_n^\ast} & [X,Y] \ar[r] \ar@{=}[u] & [\overline E^{n}\wedge X,Y] \ar[r]\ar[u]^{\bar e_{n+1}^\ast} & \cdots .\\
      }
     \end{equation}
      We deduce that we only need to show that $\varinjlim _n [\overline E^n\wedge X,Y]=0$ for every $E$-nilpotent spectrum $Y$. We will proceed by induction on the family of subcategories $C_i$ that we used in the proof of \ref{lemma:E-nilp_implies_E-loc}. Assume thus that $Y\in C_0$, i.e. that $Y\simeq E\wedge Z$ for some spectrum $Z$: we will show that the transition maps in the colimit vanish, hence the colimit vanishes too. For this we look at the fibre sequence \eqref{eqn:fund_fib_seq_E^n+1_E^n_EE^n}: it gives a long exact sequence 
     \begin{equation}
       \cdots \ra [E\wedge \overline E^n\wedge X,E\wedge Z] \overset{(e\wedge \rm{id}_X)^\ast}{\rra} [\overline E^n\wedge X, E\wedge Z] \overset{(\overline e^{n+1}\wedge \rm{id}_X)^\ast}{\rra}[\overline E^{n+1} \wedge X, E\wedge Z] \ra \cdots
     \end{equation}
     where the map $(e\wedge \rm{id_X})^\ast$ is surjective since $E\wedge Z$ is an $E$-module. The transition map in the direct limit is thus $0$. Now observe that the property that $\varinjlim _n [\overline E^n\wedge X,Y]=0$ is stable in the $Y$ variable under retracts and extensions: this implies that if every $Y \in C_{i-1}$ satisfies this property then also every $Y \in C_i$ does as well. Since the union of the $C_i$'s exhaust $\rm{Nilp}(E)$ the first point is done.

     \eqref{ENR2}. Existence is \eqref{ENR1}. Let thus $X_\bullet$ and $W_\bullet$ be $E$-nilpotent resolutions of $X$. For every pair of non-negative integers $k,r$ let $p^r_k: X_r \ra X_k$ and $q^r_k: W_r \ra W_k$ be the transition maps in the towers $\{X_\bullet\}$ and $\{W_\bullet\}$ respectively. Similarly let $p^\infty_k: X \ra X_k$ and $q^\infty_k: X \ra W_k$ be the respective projection maps. We show the existence of a unique isomorphism of pro-objects $\lambda: \{X_\bullet\} \ra \{W_\bullet\}$ such that the diagram of pro-objects
     \[
     \xymatrix{
      & X \ar[dl]_{p^\infty_\bullet} \ar[dr]^{q^\infty_\bullet} & \\
     \{X_\bullet\} \ar[rr]^{\lambda} & & \{W_\bullet\}\\
     }
     \]
     commutes. Since $W_k$ is $E$-nilpotent and $\{X_\bullet\}$ is an $E$-nilpotent resolution of $X$, we have \[\Hom_{}(\{X_\bullet\},\{W_\bullet\})=\varprojlim_k[X,W_k]\] and we chose $\lambda$ to be the map of pro-objects corresponding to the compatible system $\{q^\infty_k\}$ of the projections. More explicitly, we can chose a re-indexing function $k\mapsto n_k$ and maps $\lambda_{n_k}: X_{n_k} \ra W_k$ in $\cal{SH}(K)$ such that:
     \begin{enumerate}
       \item \label{lif} for every $k \in \bb N$, $\lambda_{n_k}$ is a lifting to $X_{n_k}$ of the projection $q^\infty_k: X \ra W_k$, i.e. \[ \lambda_{n_k}\circ p^\infty_{n_k}=q^\infty_k; \]
       \item \label{com} for every $k \in \bb N$, the diagram 
               \[
               \xymatrix@=18pt{
               X \ar@{=}[dd] \ar[rrr]^(.3){p^\infty_{n_k}} \ar[drr]_{p^\infty_{n_{k+1}}} & & & X_{n_k}\ar[dd]^{\lambda_{n_k}}\\
                & & X_{n_{k+1}} \ar[ur]_{p^{n_{k+1}}_{n_{k}}} \ar[dd]^(.35){\lambda_{n_{k+1}}} & \\
               X \ar '[rr]^(.5){q^\infty_{k}} [rrr] \ar[drr]_{q^\infty_{k+1}} & & & W_k\\
               & & W_{k+1} \ar[ur]_{q^{k+1}_{k}}& \\
               }
               \] 
               commutes.
     \end{enumerate}

     Note that, by construction, $\lambda_{n_k}$ is the unique map having as source some level of the tower $\{X_\bullet\}$ and satisfying \eqref{lif}, up to increasing the level $n_k$. Hence $\{\lambda_{n_\bullet}\}=\lambda$ as maps of pro-objects. Reversing the role of $\{X_\bullet\}$ and $\{W_\bullet\}$, for every $k \in \bb N$, one constructs maps $\mu_{m_k}: W_{m_k} \ra X_{n_k}$ in $\cal{SH}(S)$ satisfying 
     \[ \mu_{m_k}\circ q^\infty_{m_k}=p^\infty_{n_k}, \]
     and such maps are unique up to increasing the level $m_k$. As a consequence both the maps $\mu_{m_k}\circ \lambda_{n_k},q^{m_k}_k : W_{m_k} \ra W_k$ are lifts to $W_{m_k}$ of the projection $q^\infty_{k}:X\ra W_k$.  But then, up to raising $m_k$, they must coincide, since $W_k$ is $E$-nilpotent resolution of $X$. By running the same argument, with the roles of $\{X_\bullet\}$ and $\{W_\bullet\}$ interchanged, one concludes. Point \eqref{ENR3} follows by combining \eqref{ENR1} and \eqref{ENR2} with the observation that, thanks to the Milnor sequence, the homotopy inverse limit of a tower of spectra only depends on the pro-isomorphism class of the tower.

    \end{proof}

  \subsection{$\cal R$-nilpotent resolutions} 
    \label{sub:R_nilpotent_resolutions}
    \begin{defin}
    \label{defin:core_of_MW_algebras}
    A \emph{Milnor-Witt algebra} is a commutative monoid in the symmetric monoidal category $\Pi_\ast(K)$ of homotopy modules over the field $K$.
    If $\cal R$ is a Milnor-Witt algebra, we define the \emph{core} of $\cal R$ to be 
      \begin{equation}
        \xymatrix{
          c\cal R:= \rm{Eq}\Big( \cal R \ar@<.3ex>[rr]^{x\mapsto x\otimes 1} \ar@<-.3ex>[rr]_{x\mapsto 1\otimes x} & & \cal R \otimes \cal R \Big)\\
        }.
      \end{equation}
    We call $\cal R$ \emph{solid} if $c\cal R=\cal R$.
    
    \end{defin}

    \begin{lemma}
      \label{lemma:props_of_solid_rings}
      Let $\cal R$ be a Milnor-Witt algebras with the property that the multiplication map of $\mu_{\cal R} : \cal R\otimes \cal R \ra \cal R$ of $\cal R$ is an isomorphism. Then:
      \begin{enumerate}
        \item \label{PSR1} $\cal R$ is solid.
        \item \label{PSR2} For every $\cal R$-module $\cal M$ in homotopy modules, the action map $\cal R\otimes \cal M\ra \cal M$ is an isomorphism. In particular a homotopy module $\cal M$ has at most one  $\cal R$-module structure.
        \item \label{PSR3} Every map of homotopy modules $\phi: \cal M\ra \cal N$ where $\cal M$ and $\cal N$ are $\cal R$-modules is $\cal R$-linear. In particular the category of $\cal R$-modules is a full subcategory of $\Pi_\ast(K)$.
      \end{enumerate}
    \end{lemma}
    \begin{proof}
      \eqref{PSR1} is clear. For \eqref{PSR2} observe that, given a $\cal R$-module $\cal M$ in homotopy modules, we have a coequalizer diagram defining the monoidal product on the category of $\cal R$-modules
      \[
      \xymatrix{
      \cal R \otimes \cal R \otimes \cal M \ar@<.4ex>[r]^{a} \ar@<-.4ex>[r]_{\mu_{\cal R}} & \cal R\otimes \cal M \ar[r]^q & \cal R\otimes_{\cal R}\cal M. \\
      }
      \]
      Moreover the action map $a: \cal R \otimes \cal M \ra \cal M$ induces an isomorphism $\bar a: \cal R \otimes_{\cal R} \cal M\ra \cal M$. The map $\cal R\otimes \cal M \ra \cal R\otimes\cal R\otimes \cal M$ defined by $r\otimes m \mapsto r\otimes 1 \otimes m$ is an inverse of both $\mu_{\cal R}$ and $a$, so that $a=\mu_{\cal R}$ and $q$ is as isomorphism. In particular every $\cal R$-module $\cal M$ is isomorphic to the free $\cal R$-module on the homotopy module $\cal M$. Point \eqref{PSR3} follows by combining the free-forget adjunction and point \eqref{PSR2}.
    \end{proof}

    \begin{rmk}
      \label{rmk:class_of_solid_MW_algebras}
      It seems reasonable to expect that, as it happens in Algebraic Topology, $\cal R$ is solid if and only if the multiplication $\cal R \otimes \cal R  \ra \cal R $ is an isomorphism. Similarly we expect $c\cal R $ to be solid, and hence that $c\cal R$ is maximal solid sub-algebra of $\cal R$. Ideally one could hope to prove an analogue of  \cite[Proposition 3.5]{MR0308107}, giving thus a classification of solid $\cal K^{MW}$-algebras. We believe that a useful tool for this task would be the classification of homogeneous prime ideals of $\cal K^{MW}$ which is carried out in \cite{MR3503978}.
    \end{rmk}

    \begin{rmk}
      \label{rmk:AssA1_implies_solidity_of_pi_0}
      Assume $E$ is a ring spectrum satisfying the assumption \ref{sub:assumption_A1}. Then the multiplication map of $\cal R=\underline \pi_0(E)$ is an isomorphism. In particular by \ref{lemma:props_of_solid_rings} $\underline \pi_0E$ is a solid Milnor-Witt algebra. For showing this, consider first the special case that $J=\emptyset$ so that $\underline \pi_0E\simeq \cal K^{MW}/(f_1,\dots,f_r)$ for some $f_i \in K^{MW}_{q_i}(K)$. Denote by $\cal I$ the unramified ideal generated by the $f_i$'s. Then we have an exact sequence 
      \begin{equation}
        \bigoplus_{i \in I} \cal K^{MW}_{\ast-q_i} \overset{{(f_1,\dots,f_r)}}{\rrra} \cal K^{MW}_\ast \ra \cal K^{MW}/\cal I\ra 0,
      \end{equation}
      and since $(\cal K^{MW}/\cal I)$ is a $\cal K^{MW}$-module and the tensor product of homotopy modules is right exact, the commutative diagram
      \begin{equation}
        \label{eqn:ass_implies_solidity}
         \xymatrix{
         \bigoplus_{i \in I} \cal K^{MW}_{\ast} \otimes \cal K^{MW}/\cal I  \ar[rr]^{(f_1,\dots,f_r)=0} &  & \cal K^{MW}_\ast \otimes \cal K^{MW}/\cal I \ar[r]\ar[d]^{\simeq} & \cal K^{MW}/\cal I\otimes \cal K^{MW}/\cal I \ar[r]\ar@{-->}[dl]^{\mu_{\cal R}} & 0\\
         & & \cal K^{MW}/\cal I & \\
         }
       \end{equation} 
       allows to conclude. Now if $J\not = \emptyset$, we can deduce the statement from \eqref{eqn:ass_implies_solidity} since filtered direct limits are exact.
    \end{rmk}

    \begin{defin}
      \label{defin:S-nilpotent_hom_module_and_spectra}
        Let $\cal R$ be a solid Milnor-Witt algebra. We say that a homotopy module $\cal M$ is \emph{$\cal R$-nilpotent} if it has a finite filtration $\cal M = \cal M_0 \supseteq \cal M_1 \supseteq \cdots \supseteq \cal M_r$ such that $\cal M_i/\cal M_{i+1}$ has an $\cal R$-module structure for every $i$. A spectrum $X$ is called \emph{$\cal R$-nilpotent} if for each $k\in \bb Z$ the homotopy module $\underline\pi_k(X)$ is $\cal R$-nilpotent and for all but a finite set of $k \in \bb Z$, $\underline \pi_k X=0$. We denote the full subcategory of $\cal R$-nilpotent spectra by $\mathrm{Nilp}(\cal R)$.
    \end{defin}

    \begin{defin}
      \label{defin:S-nilp_resol}
      Let $X$ be an object in $\cal{SH}(K)$ and $\cal R$ be a solid Milnor-Witt algebra. An \emph{$\cal R$-nilpotent resolution of $X$} is a tower of spectra under $X$
      \[X \ra \cdots  \ra X_n \ra X_{n-1} \ra \cdots \ra X_0 \]
      satisfying the following two properties:
      \begin{enumerate}
        \item $X_n \in \mathrm{Nilp}(\cal R)$ for every $n \in \bb N$;
        \item for every $Y\in \mathrm{Nilp}(\cal R)$ the canonical map $\colim_n [X_n,Y] \ra [X,Y]$ is an isomorphism.
      \end{enumerate}
    \end{defin}
  
     We would like to prove an analogue of Proposition \ref{prop:E-nilp_resol_are_unique_and_com_the_nilp_comp} for $\cal R_\ast$-nilpotent resolutions where $\cal R_\ast=c\underline \pi_0(E)_\ast$ and $E$ is a $(-1)$-connected homotopy commutative ring spectrum. We are not able to accomplish this, so we immediately assume, in addition, that the multiplication map $\underline \pi_0 E \otimes \underline \pi_0E \ra \underline \pi_0E$ is an isomorphism. In particular $\cal R=c\underline \pi_0E=\underline \pi_0 E$. In any case, this will be slightly more involved than \ref{prop:E-nilp_resol_are_unique_and_com_the_nilp_comp}, and will require some general preliminary lemmas.
    \begin{lemma}
      \label{lemma:S-nilp_in_sequences}
      Let $\cal R$ be a Milnor-Witt algebra whose multiplication $\mu_{\cal R}$ is an isomorphism. If 
      \[\cal A\ra \cal B\ra \cal C\ra \cal D\ra \cal E\]
      is an exact sequence of homotopy modules where $\cal A,\cal B,\cal D,\cal E$ are $\cal R$-nilpotent, then $\cal C$ is as well.
    \end{lemma}

    \begin{proof}
      By breaking up the exact sequence in shorter pieces, the statement follows by combining \ref{lemma:S-nilp_in_sequences-extensions} and  \ref{lemma:S-nilp_in_sequences-ker_and_coker}.      
    \end{proof}

    \begin{lemma}
      \label{lemma:S-nilp_in_sequences-extensions}
      If $0 \ra \cal A \ra \cal B \ra \cal C \ra 0$ is a short exact sequence of homotopy modules and both $\cal A$ and $\cal C$ are $\cal R$-nilpotent, then $\cal B$ is too.
    \end{lemma}
    \begin{proof}
      A suitable filtration on $\cal B$ can be obtained by combining the filtration on $\cal A$ and the pre-image in $\cal B$ of the filtration on $\cal C$.
    \end{proof}

    \begin{lemma}
      \label{lemma:S-nilp_in_sequences-ker_and_coker}
      If $\phi: \cal C^0\ra \cal C^1$ is a map of homotopy modules and both $\cal C^0$ and $\cal C^1$ are $\cal R$-nilpotent, then both $\cal H^0=\rm{Ker} \phi$ and $\cal H^1=\rm{Coker}\phi$ are $\cal R$-nilpotent too.
    \end{lemma}
    \begin{proof}
      Up to increasing the length of the filtrations $\cal C^0_i$ and $\cal C^1_i$ we can assume that $\phi$ respects the filtrations. As a consequence $(\cal C^\bullet,\phi)$ is a filtered complex of homotopy modules. In particular, $\phi$ induces a map on graded objects 
      \[ \phi_{i/i+1}:\cal C^0_i/\cal C^0_{i+1} \ra \cal C^1_i/\cal C^1_{i+1}.\]
      For every $i$, since both source and target of $\phi_{i/i+1}$ are $\cal R$-modules by assumption, $\phi_{i/i+1}$ is $\cal R$-linear by \ref{lemma:props_of_solid_rings}. As a consequence $\cal H^p(\phi_{i/i+1})$ is an $\cal R$-module for every $i$ and $p$. Furthermore there is a finite filtration $F^\bullet$ on $\cal H^p(\cal C^\bullet)$ defined by
      \[F^i\cal H^p:=\rm{Im}\big (\cal H^p(\cal C^\bullet_i) \ra \cal H^p(\cal C^\bullet)\big )\]

      In this setting, the spectral sequence associated to a filtered complex of homotopy modules 
      \[ E_1^{p,q}=\cal H^{p+q}\big ( \cal C^\bullet_p/ \cal C^\bullet_{p+1}\big ) \Rightarrow \cal H^{p+q}(\cal C^\bullet)\]
      converges to the filtration $F^\bullet$ on $\cal H^\bullet (\cal C^\bullet)$, i.e. for every $p$ and $q$
      \[ E_\infty^{p,q}=F^p \cal H^{p+q}/F^{p+1} \cal H^{p+q}.\]
      As observed above, every term in $E_1$ is a $\cal R$-module and so the differentials are $\cal R$-linear by \ref{lemma:props_of_solid_rings}. It follows that $E^{p,q}_\infty$ is an $\cal R$-module and thus the filtration $F^\bullet$ makes $\cal H^0(\cal C^\bullet)$ and $\cal H^1(\cal C^\bullet)$ into $\cal R$-nilpotent modules.
    \end{proof}
    
    \begin{lemma}
    \label{lemma:mul_by_f_in_a_B_nilp_A-mod}
      Let $\cal R$ be a Milnor-Witt algebra whose multiplication map $\mu_{\cal R}$ is an isomorphism and let $f \in K^{MW}_\ast(K)$. Denote by $\cal S:=\cal R/(f)$. Then the multiplication map $\mu_{\cal S}$ of $\cal S$ is an isomorphism. Furthermore for every $\cal R$-nilpotent homotopy module $\cal M$ both $\ker(f\cdot: \cal M_{\ast-p} \ra \cal M_{\ast})$ and $\coker(f \cdot: \cal M_{\ast-p} \ra \cal M_\ast)$ are $\cal S$-nilpotent.
    \end{lemma}

    \begin{proof}
      First of all we note that the multiplication $\mu_{\cal R}: \cal R\otimes \cal R \ra \cal R$ induces a map on quotients 
      \[\bar\mu_{\cal R}:=\mu_{\cal R} \otimes_{\cal R} \cal R/f:  \cal R\otimes \cal R \otimes_{\cal R} \cal R/f \ra \cal R\otimes_{\cal R}\cal R/f.\]
      Such a map factors as
      \[
      \xymatrix{
      \cal R \otimes \cal R/f \ar[r]^{\bar\mu_{\cal R}}\ar[d]^p & \cal R/f \\
      \cal R/f \otimes \cal R/f \ar[ur]_{\mu_{\cal S}}\\
      }
      \]
      where $\mu_{\cal S}$ is the multiplication map of $\cal S$ and $p$ is induced by the projection $\cal R\ra \cal R/f$. Now $p$ is an isomorphism because the tensor product of homotopy modules is right-exact and because $f$ acts by zero on $\cal R/f$. Finally, since $\mu_{\cal R}$ is an isomorphism, $\mu_{\cal S}$ is an isomorphism as well.
      Let $0=\cal M_n \subseteq \cal M_{n-1} \subseteq \cdots \subseteq \cal M_0=\cal M$ be a filtration of $\cal M$ by homotopy modules $\cal M_i$ whose associated graded pieces are $\cal R$-modules. 
      If $n=1$ then $\cal M$ is an $\cal R$-module and thus both kernel and cokernel of $(f \cdot)$ are $\cal S$-modules. If $n>1$ one can proceed by induction. Indeed, since the multiplication map $f \cdot: \cal M \ra \cal M$ is $\cal K^{MW}$-linear it respects the filtration. As a consequence, for every $p\geq 1$, $f\cdot$ induces a map of short exact sequences
      \begin{equation}
        \xymatrix{
        0 \ar[r] & \cal M_{p-1} \ar[r] \ar[d]^{f \cdot} & \cal M_{p} \ar[r] \ar[d]^{f \cdot} & \cal M_p/\cal M_{p-1} \ar[r] \ar[d]^{f \cdot} & 0\\
        0 \ar[r] & \cal M_{p-1} \ar[r] & \cal M_p \ar[r] & \cal M_p/\cal M_{p-1} \ar[r] & 0.\\
        }
      \end{equation}
      Thus by combining the snake lemma together with \ref{lemma:S-nilp_in_sequences} and the inductive assumption we conclude.
    \end{proof}

    \begin{notat}
      For the rest of the section we fix a $(-1)$-connected homotopy commutative ring spectrum $E$. We assume that the multiplication map $\mu_{\underline \pi_0 E}: \underline \pi_0 E \otimes \underline \pi_0 E \ra \underline \pi_0E$, induced on $\underline \pi_0$ by $\mu_E:E\wedge E \ra E$, is an isomorphism. We also denote by $\cal R$ the homotopy module $\underline \pi_0 E$.
    \end{notat}
    
    \begin{lemma}
      \label{lemma:S-nilp_implies_E-nilp}
      If $X$ is an $\cal R$-nilpotent spectrum then it is also an $E$-nilpotent spectrum.
    \end{lemma}
    \begin{proof}
      Since $X$ is a $\cal R$-nilpotent spectrum there is an integer $k$ such that the canonical map to its Postnikov truncation $X\ra P^k(X)$ is an isomorphism in $\cal {SH}(K)$. Now using to the fundamental fibre sequence 
      \[
        \Sigma^nH\underline \pi_n (X) \ra P^n(X) \ra P^{n-1}(X),
      \]
      the fact that $E$-nilpotent spectra have the 2-out-of-3 property in fibre sequences and the connectivity of $X$, we reduce to show that for every integer $n$ the spectra $H \underline\pi_n(X)$ are $E$-nilpotent.
      By assumption for every $n \in \bb Z$, $\underline \pi_n(X)$ is $\cal R$-nilpotent so it has a finite filtration with $\cal R$-modules as quotients. So again using the 2-out-of-3 property of $E$-nilpotent objects and that $H:\Pi_\ast(K) \ra \cal{SH}(K)$ maps short exact sequences to fibre sequences, we can reduce to proving that if $\cal M$ is an homotopy module with a structure of an $\cal R$-module then $H\cal M$ is $E$-nilpotent. In order to prove this claim we immediately observe that $H\cal M$ is an $ H\cal R$-module because $H$ is lax monoidal, so by \ref{rmk:mods_over_E-nilp_ringsp_are_E-nilp} we immediately reduce to $\cal M=\cal R=\underline \pi_0E$. Now, since $E$ is $(-1)$-connective, there is a canonical ring map $E\ra H\underline\pi_0(E)$ which makes $H\underline\pi_0(E)$ into an $E$-module and hence $H\underline\pi_0(E)$ is $E$-nilpotent.

    \end{proof}


  \subsection{Application to the Motivic Adams-Novikov spectral sequence}
  
    \begin{prop}
      \label{prop:PnEn_tower_is_an_S-nilp_resol}
      Let $X \in \cal{SH}(K)$ be a connective spectrum. Then the tower $\{P^n(\overline{E}_n\wedge X)\}_n$ is a $\cal R$-nilpotent resolution of $X$.
    \end{prop}

    \begin{proof}
      We first need to check that $P^n(\overline{E}_n\wedge X) \in \mathrm{Nilp}(\cal R)$ for every $n\in \bb Z$. By the connectivity of $X$, each of the $P^n(\overline{E}_n\wedge X)$ has only a finite number of non-trivial homotopy modules, so we only need to check that $\underline \pi_k(\overline{E}_n\wedge X)_\ast$ is $\cal R$-nilpotent for every pair of integers $k,n$. Recall that for every $n$ in $\bb N$ we have a fibre sequence of the form
      \begin{equation}
        \label{eqn:the_main_fib_seq_in_ANSS}
        X \wedge \overline{E}^n\wedge E \ra X \wedge \overline{E}_n \overset{f_n}{\ra} X \wedge \overline{E}_{n-1}   
      \end{equation}
      obtained from \eqref{eqn:fund_fib_seq_Wn_En_En-1} by smashing with $X$. By construction, for $n=0$, such a fibre sequence is the cone sequence of the identity of $X\wedge E$. In particular for every $k\in \bb Z$, $\underline \pi_k( X\wedge \overline E_0)$ is a $\underline \pi_0E$-module and thus it is $\cal R$-nilpotent. We can now proceed by induction on $n$. Note that $\underline \pi_k(X\wedge \overline E^n\wedge E)$ is a $\underline \pi_0 E$-module and thus is $\cal R$-nilpotent. This observation, once combined with the inductive assumption, allows to apply Lemma \ref{lemma:S-nilp_in_sequences} to the long exact sequence of homotopy modules associated to the fibre sequence \eqref{eqn:the_main_fib_seq_in_ANSS}. 

      As a second step we need to prove that for any $Y\in \mathrm{Nilp}(\cal R)$ the canonical map 
      \[ \colim_n [P^n(\bar{E_n}\wedge X),Y] \ra [X,Y]\] is an isomorphism. For this we first observe that we have a commutative diagram 
      \begin{equation}
        \xymatrix{
          \colim_n [P^n(\overline E_n\wedge X),Y] \ar[dr]^{\phi} \ar[r] & [X,Y] \ar[d]\\
           & \colim_n[\overline E_n\wedge X,Y],\\
        } 
      \end{equation}
      where the diagonal map $\phi$ is induced by the projection maps onto the Postnikov truncations. The fact that $Y$ is $E$-nilpotent by \ref{lemma:S-nilp_implies_E-nilp}, and that the standard Adams tower $\{\overline E_n \wedge X\}$ is a $E$-nilpotent resolution by \ref{prop:E-nilp_resol_are_unique_and_com_the_nilp_comp}, implies that the map $\colim_n[\overline E_n\wedge X,Y] \ra [X,Y]$ is an isomorphism. We are left to show that $\phi$ in an isomorphism too. Since there exists an integer $k$ for which $Y\simeq  P^k(Y)$,  every map $f: \overline E_n\wedge X \ra Y$ factors uniquely through $\tilde f: P^n(\overline E_n\wedge X) \ra Y$ at least when $n\geq k$. Passing to direct limits we deduce that $\phi$ is an isomorphism.      
    \end{proof}

    \begin{prop}
      \label{prop:S-nilp_resol_are_unique_and_com_the_nilp_comp}
      For every connective spectrum $X \in \cal{SH}(K)$ the following holds:
      \begin{enumerate}
        \item \label{SNR1} The Postnikov truncation of the standard Adams tower $\{ P_n(\overline{E_n}\wedge X)\}_n$ is an $\cal R$-nilpotent resolution of $X$;
        \item \label{SNR2} There exists a unique $\cal R$-nilpotent resolution of $X$ up to unique isomorphism in the category of pro-towers under $X$;
        \item \label{SNR3} Given any $\cal R$-nilpotent resolution $\{ X_n \}_n$ of $X$ we have that $\holim_n X_n \simeq X^{\wedge}_{E}$
      \end{enumerate}
    \end{prop}
    \begin{proof}
      Point \eqref{SNR1} is \ref{prop:PnEn_tower_is_an_S-nilp_resol}. The proof of point \eqref{SNR2} is the same as that given for the second point of \ref{prop:E-nilp_resol_are_unique_and_com_the_nilp_comp}. For point \eqref{SNR3}: we just need to combine \eqref{SNR1}, \eqref{SNR2}, the observation that, thanks to the Milnor sequence, the homotopy inverse limit of a tower of spectra only depends on the pro-weak homotopy type (see \ref{defin:pro-we} for the definition) of the tower, and finally the fact the projection to the Postnikov tower $\{\overline E_n \wedge X\} \ra \{P^n(\overline E_n\wedge X)\}$ is a pro-weak-equivalence (see \ref{subs:proj_to_post_is_pwe}).
    \end{proof}
    
    \begin{lemma}
    \label{lemma:PnMoore_tower_is_an_S-nilp_resol}
      Let $E$ be a homotopy commutative ring spectrum satisfying Assumption \ref{sub:assumption_A1} in the special case that $J=\emptyset$. Then for every connective spectrum $X$ the tower $P^n(C(f_1^n)\wedge \cdots \wedge C(f_r^n)\wedge X)$ is a $\cal R$-nilpotent resolution of $X$.
    \end{lemma}

    \begin{proof}
      We need to check that for every pair of integers $k,n$ the homotopy module $\underline \pi_k(C(f_1^n)\wedge \cdots \wedge C(f_r^n)\wedge X)$ is $\underline\pi_0 E$-nilpotent. We accomplish this by induction on $n$, the base case being that of $n=1$. Assume thus that $n=1$: we proceed by induction on $r$. If $r=0$ the homotopy modules $\underline\pi_k(X)$ are $\underline\pi_0(E)=\cal K^{MW}$-modules and in particular they are $\cal K^{MW}$-nilpotent. If $I=\{f_1, \dots, f_{a+1}\}$ we can apply Lemma \ref{lemma:mul_by_f_in_a_B_nilp_A-mod} to the homotopy modules $\cal M := \underline \pi_k (C(f_1)\wedge \cdots \wedge C(f_a)\wedge X)$ and $\cal M^{'}:=\underline \pi_{k-1}(C(f_1)\wedge \cdots \wedge C(f_a)\wedge X)$ which are $\cal R=\cal K^{MW}/(f_1,\dots,f_a)$-nilpotent by the inductive assumption. It follows that the external homotopy modules of the exact sequence
        \begin{equation}
          \label{eqn:}
          0 \ra \coker (f\cdot: \cal M \rightarrow \cal M) \ra \underline \pi_k(C(f_1)\wedge \cdots \wedge C(f_{a+1})\wedge X)) \ra \ker (f\cdot : \cal M^{'}\rightarrow \cal M^{'}) \ra 0 
        \end{equation}
       are $\cal S=\cal K^{MW}/(f_1,\dots, f_{a+1})$-nilpotent and by Lemma \ref{lemma:S-nilp_in_sequences-extensions} we conclude that the central homotopy module is $\cal S$-nilpotent too. Given now any $r$-tuple of positive integers $(n_1,\dots,n_r)$ we can show that $\underline \pi_k(C(f_1^{n_1})\wedge \cdots \wedge C(f_r^{n_r})\wedge X)$ is $\underline\pi_0 E$-nilpotent by induction, using the cofibre sequences
      \begin{equation}
        \Sigma^{p_i\alpha}C(f_i) \ra C(f_i^{n_i}) \ra C(f_i^{n_i-1}) 
      \end{equation}
      for $i=1,\dots,r$ and Lemma \ref{lemma:S-nilp_in_sequences} iteratively to reduce to the above base case.

      Let us assume now that $Y$ is a $\underline \pi_0(E)$-nilpotent spectrum: we need to check that the natural map 
      \[ \psi:\colim_n[P^n(C(f_1^n)\wedge \cdots \wedge C(f_r^n)\wedge X), Y] \ra [X,Y]\]
      is an isomorphism.
      Let us denote by $C_n$ the spectrum $C(f_1^n)\wedge \cdots \wedge C(f_r^n)$. Since $\underline \pi_k(Y)$ is $\pi_0(E)$-nilpotent then every element $f_i \in I$ acts on such a homotopy module nilpotently. It follows thus that there is an integer $N>>0$ for which each of the $f^N_i$ acts by $0$ on each the homotopy modules of $Y$. Hence every map $f:X \ra Y$ factors through $X \ra P^N(C_N\wedge X)$ giving a map $\tilde f \in [P^N(C_N\wedge X),Y]$ and hence an element in the colimit. 

      It remains to show that $\psi$ is injective. Assume $r=1$. Let $\tilde f \in [ P^n(C_n\wedge X ), Y]$ be a map with the property that $f:=\tilde f_{|X}=0$. We actually need to show the following: the composition of $\tilde f$ with a finite number of the tower maps $p_k: P^k (C_k\wedge X) \ra P^{k-1}(C_{k-1}\wedge X)$, which are induced on Postnikov towers by the $p_k$'s of  \eqref{eqn:mod_x^n_mod_x^n-1_fib_seqs}, is the zero map. By the co-connectivity of $Y$ we reduce to prove the analogous statement without Postnikov towers, so we can actually assume that $\tilde f  \in [C_n\wedge X, Y]$. We have a diagram of cofibre sequences deduced from \eqref{eqn:mod_x^n_mod_x^n-1_fib_seqs}
      \begin{equation}
      \xymatrix{
      \Sigma ^{Np_1+Nq_1\alpha}X  \ar[rr]^{f_1^{N}\cdot} \ar[d]_{f_1^{N-n}\cdot} & & X \ar[r] \ar@{=}[d]& C_N \wedge X \ar[r] \ar[d] & \Sigma ^{1+Np_1+Nq_1}X \ar[d] ^{\Sigma^1 f_1^{N-n}\cdot}\\
      \Sigma^{np_1+nq_1\alpha}X  \ar[rr]^{f_1^{n}\cdot}                          & & X \ar[r]           & C_n \wedge X \ar[r]        & \Sigma^{1+np_1+nq_1} X\\
      }
      \end{equation}

      from which we deduce that $\tilde f$ factors through some $\tilde f' \in [\Sigma ^{1+np_1+nq_1}X,Y]$. Finally for $N>>0$ the composition $\tilde f'\circ \Sigma^1 f_1^{N-n}=0$ since $Y$ is $\cal R$-nilpotent: this concludes the case $r=1$. The general case is easily obtained by applying iteratively the previous argument.

    \end{proof}

    \begin{thm}
      \label{thm:convergence_mod_stuff}
      Let $E$ be a homotopy commutative ring spectrum satisfying assumption \ref{sub:assumption_A1} in the special case that $J=\emptyset$. Then for every connective spectrum $X$ the natural map $\alpha_E(X): X \ra X^{\wedge}_E$ is an $E$-equivalence. In particular the map $\beta_E(X): X_E\ra X^{\wedge}_E$ of \eqref{eqn:alphabetagamma} is an equivalence. 
    \end{thm}

    \begin{proof}
      Let $C_n:=C(f_1^n)\wedge \cdots\wedge C(f_r^n)$. Proposition \ref{prop:PnEn_tower_is_an_S-nilp_resol} and Lemma \ref{lemma:PnMoore_tower_is_an_S-nilp_resol} imply that both the tower $\{P^n(\overline E_n\wedge X)\}_n$ and the tower $\{P^n(C_n\wedge X)\}$ are $\cal R$-nilpotent resolutions of $X$. Thanks to Proposition \ref{prop:S-nilp_resol_are_unique_and_com_the_nilp_comp} there is a unique pro-isomorphism $\varphi$ between such towers which is compatible with the projection map from the constant tower $\{X\}_n$. As a consequence we can choose an isomorphism 
      \[\psi: \holim_n P^n(\overline E_n\wedge X) \overset{\simeq}{\ra}  \holim_n P^n(C_n\wedge X)\]
      between their homotopy inverse limits which fits in the commutative square
      \[
      \xymatrix{
      \holim_n P^n(\overline E_n\wedge X) \ar[r]^{\psi}_{\simeq} \ar[d] & \holim_n P^n(C_n\wedge X) \ar[d] \\
      \{ P^n(\overline E_n\wedge X) \} \ar[r]^{\phi} & \{P^n(C_n\wedge X)\}, \\
      }
      \] 
      where the vertical maps are the natural projections. The projections to the respective Postnikov towers induce identifications of the source and target of $\psi$ with $X^{\wedge}_E$ and $X^{\wedge}_{f_1,\dots,f_r}$ respectively: this follows from the third point of \ref{prop:S-nilp_resol_are_unique_and_com_the_nilp_comp} and \ref{subs:proj_to_post_is_pwe}. In conclusion $\psi$ fits in the commutative square of pro-objects
        \begin{equation}
          \label{eqn:E-loc_equals_E-nilp:first_sq}
          \xymatrix{
          \{ X^{\wedge}_E\}_n \ar[r]^{\psi}_{\simeq} \ar[d] & \{ X^\wedge_{f_1,\dots,f_r}\}_n \ar[d]\\
          \{ P^n(\overline E_n\wedge X) \}_n \ar[r]^{\simeq}_\varphi & \{ P^n(C_n\wedge X) \}_n\\
          }
        \end{equation}
        where the vertical maps are induced by the projections to the terms of the respective towers. After smashing the towers in \eqref{eqn:E-loc_equals_E-nilp:first_sq} with $E$ we obtain a new commutative square 
        \begin{equation}
          \label{eqn:E-loc_equals_E-nilp:first_sq_E}
          \xymatrix{
          \{ E\wedge X^{\wedge}_E\}_n \ar[r] \ar[d] & \{ E\wedge X_{\bb S\underline\pi_0 E}\}_n \ar[d]\\
          \{ E\wedge P^n(\overline E_n\wedge X) \}_n \ar[r] & \{ E\wedge P^n(C_n\wedge X) \}_n.\\
          }
        \end{equation}
        of pro-objects. Here the lower horizontal map is a pro-isomorphism, since pro-isomorphisms are preserved under smashing with any spectrum by \ref{cor:Ewedge_pres_pro-isoms}. We claim that the right vertical map of \eqref{eqn:E-loc_equals_E-nilp:first_sq_E} is a pro-weak equivalence. 

        For showing this claim, consider that the vertical map on the right hand side of \eqref{eqn:E-loc_equals_E-nilp:first_sq}
        \begin{equation}
          \label{eqn:E-loc_eq_E-nilp:first_sq_rhs}
          \{X^\wedge_{f_1,\dots,f_r}\}_n \ra \{ P^n(C_n\wedge X)\}_n
        \end{equation}
        factors as the composition of two maps: the projection 
        \begin{equation}
          \label{eqn:E-loc_eq_E-nilp:first_sq_proj}
          \{X^{\wedge}_{f_1,\dots,f_r}\}_n  \ra \{ C_n\wedge X\}_n,
        \end{equation}
        and the projection to the Postnikov tower 
        \begin{equation}
          \label{eqn:E-loc_eq_E-nilp:first_sq_pt}
          \{ C_n\wedge X\}_n \ra \{ P^n(C_n\wedge X)\}_n.
        \end{equation}
        The map \eqref{eqn:E-loc_eq_E-nilp:first_sq_pt} is a pro-weak-equivalence by \ref{subs:proj_to_post_is_pwe} and stays a pro-weak-equivalence after smashing with $E$ by \ref{lemma:E_wedge_preserves_pro_equiv}. 
        Concerning \eqref{eqn:E-loc_eq_E-nilp:first_sq_proj}, we observe that by construction we have a diagram of pro-spectra,
        \begin{equation}
          \label{eqn:E-loc_eq_E-nilp:first_sq_proj_dec}
        \xymatrix{
           \{ X^{\wedge}_{f_1,\dots,f_r} \}_n\ar[r] & \{ C_n\wedge X \}_n \\
           \{ X_{\bb S_{\underline\pi_0 E}} \}_n \ar[u]^{\simeq}& \{X\}_n \ar[ul] \ar[u] \ar[l]^(0.4){\lambda_{\bb S_{\underline \pi_0E}}(X)}\\ 
           }
        \end{equation}
        where the upper triangle commutes by construction and the lower triangle commutes by \ref{prop:localization_at_mod_x_1-x_n_moore_spectrum}. Since $\lambda_{\bb S\underline \pi_0E}(X)$ is an $E$-equivalence by \ref{prop:fund_ineq_of_loc}, in order to finish the proof of the above claim we are left to show that the right vertical map of \eqref{eqn:E-loc_eq_E-nilp:first_sq_proj_dec} induces, after smashing with $E$, a pro-weak equivalence of towers.

        To accomplish this task we consider the tower $\{F^{(r)}_\bullet\}$ defined as the level-wise fibre of the tower of maps $\{X\}_\bullet \ra \{C_\bullet\wedge X\}$. We now show that $\{E\wedge F^{(r)}_n\}_n$ is pro-weakly-equivalent to $0$.
        For this, remember that in the tower $\{C_n\wedge X\}_n$ the transition maps
        \[\psi_n: C_{n+1}=C(f_1^{n+1})\wedge\cdots\wedge C(f_r^{n+1}) \ra C(f_1^{n})\wedge\cdots\wedge C(f_r^{n})=C_n \] 
        are defined as $\psi_n=p_{n+1}(f_r)\circ \cdots \circ p_{n+1}(f_1)$, where for every positive integer $n$, the maps $p_{n}(f_i)$ are those defined in \ref{const:modx_1_x_n_Moore_sp} and displayed in \eqref{eqn:mod_x^n_mod_x^n-1_fib_seqs}. 
        If $r=1$ we can chose the transition maps of $\{F^{(1)}_\bullet\}$ to be a suitable suspension of $(f_1\cdot)$ as in \eqref{eqn:mod_x^n_mod_x^n-1_fib_seqs}. Since $f_1\cdot$ acts trivially on the towers of homotopy modules $\{\underline\pi_k(E\wedge F^{(1)}_\bullet)\}$, we deduce that $\{E\wedge F^{(1)}_\bullet\}$ is pro-weakly-equivalent to $0$. If $r>1$ one can argue by induction. Indeed, using the octahedral axiom, we can find a tower of fibre sequences
        \[ \{F^{(s-1)}_\bullet\} \ra \{F^{(s)}_\bullet\} \ra \{G^{(s)}_\bullet\}  \]
        having the following properties: the tower $\{E\wedge F^{(s-1)}_\bullet\}$ is pro-weakly-equivalent to $0$ by the inductive assumption, and the tower $\{E\wedge G^{(s)}_\bullet\}$ is pro-weakly-equivalent to $0$ by the case $r=1$ treated above (see Definition \ref{defin:pro-we} ).
        Hence, thanks to \ref{cor:2-3prozero}, we conclude that \eqref{eqn:E-loc_eq_E-nilp:first_sq_proj} becomes a pro-weak-equivalence after smashing with $E$.
        
        Let us consider now the commutative square of towers:
         \begin{equation}
          \label{eqn:E-loc_equals_E-nilp:second_sq}
          \xymatrix{
          \{ X \}_n \ar[r] \ar[d] & \{ X^{\wedge}_E\}_n  \ar[d] & \\
           \{\overline E_n\wedge X\}_n \ar[r] & \{ P^n(\overline E_n\wedge X) \}_n \\
          }
        \end{equation}
        and note that, in order to conclude, we only need to show that upon taking $E$-homology the map $E\wedge X\ra E\wedge X^{\wedge}_E$ is an equivalence. This map is the upper horizontal arrow of the diagram 
        \begin{equation}
          \label{eqn:E-loc_equals_E-nilp:second_sq_E}
          \xymatrix{
          \{ E\wedge X \}_n \ar[r] \ar[d] & \{E\wedge X^{\wedge}_E\}_n  \ar[d] & \\
           \{E\wedge (\overline E_n\wedge X)\}_n \ar[r] & \{ E\wedge P^n(\overline E_n\wedge X) \}_n \\
          }
        \end{equation}
        which is obtained from \eqref{eqn:E-loc_equals_E-nilp:second_sq} by smashing with $E$. In \eqref{eqn:E-loc_equals_E-nilp:second_sq_E} the right vertical map is a pro-weak equivalence as it follows from the previous part of the argument. The lower horizontal map is a pro-weak equivalence too by \ref{lemma:E_wedge_preserves_pro_equiv}. Finally the left vertical map of \eqref{eqn:E-loc_equals_E-nilp:second_sq_E} is a pro-weak equivalence as well, as we now explain. We have a tower of fibre sequences
          \begin{equation}
          \label{eqn:bo}
            \overline E^n \wedge X \ra X \ra \overline E_{n-1}\wedge X
          \end{equation}
        which we obtain, for varying $n$, from diagram \eqref{eqn:fund_comp_of_fib_seq_E^n_to_En} upon smashing with $X$. The left vertical map in the square \eqref{eqn:E-loc_equals_E-nilp:second_sq} is the map of pro-objects induced by the right hand side maps of \eqref{eqn:bo}. We claim that, after smashing \eqref{eqn:bo} with $E$, the tower $\{ \overline E^n \wedge X\}_n$ on the left hand side of \eqref{eqn:bo} becomes pro-weakly equivalent to zero. Indeed by the very inductive definition of $\overline E^n$ we have fibre sequences deduced from \eqref{eqn:fund_fib_seq_E^n+1_E^n_EE^n}
        \[ \overline E^{n+1}\wedge X \ra \overline E^n\wedge X \ra E\wedge \overline E^n\wedge X.\] 
        Here, the map on the right hand side is $e_E\wedge \rm{id}$ where $e_E: \bb S \ra E$ is the unit of the ring spectrum $E$; the map on the left hand side is $\bar e\wedge id$ (see \eqref{eqn:fund_fib_seq_E^n+1_E^n_EE^n}) and appears as the transition map in the tower $\{E\wedge \overline E^n \wedge X\}_n$.
        After smashing with $E$ we have an induced long exact sequences of homotopy modules
        \[ \cdots \ra \underline \pi_k(E\wedge \overline E^{n+1}\wedge X) \ra \underline \pi_k(E\wedge \overline E^n\wedge X) \ra \underline \pi_k(E\wedge E\wedge \overline E^n\wedge X) \ra \cdots\]
        and the map $\underline \pi_k(E\wedge \overline E^n\wedge X) \ra \underline \pi_k(E\wedge E\wedge \overline E^n\wedge X)$ is split by the multiplication of $E$. It follows that the previous map, which is the same as the transition map in the tower $\{E\wedge \overline E^n\wedge X\}_n$ is zero on every homotopy module and hence such a tower is pro-weakly equivalent to the zero tower. Using corollary \ref{cor:2-3prozero}, we deduce that the upper horizontal map of \eqref{eqn:E-loc_equals_E-nilp:second_sq_E} is a pro-weak-equivalence. Both source and target of this map are constant towers so the map is actually an isomorphism of spectra and this concludes the proof.
    \end{proof}

    \begin{rmk}
      \label{rmk:HKOcomparison}
      Let $p$ be the exponential characteristic of the base field $K$. The spectral sequence we construct in Section \ref{sec:the_E_based_MANSS}, in the special case when $E=H\bb Z/\ell$  and $\ell$ is a prime, is known as the \emph{(homological) motivic Adams Spectral sequence}. The convergence of such a spectral sequence was analysed in the work of Hu-Kriz-Ormsby \cite{MR2811716,MR2774158}. One of the main results of \cite{MR2811716} is that for a connective spectrum $X$ which has a cell presentation of finite type, in case $p=1$, the $H\bb Z/\ell$-nilpotent completion map is naturally identified with the adic completion map $\chi_{\eta,\ell}(X): X \ra X^{\wedge}_{\eta,\ell}$. We have not checked up to what extent one can weaken the assumption on the characteristic of the base field. By combining Theorem \ref{thm:red_to_moore_spt}, Theorem \ref{thm:convergence_mod_stuff} and Example \ref{ssub:ex_of_loc_HZ} we deduce a stronger version of such a convergence statement. Namely that the $H\bb Z/\ell$-nilpotent completion map is naturally identified with the adic completion map $\chi_{\eta,\ell}(X): X \ra X^{\wedge}_{\eta,\ell}$ for every connective spectrum $X$, independently on the characteristic of the base field.

      We wish to make a couple of observations. In first place we have no restriction on the characteristic of the base field: our approach works equally fine even when $\ell=\rm{char}(K)=p>1$. Unfortunately an explicit description of the dual motivic Steenrod algebra $\pi_{\ast,\ast}(H\bb Z/\ell \wedge H\bb Z/\ell)$ is still not available when $p=\ell$. It is worth mentioning that some progress in this direction has however been obtained in \cite{FrankSpitz:MDSAinPosChar}. Once this ingredient will be available, our convergence result will allow to start the computation of $\pi_{\ast,\ast}(\bb S^\wedge_{\eta,\ell})$ even when $\ell=p$.

      Another fact we would like to underline is that our result does not need that $X$ be cell of finite type. Many spectra do not satisfy this assumption: for instance a finite separable extension $L/K$ is not expected to give a cell object of $\cal{SH}(K)$.

    \end{rmk}

    \begin{lemma}
      \label{lem:Pn_inverted_Moore_tow_is_S_nilp_resol}
      Let $E$ be a homotopy commutative ring spectrum satisfying Assumption \ref{sub:assumption_A1} in the special case that $I=\emptyset$. Then for every connective spectrum $X$ the tower $\{P^n(\bb S[I^{-1}]\wedge X)\}_n$ is a $\underline\pi_0E$-nilpotent resolution of $X$.
    \end{lemma}

    \begin{proof}
      Since the unit $\bb S \ra E$ induces an isomorphism $\underline \pi_0 \bb S[I^{-1}] \overset{\simeq}{\ra} \underline \pi_0E $ we immediately conclude that the homotopy modules $\underline\pi_k(X\wedge \bb S[I^{-1}])$ are all $\underline\pi_0E$-modules and hence they are $\underline\pi_0E$-nilpotent. The mapping property of $\pi_0E$-nilpotent resolutions is an immediate consequence of the universal property of the Postnikov truncations.
    \end{proof}

    \begin{thm}
      \label{thm:convergence_inv_stuff}
      Let $E$ be an homotopy commutative ring spectrum satisfying assumption \ref{sub:assumption_A1} in the special case that $I=\emptyset$. Then for every connective spectrum $X$ the natural map $\lambda_E(X): X \ra X^{\wedge}_E$ is a $E$-equivalence. In particular the map $\beta_E(X): X_E\ra X^{\wedge}_E$ of \eqref{eqn:alphabetagamma} is an isomorphism in $\cal{SH}(K)$. 
    \end{thm}

    \begin{proof}
      The proof proceeds along the same lines as the proof of \ref{thm:convergence_mod_stuff}. More precisely we start by observing that both the towers $\{P^n(\overline E_n\wedge X)\}_n$ and $\{P^n(X\wedge \bb S[S^{-1}])\}_n$ are $\underline\pi_0E$-nilpotent resolutions of $X$ by \ref{prop:PnEn_tower_is_an_S-nilp_resol} and \ref{lem:Pn_inverted_Moore_tow_is_S_nilp_resol} respectively. We deduce, as in the proof of \ref{thm:convergence_mod_stuff}, that there is an isomorphism $\psi: X^\wedge _E \ra X\wedge \bb S[S^{-1}]$ making the following square of pro-spectra
      \begin{equation}
          \label{eqn:Inv-E-loc_equals_E-nilp:first_sq}
          \xymatrix{
          \{ X^{\wedge}_E\}_n \ar[r]^(0.45)u \ar[d] & \{ X\wedge \bb S[S^{-1}]\}_n \ar[d]\\
          \{ P^n(\overline E_n\wedge X) \}_n \ar[r]^(0.45){\simeq}_(0.45)\varphi & \{ P^n(X\wedge \bb S[S^{-1}]) \}_n\\
          }
        \end{equation}
      commutative. After smashing \eqref{eqn:Inv-E-loc_equals_E-nilp:first_sq} with $E$ the previous diagram, the lower horizontal map remains a pro-isomorphism. The vertical map on the right hand side of \eqref{eqn:Inv-E-loc_equals_E-nilp:first_sq} is a pro-weak-equivalence by \ref{subs:proj_to_post_is_pwe}, and stays a pro-weak-equivalence after smashing with $E$ by \ref{lemma:E_wedge_preserves_pro_equiv}. These two observations show that, after smashing with $E$, also the left vertical map of \eqref{eqn:Inv-E-loc_equals_E-nilp:first_sq} is a pro-weak-equivalence. The remaining part of the proof follows step by step the proof of Theorem \ref{thm:convergence_mod_stuff}.
    \end{proof}

\appendix
\section{Pro-spectra} 
  \label{sec:pro_spectra}
  \stepcounter{section}
  
  We recall in this section some facts on pro-spectra that we have used during the text. For our purposes we only need to deal with pro-objects in $\cal {SH}(K)$ which are indexed on the natural numbers; we do not need the full formalism of a model structure on pro-objects in $\spt_T^{\Sigma}(K)$. This section is independent of the rest of the paper and does not use any of the results proven in the previous part.

  \begin{defin}
  \label{defin:pro-spectrum}
    The category of motivic pro-spectra, which we denote by $\rm{Pro}(\cal {SH}(K))$, is the category of pro-object in $\cal {SH}(K)$. More precisely, an objects of $\rm{Pro}(\cal {SH}(K))$ is a diagram of spectra $X_\bullet$ in $\cal {SH}(K)$ indexed on a co-filtered category $I$. Given two motivic pro-spectra $X_\bullet$ and $Y_\bullet$ which are indexed on $I$ and $J$ respectively, we set
    \[\Hom_{\rm{Pro}}(X_\bullet,Y_\bullet):=\varprojlim_{j\in J} \varinjlim _{i \in I}[X_i,Y_j].\]
    The composition law is defined in the obvious way. In order to avoid confusion we will distinguish between a diagram $X_\bullet$ and the pro-object defined by $X_\bullet$, which we denote by $\{X_\bullet\}$. A map of diagrams $X_\bullet \ra Y_\bullet$ is called \emph{pro-isomorphism} if it induces an isomorphism of the associated pro-objects $\{X_\bullet\} \ra \{Y_\bullet\}$.
  \end{defin}
  
  We are interested in the subcategory of pro-towers. For the rest of this section we denote by $\bf N$ the category with objects $\{0,1,2,\dots\}$ and where $\Hom(n,m)=\ast$ if $n\geq m$ and $\Hom(n,m)=\emptyset$ if $n<m$. In a picture,
  \[\bf N= \{\cdots \ra n \ra \cdots \ra 2 \ra 1\ra 0\}.\]
  \begin{defin}
  \label{defin:pro-tower}
    A pro-tower in $\cal{SH}(K)$ is a motivic pro-spectrum which is indexed on the category $\bf N$. The category of pro-towers in $\cal{SH}(K)$, denoted by $\rm{Tow}(\cal{SH}(K))$, is the full subcategory of $\rm{Pro}(\cal{SH}(K))$ whose objects are pro-towers.
  \end{defin}

  \subsection{} 
    The category of pro-towers has a fairly explicit description as a suitable localization of the category of diagrams $\underline{\rm{Fun}}(\bf N, \cal{SH}(K))$. We recall such a description.

    A re-indexing function is a function $n_\bullet: \bb N \ra \bb N$ such that, $\forall k\in \bb N$, $n_k\geq \max\{n_{k-1},k\}$. The set of re-indexing functions is denoted by $\cal N$. Every re-indexing function $n_\bullet$ defines a re-indexing functor
     \[\rho_{n_\bullet}: \underline{\rm{Fun}}(\bf N, \cal{SH}(K) ) \ra \underline{\rm{Fun}}(\bf N, \cal {SH}(K)).\]
    On objects, $\rho_{n_\bullet}$ associates with every diagram $X_\bullet$ the diagram $X_{n_\bullet}$ where, for every $k$, the tower map $X_{n_k} \ra X_{n_{k-1}}$ is the composition of the tower maps $X_{n_k} \ra X_{n_k-1} \ra \cdots \ra X_{n_{k-1}}$. In \cite{MR1428551} $X_{n_\bullet}$ is called $n_\bullet$-spaced tower. On morphisms $\rho_{n_\bullet}$ is defined in the obvious way.
  
    Given two re-indexing functions $n_\bullet,m_\bullet$ one can define a composition $n\circ m:=k \mapsto n_{m_k}$. This operation is associative and unital, where the unit is the identity function $i$ of $\bb N$. The set $\cal N$ of re-indexing functions, with the operation $\circ$, is thus an associative monoid with unit given by the identity function. We can also endow $\cal N$ with a partial order by setting $n_\bullet \leq m_\bullet$ if for every $k \in \bb N$, $n_k \leq m_k$. Clearly for every pair $n_\bullet, m_\bullet$, $n\circ m \geq n_\bullet,m_\bullet$, and for every $n_\bullet \in \cal N$ we have that $n_\bullet\geq id_\bullet$. 

    Given two re-indexing functions $m_\bullet, n_\bullet$ such that $n_\bullet \leq m_\bullet$, we define a natural transformation of diagrams $\fk p_n^m:X_{m_\bullet} \ra X_{n_\bullet}$ by defining $X_{m_k} \ra X_{n_k}$ to be the composition of the tower maps $X_{m_k} \ra X_{m_k-1} \ra \cdots \ra X_{n_{k}}$. In \cite{MR1428551} the maps $\fk p_n^m$ are called \emph{basic self-tower maps} of $X$. It follows that the map of sets 
    \[\cal N \ra \rm{End( \underline{\rm{Fun}}(\bf N, \cal{SH}(K)))}\]
    upgrades to a functor when we see $\cal N$ as the category associated to the partially ordered set ($\cal N,\geq$). 

    The category of pro-towers in $\cal{SH}(K)$ is thus equivalent to the localization of the functor category $\underline{\rm{Fun}}(\bf N, \cal {SH}(K))$ at the basic self-tower maps. In particular given any re-indexing function $n_\bullet$ and any tower $X_\bullet$, the natural map $X_{n_\bullet} \ra X_\bullet$ represents the identity of the pro-tower $\{X_\bullet\}$. From these observations we deduce the following criterion.
    \begin{cor}
    \label{cor:crit_pro_iso}
      A map of pro-towers $f: \{X_\bullet \} \ra \{ Y_\bullet \}$ is an isomorphism in $\rm{Tow}(\cal{SH}(K))$ if and only if the following condition is satisfied: up to re-indexing $X_\bullet$, so that $f$ can be represented by a natural transformation $k\mapsto f_k: X_k \ra Y_k$, for every $k\in \bb N$ there is an integer $n_k >> k$ and a map $Y_{n_k} \ra X_k$ making both the triangles
      \[
      \xymatrix{
        X_{n_k} \ar[r]^{f_{n_k}} \ar[d] & Y_{n_k}\ar[d] \ar[dl]\\
        X_k \ar[r]_{f_k} & Y_k\\
      }
      \] to commute. In particular a tower $\{Z_\bullet\}$ is pro-isomorphic to $0$ if and only if for every $k \in \bb N$ there exists an integer $n_k>>k$ such that the tower map $X_{n_k} \ra X_k$ is zero.
    \end{cor}

  \begin{cor}
    \label{cor:Ewedge_pres_pro-isoms}
    Let $E$ be a motivic spectrum. Then the functor $E\wedge-$ preserves pro-isomorphisms of towers.
  \end{cor}

  \begin{defin}
  \label{defin:pro-we}
    A map of pro-towers $f: \{ X_\bullet\} \ra \{Y_\bullet\}$ is a \emph{pro-weak-equivalence} if, for every integer $p$, the induced map $\{\underline \pi_p (X_\bullet)\} \ra \{\underline \pi_p (Y_\bullet)\}$ is a pro-isomorphism of homotopy modules. A pro-tower $\{X_\bullet \}$ is \emph{pro-contractible} if, for every integer $p$, the pro-homotopy module $\{\pi_p(X_\bullet)\}$ is pro-isomorphic to $0$.
  \end{defin}
  
  \subsection{}
  \label{subs:proj_to_post_is_pwe}
  Of course every pro-isomorphism is a pro-weak-equivalence. Moreover, given any tower $X_\bullet$ in $\cal{SH}(K)$, the projection to the Postnikov tower 
  \[ X_k \ra P^k(X_k)\]
  induces a pro-weak-equivalence of towers $\{X_\bullet\} \overset{\pi_k}{\ra} \{ P^\bullet (X_\bullet)\}.$
  Note that in general the projection map $\{\pi_\bullet\}$ does not need to be a pro-isomorphism.

  \subsection{}
  Since $\Pi_\ast(K)$ is an abelian category, then $\rm{Pro}(\Pi_\ast(K))$ is an abelian category by Proposition 4.5 of \cite[Appendix]{MR0245577}. Moreover the full subcategory of pro-towers $\rm{Tow}(\Pi_\ast(K))$ is closed under finite limits and colimits: this can be proved following the lines of the proof of Proposition 2.7 in \cite{MR1428551}. In particular $\rm{Tow}(\Pi_\ast(K))$ is an abelian category. It follows that a map of pro-towers $\{f_\bullet\} :\cal \{M_\bullet\} \ra \cal \{N_\bullet\}$ is an isomorphism if and only if both $\ker(\{f_\bullet\})$ and $\coker(\{f_\bullet\})$ are pro-isomorphic to $0$. In particular we conclude the following.
  \begin{cor}
  \label{cor:2-3prozero}
    Let $X_\bullet, Y_\bullet, Z_\bullet$ be diagrams in $\cal{SH}(K)$ indexed on $\bf N$. Assume we have maps of diagrams $f_\bullet: X_\bullet \ra Y_\bullet$ and $g_\bullet : Y_\bullet \ra Z_\bullet$ such that, for every $n\in \bf N$, the digram 
    \[X_n \overset{f_n}{\ra} Y_n \overset{g_n}{\ra} Z_n\]
    is a fibre sequence. Then the map $g_\bullet$ is a pro-weak-equivalence if and only if the tower $\{X_\bullet\}$ is pro-contractible.
  \end{cor}

  \begin{lemma}
  \label{lemma:E_wedge_preserves_pro_equiv}
    Assume $\{ X_\bullet\}$ is a pro-tower of motivic spectra. Then for every connective spectrum $E$, the projection to the Postnikov tower induces a pro-weak-equivalence 
    \[\{E\wedge X_n\}_{n\in \bf N} \ra \{ E\wedge P^n (X_n)\}_{n \in \bf N}.\]
  \end{lemma}
  \begin{proof}
      Consider the fundamental fibre sequence 
     \[ E\wedge P_{n_k+1}(X_{n_k}) \ra E\wedge X_{n_k}\ra E\wedge P^{n_k}(X_{n_k}). \] 
      If $E$ is $c$-connective then $E\wedge P_{n_k+1}(X_{n_k})$ is $(c+n_k+1)$-connective. So once we have fixed an integer $p \in \bb Z$, for every $k \in \bb N$ we can chose an integer $n_k$ such that $c+n_k+1>p$ and so the projection to the Postnikov tower induces an isomorphism of homotopy modules
      \[ \underline \pi_p (E\wedge X_{n_k}) \ra \underline \pi_p(E\wedge P^{n_k}(X_{n_k})).\]
      By applying \ref{cor:crit_pro_iso} we conclude.
   \end{proof} 
      

\bibliography{locandcomp_paperversion}
\bibliographystyle{alpha}{}
\end{document}